\documentclass[a4paper,10pt]{article}
\usepackage{amsmath,amssymb,amsfonts,amsthm,mathrsfs}
\usepackage{hyperref}
\usepackage{xcolor}
\usepackage{tikz-cd}
\usetikzlibrary{decorations.pathmorphing} 
\usepackage[all]{xy}
\usepackage[]{algorithm2e}
\usepackage{enumitem}

\usepackage{import}
\graphicspath{{Figures/}}
\usepackage[margin=20pt]{caption} 

\usepackage[margin=1.5 in]{geometry}
\usepackage[style=alphabetic,sorting=nyt,maxbibnames=7,maxcitenames=5, backend=biber, style=alphabetic]{biblatex}
\newtheorem{thm}{Theorem}[section]
\newtheorem{mainthm}{Theorem}

\newtheorem{maincor}[mainthm]{Corollary}

\newtheorem{cor}[thm]{Corollary}
\newtheorem{lem}[thm]{Lemma}
\newtheorem{prop}[thm]{Proposition}

\newtheorem{claim}{Claim}
\newtheorem*{claim*}{Claim}

\theoremstyle{definition}
\newtheorem{dfn}[thm]{Definition}
\newtheorem{rem}[thm]{Remark}
\AtEndEnvironment{rem}{\null\hfill\ensuremath{\triangle}}
\newtheorem{setup}[thm]{Setup}

\newcommand{\st}{\mid} 
\renewcommand{\tilde}{\widetilde}

\newcommand{\RR}{\mathbb{R}}
\newcommand{\NN}{\mathbb{N}}
\newcommand{\ZZ}{\mathbb{Z}}
\newcommand{\Sph}{\mathbb{S}}
\newcommand{\DD}{\mathbb{D}}

\newcommand{\into}{\hookrightarrow}
\newcommand{\onto}{\twoheadrightarrow}

\newcommand{\C}{\mathcal{C}} 
\newcommand{\E}{\operatorname{E}} 
\newcommand{\K}{\operatorname{K}} 
\newcommand{\F}{\operatorname{F}} 

\newcommand{\TopS}{\Sigma_\mathrm{top}}

\newcommand{\Hom}{\operatorname{Hom}}
\newcommand{\Int}{\operatorname{Int}} 
\newcommand{\TopHom}{\operatorname{Hom}_\mathbf{TopGr}}

\newcommand{\HomSet}{\Hom_{\mathbf{Set}}}

\newcommand{\id}{\mathrm{id}}
\renewcommand{\epsilon}{\varepsilon}

\newcommand{\HomSS}{\Hom_\mathbf{SSet}}
\newcommand{\SD}{\operatorname{SD}} 

\DeclareMathOperator{\im}{im}

\DeclareMathOperator{\Map}{Map}
\DeclareMathOperator{\Ret}{Ret}

\DeclareMathOperator{\pr}{pr} 

\title{Geometric invariants of locally compact groups:\\the homotopical perspective}
\author{Kai-Uwe Bux$^1$ \and Elisa Hartmann$^1$ \and José Pedro Quintanilha$^2$}

\date{$^1$Universität Bielefeld\\
	$^2$Ruprecht-Karls-Universität Heidelberg\\\bigskip
	\today}

\begin{document}

\maketitle

\begin{abstract}
	We extend the classical theory of homotopical $\Sigma$-sets~$\Sigma^n$ developed by Bieri, Neumann, Renz and Strebel for abstract groups, to $\Sigma$-sets~$\TopS^n$ for locally compact Hausdorff groups. Given such a group~$G$, our $\TopS^n(G)$ are sets of continuous homomorphisms $G \to \RR$ (``characters''). They match the classical $\Sigma$-sets $\Sigma^n(G)$ if $G$~is discrete, and refine the homotopical compactness properties~$\mathrm C_n$ of Abels and Tiemeyer. Moreover, our theory recovers the definition of $\TopS^1$~and~$\TopS^2$ proposed by Kochloukova.
	
	Besides presenting various characterizations of $\TopS^n$ (particularly for $n\in \{1,2\}$), we show that characters in $\TopS^n(G)$ are also in $\TopS^n(H)$ if $H\le G$ is a closed cocompact subgroup, and we generalize several classical results. Namely, we prove that the set of nonzero elements of $\TopS^n(G)$ is open, we prove that characters in a group of type~$\mathrm C_n$ that do not vanish on the center always lie in~$\TopS^n(G)$, and we relate the $\Sigma$-sets of a group with those of its quotients by closed subgroups of type~$\mathrm C_n$. Lastly, we describe how $\TopS^n(G)$ governs whether a closed normal subgroup with abelian quotient is of type~$\mathrm C_n$, generalizing one of the highlights of the classical theory.
\end{abstract}

 \tableofcontents

\section{Introduction}

The properties of a group being finitely generated or finitely presented are well known to be, respectively, the cases $n=1$ and $n=2$ of the finiteness properties~$\mathrm F_n$. In a nutshell, every group~$G$ has a canonically associated homotopy type -- that of a classifying space~$\mathrm K(G,1)$ -- and $G$~is said to be of type~$\mathrm F_n$ when this homotopy type is realizable by a CW complex with finite $n$-skeleton.

It is often surprisingly challenging to determine when property~$\mathrm F_n$ descends to subgroups, even for $n=1$, and towards the end of the last century, this circle of questions led to the discovery of $\Sigma$-sets (also known as ``$\Sigma$-invariants'',  ``geometric invariants'', or ``BNSR-invariants''). The earliest contributions were due to Bieri, Strebel \cite{BS80, BS81,BS82}, and later also Neumann \cite{BNS87}, and associated to each finitely generated group~$G$ certain sets of characters (that is, group homomorphisms $G\to\RR$) in a few different yet related ways. A common feature of these character sets is that, under some circumstances, they allow one to understand when subgroups of~$G$ are finitely generated.
For a survey on the most widely studied variant, nowadays denoted by~$\Sigma^1(G)$, see the (vast, yet unfinished) memoir by Strebel \cite{Str13}.

Much like the property of~$G$ being finitely generated is but the first instance of the finiteness properties~$\mathrm F_n$, it turns out that $\Sigma^1(G)$ too is the first in a descending sequence of sets~$\Sigma^n(G)$, as explained by Bieri and Renz \cite{Ren88,BR88}. This parallel is more than an analogy: from a certain perspective, which we shall adopt in the text, each~$\Sigma^n(G)$ is a refinement of property~$\mathrm F_n$, accessible by peering into the $n$-th level of the homotopy theory of (classifying spaces for)~$G$. A suggestive slogan is that a character $\chi \colon G\to \RR$ lies in $\Sigma^n(G)$ if $G$~is ``of type~$\F_n$ along~$\chi$''.
In keeping with the theme of computing finiteness properties of subgroups,
one of the most important foundational results of $\Sigma$-theory is that for $G$~finitely generated, a normal subgroup $H\le G$ with abelian quotient is of type~$\mathrm F_n$ if and only if $\Sigma^n(G)$~contains all characters that vanish on~$H$ \cite[Satz~C]{Ren88}.

An entirely different direction in which to generalize the finiteness properties~$\mathrm F_n$ was pursued by Abels and Tiemeyer \cite{AT97}, who took inspiration from a classical criterion of Brown to define the ``compactness properties'' $\mathrm C_n$ for topological groups. More precisely, these are defined for locally compact Hausdorff groups~$G$, and recover the classical properties~$\mathrm F_n$ when $G$~has the discrete topology.

Our contribution with this article is to complete the bottom-right corner of the rectangle
$$\begin{tikzcd}
	\mathrm F_n \ar[squiggly, rrr, "\substack{\text{Bieri, Neumann,} \\ \text{Renz, Strebel}}"] \ar[squiggly, d, "\substack{\text{Abels,}\\ \text{Tiemeyer}}"'] &&& \Sigma^n \ar[d, dashed ]\\
	\mathrm C_n \ar[dashed, rrr]&&& \TopS^n
\end{tikzcd},$$
in other words, to lay out a theory of $\Sigma$-sets for locally compact Hausdorff groups~$G$, which we denote by $\TopS^n(G)$. These are sets of characters $\chi \colon G \to \RR$ (which are now \emph{continuous} homomorphisms), refining the compactness properties~$\mathrm C_n$ of Abels--Tiemeyer, and restricting to the classical~$\Sigma^n(G)$ when $G$~is discrete. We also extend a variety of results from the classical theory to this broader context.

It should be noted that the first steps in this direction have been taken by Kochlou\-kova, who proposed definitions of $\TopS^1(G)$~and~$\TopS^2(G)$ \cite{Ko04}. Her construction has the advantage of being very hands-on, relying on explicit generating sets and presentations, at the cost of leaving it unclear how to define $\TopS^n$ for $n\ge 3$ (in much the same way that the standard definitions of a group being finitely generated or finitely presented make it unclear how to define the higher properties~$\mathrm F_n$). Here, we provide a uniform definition of $\TopS^n$ for every~$n\in \NN$, and show that it recovers those of Kochloukova for $n\in \{1,2\}$.

There is a parallel version of this story, where instead of the finiteness properties $\mathrm{F}_n$, one starts with their homological counterparts, properties~$\mathrm{FP}_n$ (possibly with the extra input of a commutative ring~$R$, which by omission is typically taken to be~$\ZZ$).
Bieri and Renz have refined these properties to homological $\Sigma$-sets $\Sigma^n(G;R)$ \cite{BR88}, and the work of Abels--Tiemeyer provides topological analogues, properties $\mathrm{CP}_n$ \cite{AT97}. In a separate paper, we focus on this perspective, defining homological $\Sigma$-sets
$\TopS^n(G; R)$ for locally compact Hausdorff groups~$G$ \cite{BHQ}. Besides establishing several analogues of results in the present article, that paper re-frames both theories in terms of coarse geometry and relates them by showing that, as in the discrete case,
\begin{align*}
	\TopS^1(G) &= \TopS^1(G; \ZZ),\\
	\TopS^n(G) &= \TopS^n(G; \ZZ) \cap \TopS^2(G) \text{,\quad for $n\ge 2$}.
\end{align*}

\subsection{Main results}

Our proposed definition of when a character $\chi\colon G\to \RR$ lies in~$\TopS^n(G)$, Definition~\ref{dfn.topsigma}, considers the free simplicial set~$\E G$ on~$G$, filtered by the simplicial sub\-sets~$G_\chi \cdot \E C$, where $G_\chi := \chi^{-1} ([0,+\infty[)$, and $C$~varies over the poset~$\C(G)$ of compact subsets of~$G$ (see Section~\ref{sec:prelim} for the definition of free simplicial sets and relevant homotopy-theoretical notions). We declare that $\chi \in \TopS^n(G)$ if and only if this filtration is essentially $(n-1)$-connected. When $\chi = 0$, so $G_\chi=G$, this recovers precisely Abels--Tiemeyer's definition of property~$\mathrm C_n$; in other words, $G$~is of type $\mathrm C_n$ if and only if $0\in \TopS^n(G)$.

It is not nearly as obvious that our theory recovers the classical~$\Sigma^n(G)$ of Bieri, Neumann, Renz and Strebel for discrete groups. This is our first main result (Theorem~\ref{thm.sigmasmatch} in the text):	

\begin{mainthm}[The classical $\Sigma^n$] \label{thm:sigmasmatch_intro}
	If $G$~is a discrete group, then for every $n\in \NN$, we have
	\[\Sigma^n(G) = \TopS^n(G).\]
\end{mainthm}

On the way to proving Theorem~\ref{thm:sigmasmatch_intro}, we establish an analogue of Brown's Criterion that is useful for detecting that a character lies in~$\Sigma^n(G)$.
Recall that Brown's Criterion for property~$\mathrm F_n$ uses a filtration $(Y_\alpha)_{\alpha\in A}$ of an $(n-1)$-connected $G$-CW complex~$Y$ by cocompact subcomplexes \cite[Theorem~7.4.1 and exercise on p.~179]{Geo08}. To test whether a character $\chi \colon G\to \RR$ lies in~$\Sigma^n(G)$, we supplement~$Y$ with a ``valuation'' $v\colon Y \to \RR$ compatible with~$\chi$ (see Definition~\ref{dfn:valuation}), and filter~$Y$ by all subcomplexes~$Y_{\alpha,s}$, where $(\alpha,s)$ ranges over $A\times \RR$ and $Y_{\alpha,s}$ is obtained from~$Y_\alpha$ as the span of the vertices~$y$ with $v(y)\ge s$.

\begin{mainthm}[Brown's Criterion for characters]
	\label{thm:brownforchars_intro}
	Let $G$ be a discrete group, let $n\in \NN$, let~$X$ be an $(n-1)$-connected rigid $G$-CW complex, and $(X_\alpha)_{\alpha \in A}$ a filtration of~$X$ by $G$-subcomplexes with $G$-finite $n$-skeleta. Assume that for each $k$-cell~$e$, its stabilizer~$G_e$ is of type $\mathrm{F}_{n-k}$. Let $\chi\colon G \to \RR$ be a character and $v$~a valuation on~$X$ associated to~$\chi$.
	
	Then $\chi \in \Sigma^n(G)$ if and only if the filtration $(X_{\alpha, s})_{(\alpha, s)\in A\times \RR}$ is essentially $(n-1)$-connected.
\end{mainthm}
 
 This result is presented as Theorem~\ref{thm:brownforchars_general} in the text. It has closely related analogues in the work of Meinert, namely its counterpart for homological $\Sigma$-sets \cite[Theorem~3.8]{Mei95}, and a weaker version of the case $n=2$ where the $2$-skeleton $X^{(2)}$~is assumed to be $G$-finite, rather than filtered by $G$-finite subcomplexes \cite[Theorem~A]{Mei97}. In his doctoral thesis, Meinert also gives a version of the case $n=2$ in terms of a filtration, but with the assumption that the $G$-action is free \cite{Mei93}. To the best of our knowledge, a statement in the generality of Theorem~\ref{thm:brownforchars_intro} is absent from the literature, but given the abundance of similar results, we make no claim of originality.
 
We give alternative, explicit descriptions of $\TopS^1$ and $\TopS^2$. In particular, the following theorem shows they agree with the ones given by Kochloukova:

\begin{mainthm}[$\TopS^1$ and $\TopS^2$]\label{thm:Koch_intro}
	For every character $\chi\colon G\to \RR$, we have:
	\begin{enumerate}
		\item $\chi \in \TopS^1(G) \iff \text{$G$~is compactly generated along~$\chi$}$.
		\item $\chi \in \TopS^2(G) \iff \text{$G$~is compactly presented along~$\chi$}$.
	\end{enumerate}
\end{mainthm}
The notions of being compactly generated / presented ``along~$\chi$'' are explained in Sections \ref{sec:sigma1}~and~\ref{sec:sigma2} respectively, both algebraically (in terms of words and relations) and geometrically (using Cayley graphs and Cayley complexes). Theorem~\ref{thm:Koch_intro} is established in the text as Theorems~\ref{thm:sigma1_cptgen} and~\ref{thm:sigma2_cptpres}.

Regarding the structure of the set~$\TopS^n(G)$, it is fairly straightforward to show that it is nonempty precisely if $G$~is of type~$\mathrm C_n$, in which case it is a cone at~$0$ in the $\RR$-vector space $\TopHom(G,\RR)$ of continuous group homomorphisms $G \to \RR$ (see Proposition~\ref{prop.zerocharacter}). 
A notable feature of the classical theory is that the property of a nonzero character being in $\Sigma^n(G)$ is stable under small perturbation \cite[Satz~A]{Ren88}. This fact persists in the locally compact setting:

\begin{mainthm}[Openness]\label{thm:openness_intro}
	For every $n\in \NN$, if $G$~is of type~$\mathrm C_n$, then $\TopS^n(G)$ is the cone at~$0$ over an open subset of $\TopHom(G,\RR)\setminus\{0\}$.
\end{mainthm}

This result is restated in the text as Theorem~\ref{thm:openness}. Its proof relies on describing when a character lies in~$\TopS^n(G)$ in terms of certain compact subsets of~$G$, which we call ``connecting sets'' (Definition~\ref{def:connectingset}). We use them to establish various $\TopS^n$-criteria, especially under the assumption that $G$~is of type $\mathrm C_n$. The following theorem packages Propositions~\ref{prop:alonsoandsigma}, \ref{prop:theraisingcrit} and~\ref{prop:alonsocriterion} (the $\chi$-subscript in item~4 denotes the simplicial subset spanned by the vertices~$x$ with $\chi(x)\ge0$, and the $s$-subscript indicates we consider the span of the vertices with $\chi(x) \ge s$):

\begin{mainthm}[Criteria for~$\TopS^n$]\label{thm:sigmacriteria_intro}
	Let $n\in \NN$, let $\chi \colon G\to \RR$ be a nonzero character, let $K$~be an $n$-connecting set for the zero character, and $K_0$~an $(n-1)$-connecting set, also for the zero character. The following are equivalent:
	\begin{enumerate}
		\item $\chi \in \TopS^n(G)$,
		\item There is an $n$-connecting set for~$\chi$,
		\item There is a $G$-equivariant $\varphi \in \Map(G\cdot \E K^{(n)}, G\cdot \E K)$
		that raises $\chi$-values,
		\item There are $D\in \C(G)_{\supseteq K_0}$ and $s\in \RR_{\le 0}$ such that the inclusion
		$$(G\cdot \E K_0)^{(n-1)}_\chi \into (G\cdot \E D)_s$$
		is $\pi_k$-trivial for all $k\le n-1$.
	\end{enumerate}
\end{mainthm}

We also present a series of results on the behaviour of $\TopS^n(G)$ with respect to subgroups and quotients of~$G$. The compactness properties~$\mathrm C_n$ of Abels and Tiemeyer descend to closed cocompact subgroups \cite[Theorem~3.2.2]{AT97}, and we establish an analogous statement, with essentially the same proof (Theorem~\ref{thm:cocptsubgps} in the text):

\begin{mainthm}[Closed cocompact subgroups]\label{thm:cocptsubgps_intro}
	Let $H\le G$ be a closed cocompact subgroup, let $n\in \NN$, and let $\chi\colon G \to \RR$ be a character. Then,
	$$\chi\in\TopS^n(G) \iff \chi|_H\in\TopS^n(H).$$
\end{mainthm}
One can use this result to deduce that if $G$~is compactly generated and abelian, then every character lies in~$\TopS^n(G)$ (Corollary~\ref{cor:abelian}). More generally, we show (see Theorem~\ref{thm:center}):

\begin{mainthm}[Non-vanishing on the center]\label{thm:center_intro}
	Let $n \in \NN$ and suppose $G$~is of type~$\mathrm C_n$. Then $\TopS^n(G)$~contains all characters that do not vanish on the center~$\operatorname{Z}(G)$.
\end{mainthm}

The following results relate the $\Sigma$-sets of groups in a short exact sequence of locally compact Hausdorff groups
$$1 \to N \to G \xrightarrow{p}  Q \to 1.$$

\begin{mainthm}[$\TopS^n$ and quotients]\label{thm:sigma_quotient_intro}
	Fix a character $\chi \colon Q \to \RR$ and let $n\in \NN$.
	\begin{enumerate}
		\item If $p$~has a section and $\chi\circ p \in \TopS^n(G)$, then $\chi \in \TopS^n(Q)$.
		\item If $N$~is of type~$\mathrm C_n$ and $\chi \in \TopS^n(Q)$, then $\chi \circ  p \in \TopS^n(G)$.
		\item If $N$~is of type~$\mathrm C_n$ and  $\chi \circ p \in \TopS^{n+1}(G)$, then $\chi \in \TopS^{n+1}(Q)$.
	\end{enumerate}
\end{mainthm}

%

\begin{mainthm}[Co-abelian subgroups]\label{thm:coabelian_intro}
	Suppose $Q$~is abelian and fix $n\in\NN$. If for every character $\chi \colon Q\to \RR$ we have $\chi\circ p \in \TopS^n(G)$, then $N$~is of type~$\mathrm{C}_n$.
\end{mainthm}

In the text these results are presented as Proposition~\ref{prop:split}, Theorem~\ref{thm:sigmaofquotient}, and Theorem~\ref{thm:coabelian}. By far, the deepest and most technically demanding among them is Theorem~\ref{thm:coabelian_intro}. It generalizes a fundamental result of the classical theory \cite[Satz~C]{Ren88}, embodying the principle of using $\Sigma$-sets to transfer finiteness (or, in our case, compactness) properties to subgroups.

To the best of our knowledge, the ``$\chi=0$'' case of Theorem~\ref{thm:sigma_quotient_intro}, which pertains only to the classical property~$\mathrm C_n$, is absent from the literature. For the benefit of readers uninterested in $\Sigma$-theory, we spell it out explicitly:

\begin{maincor}[Compactness properties and quotients]
	Fix a short exact sequence $1 \to N \to G \xrightarrow{p}  Q \to 1$ of locally compact Hausdorff groups, and let $n\in \NN$.
	\begin{enumerate}
		\item If $p$ has a section and $G$~is of type~$\mathrm C_n$, then $Q$~is of type~$\mathrm C_n$.
		\item If $N$~and~$Q$ are of type $\mathrm C_n$, then so is~$G$.
		\item If $N$~is of type~$\mathrm C_n$ and $G$~is of type~$\mathrm C_{n+1}$, then $Q$~is of type $\mathrm C_{n+1}$.
	\end{enumerate}	
\end{maincor}

\subsection{Structure of the article}

In Section~\ref{sec:prelim} we import a few techniques from the homotopy theory of simplicial sets. Some basic familiarity with simplicial sets will be assumed, but we try to keep this overhead to a minimum, and provide references for introductory material.

Section~\ref{sec:basics} introduces and states basic properties of the sets~$\TopS^n$. The definition relies on a filtration of a certain simplicial set, and we provide characterizations in terms of a few alternative filtrations, which will be useful at various points in the article.

In Section~\ref{sec:classicalsigma}, we restrict our attention to discrete groups~$G$, with the aim of comparing our $\TopS^n(G)$ with Renz's $\Sigma^n(G)$. We begin by recalling the basic definitions given by Renz. Then, we establish an analogue of the free version of Brown's Criterion for characters (Proposition~\ref{prop:brownforchars_free}) and use it to prove that our theory recovers the classical theory of $\Sigma$-sets (Theorem~\ref{thm:sigmasmatch_intro}). At the end of the section, we prove the stronger version of Brown's Criterion for characters stated in Theorem~\ref{thm:brownforchars_intro}.

Sections \ref{sec:sigma1} and \ref{sec:sigma2} introduce, respectively, the notions of a generating set and of a presentation of a group~$G$ ``along a character $\chi \colon G\to \RR$'', and use them to give more hands-on descriptions of $\TopS^1(G)$~and~$\TopS^2(G)$. In particular, we recover the definitions given by Kochloukova, establishing Theorem~\ref{thm:Koch_intro}.

Section~\ref{sec:properaction} describes $\TopS^n(G)$ in terms of proper actions of~$G$ on a space (Proposition~\ref{prop:properaction}), and we use this characterization to prove the theorem on restrictions of characters to closed cocompact subgroups (Theorem~\ref{thm:cocptsubgps_intro}).

Our first result relating $\Sigma$-sets of topological groups in a short exact sequence, Theorem~\ref{thm:sigma_quotient_intro}, appears in Section~\ref{sec:quotients}. The statement concerning split short exact sequences is established first, with a fairly straightforward argument. Then, by proving a lifting property satisfied by group quotients with kernels of type~$\mathrm C_n$ (Lemma~\ref{lem:Cnlift}), we deduce the statements about the behaviour of~$\TopS^n$ with respect to such quotients.

As preparation for the most demanding proofs in the paper, Section~\ref{sec:raisingprinciple} presents a technical tool for showing that a character lies in~$\TopS^n(G)$, which we use to prove the result on characters that do not vanish on the center (Theorem~\ref{thm:center_intro}). We also introduce topologies on Hom-sets between relevant simplicial sets, which will be essential in turning various finiteness arguments from the classical theory of $\Sigma$-sets into compactness arguments.

In Section~\ref{sec:stability}, we prove our openness result (Theorem~\ref{thm:openness_intro}). The discussion is largely centered on the aforementioned connecting sets, and this is, in particular, where we establish the various $\TopS^n$-criteria given in Theorem~\ref{thm:sigmacriteria_intro}. A large portion of this section is devoted to proving Proposition~\ref{prop:alonso_main}, a statement that lies at the core of all our results about connecting sets.
		
Finally, Section~\ref{sec:coabelian} proves the theorem on using~$\TopS^n$ to transfer compactness properties to co-abelian subgroups (Theorem~\ref{thm:coabelian_intro}). The argument has various structural similarities to that of Theorem~\ref{thm:openness_intro}. In fact, the main auxiliary statement, Proposition~\ref{prop:radialalonso}, is about the existence of certain analogues of the connecting sets from Section~\ref{sec:stability}, adjusted to the current setup.

\subsection*{Acknowledgements}

This paper stemmed from discussions between the first author and Thomas Schick. We thank Robert Bieri, Ilaria Castellano, Dorian Chanfi, Yuri Santos Rego, Stefan Witzel, Xiaolei Wu and Matt Zaremsky for various discussions. We also thank Lukas Waas for helpful clarifications on subdivisions of simplicial sets.
Finally, we thank an anonymous reviewer for a very detailed report, which led to several improvements throughout the text.
This research was funded in part by the Deutsche Forschungsgemeinschaft (DFG, German Research Foundation) grant BU 1224/4-1
within SPP 2026 Geometry at Infinity.
The authors were also members of TRR 358, Project-ID 491392403 of the DFG.

\section{Preliminaries -- simplicial sets}\label{sec:prelim}

As in Abels and Tiemeyer's definition of the compactness properties~$\mathrm{C}_n$, simplicial sets will play a central role in constructing our homotopical $\Sigma$-sets.
Simplicial sets can be thought of as generalizing simplicial complexes in much the same way that graphs with edge orientations, parallel edges, or loops generalize simplicial graphs. For an intuitive introduction, see Friedman's survey \cite{Fri12}, and for a more extensive treatment see for example Curtis's article \cite{Cur71}.

	\subsection{Free simplicial sets}

Throughout the article, we will be primarily interested in simplicial sets of the following type, and subsets thereof.

\begin{dfn}\label{dfn:free_simplicial_set}
	The \textbf{free simplicial set} on a set~$S$ is the simplicial set $\E S$ with $n$-simplices given by $(\E S)_n := S^{n+1}$. The $i$-th face and degeneracy maps are given on an $n$-simplex $\sigma = (x_0, \ldots, x_n)$, respectively, by deleting or doubling the $i$-th entry:
	\begin{align*}
		d_i(\sigma) &= (x_0, \dots, x_{i-1}, x_{i+1}, \dots, x_n),\\
		s_i(\sigma) &= (x_0,\dots, x_{i-1}, x_{i}, x_{i}, x_{i+1}, \dots, x_n).
	\end{align*}
\end{dfn}

Readers familiar with group (co)homology should notice that if $G$~is a group, then 
$\E G$ is the simplicial set whose simplicial chain complex is the bar resolution of~$G$.

This construction extends to a functor $\mathbf{Set} \to \mathbf{SSet}$, which is right-adjoint to the $0$-skeleton functor. In other words, given a set~$S$ and a simplicial set~$X$, there is a natural bijection
\[ \Hom_{\mathbf{Set}}(X^{(0)}, S) \cong \Hom_{\mathbf{SSet}}(X, \E S).\]

	\subsection{Semisimplicial approximation}\label{sec:semisimplicial}

Simplicial sets come with a tool analogous to the simplicial approximation theorem for simplicial complexes. The barycentric subdivision operator for simplicial complexes is replaced by a certain ``double barycentric subdivision'' functor~$\SD \colon \mathbf{SSet} \to \mathbf{SSet}$
\cite[Definition~12.5]{Cur71}. The main property of~$\SD$ we will need is that for every simplicial set~$X$ there exists a natural map $\Phi \colon \SD(X) \to X$ whose realization~$|\Phi|$ is homotopic to a homeomorphism $|\SD(X)| \cong |X|$. Naturality implies in particular that $\Phi$~induces an inclusion of $0$-skeleta $X^{(0)} \into \SD(X)^{(0)}$, and by inspecting the construction, one can see that the homotopy from $|\Phi|$ to the aforementioned homeomorphism $|\SD(X)|\cong |X|$ is relative~$|X^{(0)}|$.
Regarding a simplex $\sigma$ of~$X$ as a simplicial subset of~$X$ (explicitly, as the minimal simplicial subset containing~$\sigma$), it will also be useful to take note of the fact that every vertex of $\SD(\sigma)$ lies within edge-distance at most~$2$ of some vertex of~$\sigma$.
For details see, for example, the accounts of Curtis \cite[Theorem~12.7]{Cur71} or Jardine \cite[Corollary~4.8]{Jar04}.

Recall that simplicial set~$X$ is \textbf{finite} if it has only finitely many nondegenerate simplices, and given a simplicial subset $A\subseteq X$, we say $X$~is \textbf{finite relative to~$A$} if $X$~has only finitely many nondegenerate simplices that are not in~$A$.

\begin{thm}[Semisimplicial approximation]\label{thm:approx}
	Let $A\subseteq X$ be simplicial sets with $X$~finite relative to~$A$, and let~$Y$ be another simplicial set. Given a simplicial map $f\colon A \to Y$ and a continuous map $p \colon |X| \to |Y|$ extending~$|f|$, there is $m\in \NN$ and a simplicial map $f' \colon \SD^m (X) \to Y$ restricting to $f\circ \Phi^m$ on $\SD^m(A)$, and whose realization~$|f'|$ is homotopic to~$p\circ |\Phi^m|$ relative $|\SD^m(A)|$.
	\[\begin{tikzcd}
		\vert \SD^m(A)\vert \ar[r,"\vert\Phi^m\vert"]\ar[d,hook]& \vert A\vert \ar[rd,"|f|"] \ar[d,hook]\\
		\vert\SD^m(X)\vert \ar[r,"\vert\Phi^m\vert"] \ar[rr,bend right, "\vert f'\vert"]& \vert X\vert \ar[r,"p"]&\vert Y\vert
	\end{tikzcd}\] 
\end{thm}
	
Given simplicial sets~$X,Y$, the following set collects all simplicial maps from some iterated subdivision of~$X$ to~$Y$:
\[\Map(X,Y) := \bigsqcup_{m\in \NN} \HomSS(\SD^m (X),Y).\]
We will slightly abuse this terminology when given maps $f\colon \SD^n(X) \to Y$ and $g \colon \SD^m (Y) \to Z$, by talking about ``the composition $g \circ f$'' when in rigor we mean $g\circ \SD^m(f) \in \Map(X,Z)$.

\subsection{Homotopy groups}

Most accounts of the theory of simplicial sets define homotopy groups only for Kan complexes. The free simplicial set on a set is always a Kan complex, 
but various simplicial subsets that interest us are not. We will thus forsake the traditional, purely combinatorial, definition of homotopy groups in favor of a more hands-on one:

\begin{dfn}
	Given $k\in \NN$, the $k$-th homotopy group (pointed set, if $k=0$) of a simplicial set~$X$ with basepoint $x_0 \in X^{(0)}$ is the $k$-th homotopy group of its geometric realization:
	\[\pi_k(X,x_0) := \pi_k(|X|,x_0).\]
\end{dfn}

For a Kan complex~$X$, this definition agrees with the classical one. More precisely, the canonical map $X \to \operatorname{Sing}(|X|)$ to the singular set of its topological realization is a weak equivalence \cite[Proposition~I.11.1]{GJ99}.

Despite our topological notion of homotopy groups, we can use semisimplicial approximation to return to the combinatorial world of simplicial sets. To see this, let us fix once and for all a homeomorphism $\Sph^k \cong |\partial\Delta^{k+1}|$ that identifies the preferred basepoint of~$\Sph^k$ with a vertex $v$ of $\partial \Delta^{k+1}$. Then every element of $\pi_k(X,x_0)$ can be represented by an element of~$ \Map(\partial \Delta^{k+1}, X)$, in the following sense:

\begin{lem}[Combinatorial $\pi_k$]\label{lem:combspheres}
	For every pointed map $\eta\colon |\partial \Delta^{k+1}| \to |X|$, there are $m\in \NN$ and $\eta' \colon \SD^m(\partial \Delta^{k+1}) \to X$ such that $|\eta'|$~is pointed-homotopic to $\eta\circ |\Phi^m|$.
\end{lem}
\begin{proof}
	Apply semisimplicial approximation (Theorem~\ref{thm:approx}) with $f$~the constant simplicial map $\{v\} \to \{x_0\}$ and $p=\eta$. The resulting~$m$ and $\eta':=f' \colon \SD^m(\partial \Delta^{k+1}) \to X$ are then as required.
\end{proof}

The map $\eta'$ deserves to be called a representative of~$\eta$ because, due to the properties discussed in Section~\ref{sec:semisimplicial},  $|\Phi^n|$ is pointed-homotopic to a homeomorphism, and therefore a pointed-homotopy equivalence. A homotopy-inverse~$g\colon |\partial \Delta^{k+1}| \to |\SD^m(\partial \Delta^{k+1})|$ recovers the pointed homotopy class of~$\eta$ as $\eta \simeq |\eta'|\circ g$.
	
We will for the most part not be interested in the homotopy groups themselves, but rather on whether a simplicial map $f\colon X \to Y$ is trivial on~$\pi_k$ (on all connected components of~$X$).
The following characterization will be used on multiple occasions:

\begin{lem}[Combinatorial $\pi_k$-triviality]\label{lem:combfilling}
	Given $k\in \NN$ and $m_0\in \NN$, a simplicial map $\eta\colon \SD^{m_0}(\partial \Delta^{k+1}) \to X$
	represents the trivial element of~$\pi_k(X)$ if and only if there are $m\in \NN$ and
	a map $\mu\colon \SD^{m_0+m}(\Delta^{k+1}) \to X$ extending $\eta \circ \Phi^m$.
\end{lem}
\begin{proof}
	($\Leftarrow$) We are given the following diagram:
	$$\begin{tikzcd}
		\SD^{m_0+m}(\partial \Delta^{k+1}) \ar[r,"\Phi^m"] \ar[d,hook]& \SD^{m_0}(\partial \Delta^{k+1}) \ar[r,"\eta"] & X\\
		\SD^{m_0+m}(\Delta^{k+1}) \ar[rru, "\mu"]
	\end{tikzcd}$$
	After applying to it the realization functor, we use the fact that $|\Phi^m|$~is a homotopy equivalence, and that $|\Delta^{k+1}|$ is contractible to deduce that $|\eta|$ is nullhomotopic.
	
	($\Rightarrow$) Suppose $\eta$~is trivial, which means that its realization~$|\eta|$ extends to a continuous map $\mu_0 \colon |\SD^{m_0}(\Delta^{k+1})| \to |X|$. Applying semi\-simplicial approximation (Theorem~\ref{thm:approx}) with $f=\eta$ and $p=\mu_0$, we find~$m\in \NN$ and $\mu:= f'\colon \SD^{m_0 + m}(\Delta^{k+1}) \to X$ extending $\eta \circ \Phi^m$.
\end{proof}

In this situation, we will often simplify the language and notation, saying, for example, that ``$\mu$ fills~$\eta$'', where in rigor we should say ``$\mu$~fills $\eta \circ \Phi^m$''.

\subsection{Simplicial homotopies}
	
The \textbf{standard $1$-simplex}~$\Delta^1$ is the simplicial set whose $k$-simplices are the $(k+1)$-tuples of the form $(0,\dots, 0,1, \dots, 1)$, and whose $i$-th face (resp. degeneracy) map is obtained by deleting (resp. doubling) the $i$-th entry. For each $i\in \{0,\dots,k+1\}$, it will be convenient to denote by~$\tau_i^k$ the $k$-simplex of~$\Delta^1$ with $i$~zeros.
	
\begin{dfn}
	A \textbf{simplicial homotopy} between simplicial maps $f, g \colon  X \to Y$ is a simplicial map $H\colon X \times \Delta^1 \to Y$ such that for every $k$-simplex $\sigma$ of~$X$, we have $H(\sigma, \tau_{k+1}^{k}) = f(\sigma)$ and $H(\sigma, \tau_{0}^{k}) = g(\sigma)$.
\end{dfn}

Of particular importance for us is the case where $X$~and~$Y$ are simplicial subsets of a free simplicial set. Then, a simplicial homotopy from~$f$ to~$g$ has the form
\[(\sigma, \tau_i^k) \mapsto (f(x_0), \dots, f(x_{i-1}), g(x_{i}), \dots, g(x_k)).\]
This assignment defines a homotopy (with codomain~$Y$) precisely if every tuple as on the right hand side is a $k$-simplex in~$Y$. Note that since the $0$-skeleton of~$X\times \Delta^1$ is
\[(X\times \Delta^1)^{(0)} = X^{(0)} \times \{0,1\},\]
a simplicial homotopy from~$f$ to~$g$ whose codomain is contained in a free simplicial set is determined by $f$~and~$g$ (if it exists).

From a topological perspective, simplicial homotopies are very restrictive: there are certainly simplicial maps~$f,g$ that are not connected by a simplicial homotopy, though their realizations~$|f|,|g|$ are homotopic. A more comprehensive definition would be as an element of $H\in \Map(X\times \Delta^1 , Y)$ restricting to $f\circ \Phi^m$ and $g\circ \Phi^m$ at the endpoints of~$\Delta^1$ (for the appropriate~$m$). We shall however not make use of this notion.

\subsection{Filtrations and essential connectedness}\label{sec:filtrations}

Given a simplicial set~$X$, a \textbf{filtration} of~$X$ over a directed poset~$A$ is a family of simplicial subsets $X_\alpha\subseteq X$ indexed by~$A$, such that $\bigcup_{\alpha \in A}X_\alpha =X$ and whenever $\alpha \le \beta$, we have $X_\alpha \subseteq X_{\beta}$. In this paper, the poset~$A$ will most often be the poset of compact subspaces of a topological space.

If $X$~comes with a chosen basepoint~$x\in X^{(0)}$ that is contained in all elements of a filtration $(X_\alpha)_{\alpha \in A}$, we associate, for each $k\in \NN$, the directed system of homotopy groups $(\pi_k(X_\alpha, x_0))_{\alpha \in A}$, which comes with inclusion-induced maps $\pi_k(X_\alpha, x) \to \pi_k(X_\beta, x)$ whenever $\alpha\le\beta$. 

A directed system of pointed sets $(\pi_\alpha)_{\alpha \in A}$ is \textbf{essentially trivial} if for each $\alpha\in A$ there is $\beta \ge \alpha$ such that the map $\pi_\alpha \to \pi_\beta$ is trivial. For us, the relevant pointed sets will be homotopy groups.

\begin{rem}[Uniform vanishing]
	We warn the reader that essential triviality is stronger than the condition $\operatorname{colim}_{\alpha \in A} \pi_\alpha=\{1\}$. More concretely, when proving that $(\pi_\alpha)_{\alpha \in A}$ is essentially trivial, one should not start with $\alpha \in A$ and $\eta \in \pi_\alpha$, and then seek $\beta \ge \alpha$ such that $\eta$ gets mapped to $1 \in \pi_\beta$. That would be incorrect because $\beta$ is not allowed to depend on~$\eta$. Instead, one should fix~$\alpha$, and then find $\beta \ge \alpha$ such that \emph{all} $\eta \in \pi_\alpha$ vanish under $\pi_\alpha \to \pi_\beta$.
\end{rem}

Suppose we are given a filtration $(X_\alpha)_{\alpha \in A}$ of a nonempty simplicial set~$X$. Then, given~$x \in X$, we may consider the subposet $A_x := \{\alpha \in A \mid x \in X_\alpha \}$, and the directed system of homotopy groups $(\pi_k(X_\alpha,x))_{\alpha \in A_x}$, with the inclusion-induced maps.

\begin{lem}[Independence of basepoints]\label{lem:independenceofbasepoints}
	Suppose $X$~is connected, let $x, y\in X$, and fix $k\in \NN$. Then the directed system $(\pi_k(X_\alpha,x))_{\alpha \in A_x}$ is essentially trivial if and only if $(\pi_k(X_\alpha,y))_{\alpha \in A_y}$ is.
\end{lem}
\begin{proof}
	Since $X$~is connected, choose an edge-path~$\gamma$ from~$x$ to~$y$ in the $1$-skeleton $X^{(1)}$ and consider the subposet
	$$A_{\gamma} = \{\alpha \in A \mid \gamma \subseteq X_\alpha\}.$$
	The fact that $A_\gamma$~is cofinal in~$A_x$ implies that $(\pi_k(X_\alpha, x))_{\alpha \in A_x}$ is essentially trivial if and only if $(\pi_k(X_\alpha, x))_{\alpha \in A_\gamma}$ is. With the analogous observation for~$A_y$, we see it suffices we compare the directed systems
	$$(\pi_k(X_\alpha,x))_{\alpha \in A_{\gamma}}, \qquad (\pi_k(X_\alpha,y))_{\alpha \in A_{\gamma}}.$$
	But for each $\alpha \in A_\gamma$, the path~$\gamma$ induces an isomorphism $\pi_k(X_\alpha,x) \cong \pi_k(X_\alpha,y)$, which commutes with all inclusions.
\end{proof}

\begin{rem}[Ind-objects]
	Readers familiar with the ind-completion of a category might have noticed that the above proof actually shows the stronger fact that the given directed systems are isomorphic as ind-groups (or ind-pointed sets, when $k=0$). Most of the present subsection may, in fact, be more succinctly expressed in the language of ind-objects, but we have opted for a more hands-on exposition.
\end{rem}

Given $n\in \NN$, a filtration $(X_\alpha)_{\alpha\in A}$ of a pointed simplicial set~$(X,x)$ is \textbf{essentially $n$-connected} if for every $k\le n$, the directed system $(\pi_k(X,x))_{\alpha \in A_x}$ is essentially trivial. If the total space~$X$ is connected, as will be the case for all filtrations we consider in this article, Lemma~\ref{lem:independenceofbasepoints} tells us this notion is independent of~$x$, so we will routinely suppress base-points from our notation. In Section~\ref{sec:classicalsigma}, we will consider filtrations of CW complexes, rather than simplicial sets. The definition of essential $n$-connectedness in that context is the obvious adaptation of the one just presented.

To transfer the property of being essentially $n$-connected between different filtrations, we introduce the following notions:

\begin{dfn}\label{dfn:hoequiv}
	Given filtrations $(X_\alpha)_{\alpha\in A}$, $(Y_\beta)_{\beta \in B}$ of simplicial sets~$X, Y$ respectively and a map of posets $\phi \colon A \to B$, we say that a family of simplicial maps $(f_\alpha \colon X_\alpha \to Y_{\phi(\alpha)})_{\alpha \in A}$ \textbf{eventually commutes up to simplicial homotopy} if for every $\alpha, \alpha' \in A$ with $\alpha \le \alpha'$, there is $\beta\ge \phi(\alpha')$ such that the following diagram commutes up to simplicial homotopy:
	$$\begin{tikzcd}
		X_\alpha  \ar[d,hook]\ar[r,"f_\alpha"] & Y_{\phi(\alpha)} \ar[r,hook] & Y_\beta \\ X_{\alpha'}  \ar[r,"f_{\alpha'}"] & Y_{\phi(\alpha')} \ar[ur, hook]
	\end{tikzcd}$$
	The word ``eventually'' may be dropped if one can take $\beta = \phi(\alpha')$.
	
	This data $(\phi, (f_\alpha)_{\alpha\in A})$ is a  \textbf{simplicial homotopy equivalence} between the filtrations if there exist $\psi\colon B\to A$ and maps $(g_\beta\colon Y_\beta \to X_{\psi(\beta)})_{\beta \in B}$ eventually commuting up to simplicial homotopy, such that for every $\alpha \in A$ and $\beta \in B$, 
	there are $\alpha'\ge \psi(\phi(\alpha)), \beta'\ge \phi(\psi(\beta))$ making the following diagrams commute up to simplicial homotopy:	
	$$\begin{tikzcd}
	X_\alpha \ar[r,hook]\ar[d,"f_\alpha"] & X_{\alpha'}  \\
	Y_{\phi(\alpha)}\ar[r, "g_{\phi(\alpha)}"] & X_{\psi(\phi(\alpha))} \ar[u, hook]
	\end{tikzcd} \qquad
	\begin{tikzcd}
	Y_\beta \ar[r,hook] \ar[d,"g_\beta"] &Y_{\beta'} \\
	 X_{\psi(\beta)}\ar[r,"f_{\psi(\beta)}"]& Y_{\phi(\psi(\beta))} \ar[u, hook]
	\end{tikzcd}$$
	Then $(\psi, (g_\beta)_{\beta \in B})$ is also a simplicial homotopy equivalence, which we call a \textbf{simplicial homotopy inverse} to~$(\phi, (f_\alpha)_{\alpha \in A})$.
\end{dfn}

As before, since these definitions require all maps~$f_\alpha, g_\beta$ to be simplicial and the diagrams to commute up to simplicial homotopy, they are too restrictive to fully capture the analogous topological phenomena.
They are however simpler to state and will suffice for our purposes, as we only construct homotopy equivalences of filtrations at two points in the article (Lemma~\ref{lem:trimmedfiltrations} and Proposition~\ref{prop:properaction}), where the setting happens to be tame enough that simplicial ones suffice.

\begin{lem}[Homotopy equivalences and $\pi_k$ for filtrations]\label{lem:homotopyequiv}
	If two filtrations $(X_\alpha)_{\alpha\in A}$, $(Y_\beta)_{\beta \in B}$  are simplicially homotopy-equivalent, then for each~$k\in \NN$, the system $(\pi_k(X_\alpha))_{\alpha\in A}$ is essentially trivial if and only if $(\pi_k(Y_\beta))_{\beta\in B}$~is.
\end{lem}
\begin{proof}
	By symmetry, it suffices to prove one of the directions; we assume that $(\pi_k(Y_\beta))_{\beta\in B}$ is essentially trivial and recover the notation from Definition~\ref{dfn:hoequiv}. The proof consists of the  diagram
	$$\begin{tikzcd}
		X_\alpha\ar[r,hook] \ar[d,"f_{\alpha}"]& X_{\alpha'} \ar[r, hook]& X_{\alpha'''}\\
		Y_{\phi(\alpha)} \ar[r, "g_{\phi(\alpha)}"] \ar[d, hook]& X_{\psi(\phi(\alpha))}  \ar[u, hook] \ar[r, hook]& X_{\alpha''} \ar[u, hook]\\
		Y_\beta \ar[r, "g_\beta"]& X_{\psi(\beta)} \ar[ur, hook]
	\end{tikzcd},$$
	which is constructed as follows:
	\begin{enumerate}
		\item $\beta\ge \phi(\alpha)$ is chosen such that the inclusion $Y_{\phi(\alpha)} \into Y_\beta$ is $\pi_k$-trivial,
		\item $\alpha'$ and the top-left square are  as in the definition of a simplicial homotopy inverse,
		\item $\alpha''$ and the bottom pentagon are obtained from the fact that the $(g_\beta)_{\beta \in B}$ eventually commute up to simplicial homotopy,
		\item $\alpha'''$~is a common upper bound for $\alpha'$ and $\alpha''$.
	\end{enumerate}
	This diagram is commutative up to simplicial homotopy, and therefore yields a commutative diagram on~$\pi_k$, whose bottom-left vertical arrow is trivial. This implies that the upper composition is trivial, in other words, that the inclusion $X_\alpha \into X_{\alpha'''}$ is $\pi_k$-trivial.
\end{proof}

A special case to consider is when given two filtrations $(X_\alpha)_{\alpha\in A}, (Y_\beta)_{\beta\in B}$ of the same simplicial set. If there is a poset map $\phi \colon A \to B$ such that for every $\alpha \in A$ we have $X_\alpha \subseteq X_{\phi(\alpha)}$, we will say that $(X_\beta)_{\beta \in B}$ is \textbf{cofinal} to $(X_\alpha)_{\alpha \in A}$. If each of the filtrations is cofinal to the other, we will simply say that ``they are cofinal''. It is easy to verify that the inclusion maps then assemble to a pair of mutually inverse homotopy equivalences. In particular, by Lemma~\ref{lem:homotopyequiv}, 
cofinal filtrations have the same essential connectedness properties.

\section{The homotopical $\Sigma$-sets}\label{sec:basics}

\subsection{Definition and basic properties}

Throughout the rest of this article, $G$~will denote a topological group that is Hausdorff and locally compact (that is, each neighborhood of a point contains a compact neighborhood).
Every $G$-action on a set~$S$ extends to a $G$-action on $\E S$, and if $S$~is a free $G$-set, the $G$-action on the realization~$|\E S|$ is also free. Note that this does not hold for the free simplicial \emph{complex} on~$S$ if $G$~has torsion.

For a topological space $X$, denote by~$\C(X)$ the directed system of its compact subsets, ordered by inclusion. Given $C\in \C(X)$, we will also write $\C(X)_{\supseteq C}$ to denote the sub-poset consisting of the subsets containing~$C$. If $X$ carries a $G$-action (in particular, if $X = G$), then we consider the filtration of~$\E X$ by the subcomplexes
\[G\cdot \E C := \bigcup_{g \in G} g\cdot \E C,\]
where $C$~ranges over $\C(X)$.
In other words, $G\cdot \E C$ consists of all simplices in~$\E G$ of the form~$(g\cdot x_0, \ldots, g\cdot x_k)$, where $g\in G$ and $x_0, \ldots, x_k \in C$.

The definition of the homotopical $\Sigma$-sets for locally compact groups, to be given soon, is inspired by the following generalization of the finiteness properties $\mathrm{F}_n$ on abstract groups to locally compact groups, due to Abels and Tiemeyer \cite{AT97}.

\begin{dfn}
	For each $n\in \NN$, the group $G$ is of \textbf{type $\mathrm{C}_n$} if the filtration $(G\cdot \E C)_{C\in\C(G)}$ of~$\E G$ is essentially $(n-1)$-connected.
\end{dfn}

An element $\chi\in\TopHom(G,\RR)$, that is, a continuous group homomorphism $\chi \colon G \to \RR$, will be called a \textbf{character} on~$G$. We shall understand the space of characters $\TopHom(G, \RR)$ as equipped with the compact-open topology. Given a character~$\chi$, we consider the translates of~$\E C$ by elements of~$G_\chi := \chi^{-1}(\RR_{\ge 0})$:
\[G_\chi \cdot \E C := \bigcup_{g \in G_\chi} g\cdot \E C.\]
		
\begin{dfn}\label{dfn.topsigma}
	For each $n\in \NN$, the $n$-th \textbf{homotopical $\Sigma$-set} of~$G$ is
	\[\TopS^n(G) := \{\chi \in \TopHom(G,\RR) \st \text{$(G_\chi \cdot \E C)_{C \in \C(G)}$ is essentially $(n-1)$-connected}\}.\]
	For the case $n=0$, we interpret ``being essentially $(-1)$-connected'' as a vacuous condition, so $\TopS^0(G)= \TopHom(G,\RR)$.
\end{dfn}

This defines a descending sequence of subsets
\[ \TopHom(G,\RR) = \TopS^0(G) \supseteq \TopS^1(G) \supseteq \TopS^2(G) \supseteq \dots\]

An immediate consequence is that this theory is trivial for compact groups:

\begin{prop}[Compact groups]
	If $G$~is compact, then for every $n\in \NN$, we have $\TopS^n(G) = \TopHom(G,\RR)$.
\end{prop}
\begin{proof}
	The directed system $\C(G)$ has $G$~itself as a maximum, and for every character~$\chi$ we have that $G_\chi \cdot \E G = \E G$ is contractible.
\end{proof}

In order to connect property~$\mathrm C_n$ and $\TopS^n$, we establish the following fact:

\begin{lem}[$\mathrm C_n$ from~$\TopS^n$]\label{lem:CnfromSigma}
		Let $k \in \NN$, let $\chi\colon G\to \RR$ be a character, and let $X \subseteq Y$ be subsets of~$G$ such that the inclusion $G_\chi \cdot \E X \into G_\chi \cdot \E Y$ is $\pi_k$-trivial. Then the inclusion $G\cdot \E X \into G \cdot \E Y$ is also $\pi_k$-trivial.
\end{lem}
\begin{proof}
	The statement vacuous if $\chi = 0$, so assume $\chi$~is non-trivial.
	
	We shall use the description of~$\pi_k$ given by Lemma~\ref{lem:combspheres}, so let $\eta\in \Map(\partial \Delta^{k+1}, G\cdot \E X)$, which we wish to show represents the trivial element in $G\cdot \E Y$.
	The domain of~$\eta$ is the finite simplicial set~$\SD^{m_0}(\partial \Delta^{k+1})$ for some $m_0\in \NN$. Since $\eta$ maps each simplex into some $G$-translate of~$\E X$,  there is a finite subset $F\subseteq G$ such that the image of~$\eta$ lies in~$F \cdot \E X$. Choose $g\in F$ with $\chi(g)$~minimal, so $F\subseteq g G_\chi$ and thus
	$\eta$~has image contained in $gG_\chi \cdot \E X$. As this simplicial set includes $\pi_k$-trivially into $gG_\chi \cdot \E Y$,  Lemma~\ref{lem:combfilling} yields $m\in \NN$ and $\mu \colon \SD^{m_0+m}(\Delta^{k+1})\to g G_\chi \cdot \E D$ extending $\eta \circ \Phi^m$.
	This filling~$\mu$ also witnesses that $\eta$~is trivial in $G\cdot \E Y$.
\end{proof}

We can now show that $\TopS^n$ does indeed refine property~$\mathrm C_n$.

\begin{prop}[The zero character]\label{prop.zerocharacter}
	For each $n \in \NN$, the following are equivalent:
	\begin{enumerate}
		\item $G$ is of type~$\mathrm{C}_n$,
		\item $0 \in \TopS^n(G)$,
		\item $\TopS^n(G) \ne \emptyset$.
	\end{enumerate}
\end{prop}

\begin{proof}
	($1 \Leftrightarrow 2$) For $\chi = 0$, we have $G_\chi = G$, so the definition of $G$~having type~$\mathrm C_n$ precisely matches the defining condition of $\chi \in \TopS^n(G)$.
	
	($2 \Rightarrow 3$) This statement is trivial.
	
	($3 \Rightarrow 1$) Assume Condition~3 and let $C\in \C(G)$. By choosing any $\chi \in\TopS^n(G)$, we find $D\in \C(G)_{\supseteq C}$ such that the inclusion $G_\chi \cdot \E C \into G_\chi \cdot \E D$ is $\pi_k$-trivial for all $k\le n-1$. Lemma~\ref{lem:CnfromSigma} then yields that also $G \cdot \E C \into G \cdot \E D$ is $\pi_k$-trivial, whence $G$~is of type $\mathrm C_n$.
\end{proof}

As one can directly see from the definition, if some character~$\chi$ lies in $\TopS^n(G)$, then so does every multiple $\lambda \chi$ with $\lambda \in \RR_{>0}$. Together with Proposition~\ref{prop.zerocharacter}, we see that $\TopS^n(G)$ is a cone at~$0$ in the $\RR$-vector space $\TopHom(G,\RR)$. A much deeper fact, which we will prove in Section~\ref{sec:stability}, is that it is the cone over an \emph{open} subset of $\TopHom(G,\RR)\setminus \{0\}$.

\begin{rem}[Characters vs.~rays]
	It would be more in line with tradition to define~$\TopS^n(G)$ only under the assumption that $G$~is of type~$\mathrm C_n$, and then declare $\TopS^n(G)$ to be not a set of characters, but rather a set of equivalence classes of \emph{nonzero} characters, where $\chi \sim \chi'$ if and only if $\chi'$~is a positive real multiple of~$\chi$. This was the approach taken when classical $\Sigma$-sets were originally introduced, and explains why they are denoted by the letter~$\Sigma$: if a discrete group~$G$ is of type~$\mathrm F_n$, then under that convention, $\Sigma^n(G)$~is a subset of the \emph{sphere} $\left(\Hom(G,\RR)\setminus \{0\}\right) / {\sim}$.
	
	However, Proposition~\ref{prop.zerocharacter} tells us how to meaningfully incorporate the zero character into the theory, so that $\TopS^n(G)$ contains the datum of whether $G$~is of type~$\mathrm C_n$, rather than assuming it. With this convention, one can often write statements and proofs about whether a character~$\chi$ lies in~$\TopS^n(G)$ without assuming that $\chi \neq 0$, and with a uniform argument also talk about property~$\mathrm C_n$.
	
	To avoid burdening the text with what is ultimately a minor convention detail, when referring to results in the literature (be it about classical $\Sigma$-sets, or about $\TopS^1(G)$ and~$\TopS^2(G)$ as defined by Kochloukova), we will without further mention state them in this language.
\end{rem}

\subsection{Alternative filtrations}\label{sec:altfilt}

Our next goal is to characterize $\TopS^n(G)$ by means of different filtrations, some of which filter $\E G_\chi$ rather than $\E G$. To compare them, we will use the tools introduced in Section~\ref{sec:filtrations}.
We will also routinely employ the fact that the product~$CD := \{xy \in G \mid x\in C, y\in D\}$ of compact subsets $C,D\subseteq G$ is compact -- indeed, it is the image of the compact set $C\times D$~under the (continuous) multiplication map $G\times G \to G$.

Given a simplicial subset $Y\subseteq \E G$ and a character $\chi \colon G\to \RR$, we denote by~$Y_\chi$ the simplicial subset spanned by the vertices with non-negative $\chi$-value. In other words,
\[Y_\chi := Y \cap \E G_\chi.\]

\begin{lem}[Filtrations of~$\E G_\chi$]\label{lem:trimmedfiltrations}
	Let $\chi \colon G\to \RR$ be a character.
	\begin{enumerate}
		\item The filtrations $((G\cdot \E C)_\chi)_{C\in \C(G)}$ and $((G_\chi \cdot \E C)_\chi)_{C\in \C(G)}$ of $\E G_\chi$ are cofinal.
		\item The filtrations $((G_\chi \cdot \E C)_\chi)_{C\in \C(G)}$ of~$\E G_\chi$ and $(G_\chi \cdot \E C)_{C\in \C(G)}$ of~$\E G$ are simplicially homotopy equivalent.
	\end{enumerate}
\end{lem}
\begin{proof}
	(1) We obviously have inclusions $(G_\chi \cdot \E C)_\chi \into (G\cdot \E C)_\chi$ for each $C\in \C(G)$.
	
	For the other direction, we claim that $(G\cdot \E C)_\chi \subseteq (G_\chi \cdot \E (C^{-1}C))_\chi$. To see this, note that a simplex $g\cdot (x_0, \dots, x_k)$ of $(G\cdot \E C)_\chi$ (with $g\in G$ and $x_i\in C$) may be written as $gx_0\cdot (1, x_0^{-1} x_1, \dots, x_0^{-1} x_k)$, where $gx_0 \in G_\chi$ and $x_0^{-1}x_i\in C^{-1}C$.

	(2) The inclusions $(G_\chi \cdot \E C)_\chi \into G_\chi \cdot \E C$ clearly assemble to a commuting family of maps. We claim there is a simplicial homotopy inverse, consisting of maps $$r_C \colon G_\chi \cdot \E C \to (G_\chi \cdot \E (CC^{-1}))_\chi.$$
	
	To construct each~$r_C$, choose $t_C\in C$ such that the restriction~$\chi|_C$ attains its minimum at~$t_C$, and define~$r_C$ on vertices by~$g \mapsto gt_C^{-1}$. Note that, with this definition, $r_C$~does indeed have image in $G_\chi \cdot \E (CC^{-1})_\chi$, since for every vertex $g x$~of the source space (with $g\in G_\chi$ and $x \in C$), we have $\chi(xt_C^{-1}) \ge 0$, and thus also $\chi(gxt_C^{-1}) \ge 0$.
	
	To see the maps~$r_C$ commute up to simplicial homotopy, let $C\subseteq D$ be compact subsets of~$G$ and consider the (unique) simplicial homotopy
	$H\colon G_\chi \cdot \E C \times \Delta^1 \to \E G_\chi$
	from~$r_C$ to~$r_D$. Explicitly, it is given by $$H\left((g_0, \dots, g_k), \tau^k_i\right) = (g_0 t_C^{-1}, \dots, g_{i-1} t_C^{-1}, g_i t_D^{-1}, \dots, g_k t_D^{-1}),$$ and from this description we see, as before, that $H$~has image in $(G_\chi \cdot \E (DD^{-1}))_\chi$.
	
	With similar arguments, one sees that for every $C\in\C(G)$, both diagrams
	$$\begin{tikzcd}
		(G_\chi \cdot \E C)_\chi \ar[d, hook] \ar[r,hook]&(G_\chi \cdot \E (C\cup CC^{-1}))_\chi\\
		G_\chi \cdot \E C \ar[r, "r_C"] & (G_\chi \cdot \E(CC^{-1}))_\chi \ar[u, hook]
	\end{tikzcd}
	\begin{tikzcd}
		G_\chi \cdot \E C \ar[d, "r_C"] \ar[r,hook]& G_\chi \cdot \E (C\cup CC^{-1})\\
		(G_\chi \cdot \E (CC^{-1}))_\chi \ar[r, hook] & G_\chi \cdot \E (CC^{-1}) \ar[u, hook]
	\end{tikzcd}$$
	commute up to simplicial homotopy, finishing the proof.
\end{proof}

\begin{prop}[$\TopS^n$ via filtrations of~$\E G_\chi$]\label{prop.trimmedsigma}
	Let $n\in \NN$ and let $\chi \colon G\to \RR$ be a character. The following conditions are equivalent:
	\begin{enumerate}
		\item $\chi \in \TopS^n(G)$,
		\item The filtration $\left((G\cdot\E C)_\chi\right)_{C\in \C(G)}$ of~$\E G_\chi$ is essentially $(n-1)$-connected.
		\item The filtration $\left((G_\chi\cdot\E C)_\chi\right)_{C\in \C(G)}$ of~$\E G_\chi$ is essentially $(n-1)$-connected.
	\end{enumerate} 
\end{prop}

\begin{proof}
	From Lemma~\ref{lem:trimmedfiltrations}~(2), together with Lemma~\ref{lem:homotopyequiv}, we see that for every $k\in \NN$, the filtration $(G_\chi \cdot \E C)_{C\in \C(G)}$ is essentially $k$-connected if and only if  $((G_\chi \cdot \E C)_\chi)_{C\in \C(G)}$ is. In particular, we have ``$1\Leftrightarrow 3$''. Similarly, from Lemma~\ref{lem:trimmedfiltrations}~(1) we deduce $``2\Leftrightarrow 3$''.
\end{proof}

We collect yet another characterization of~$\TopS^n(G)$, this time by filtering~$\E G$ simultaneously by the compact subsets of~$G$ and by $\chi$-level sets, which is more reminiscent of the classical theory of $\Sigma$-sets. This will in particular be useful in establishing the equivalence between our $\Sigma$-sets and the classical ones for discrete groups (Theorem~\ref{thm.sigmasmatch}).

Given a character $\chi\colon G\to \RR$, a simplicial subset $Y\subseteq \E G$, and $s\in \RR$, write
\[Y_s := Y \cap \E(\chi^{-1}({[s, + \infty[})).\]
The character $\chi$~is implicit in this notation and should be clear from the context. For $s=0$ we recover $Y_\chi$. We obtain a filtration $(Y_s)_{s\in \RR}$ of~$Y$ by equipping~$\RR$ with the opposite of its usual order.
Recalling that the product $A\times B$ of two posets $A,B$ carries a partial order given by \[(a,b)\le (a',b') \iff \text{$a\le a'$ and $b\le b'$},\]
the character~$\chi$ also determines a filtration of~$\E G$ given by
\[((G_\cdot \E C)_s)_{(C,s)\in \C(G)\times \RR}.\]

\begin{prop}[$\TopS^n$ via a double filtration]\label{prop:doublefiltration}
	For every character $\chi\colon G\to \RR$, the filtrations $(G_\chi \cdot\E C)_{C \in \C(G)}$ and $((G\cdot \E C)_s)_{(C,s) \in \C(G)\times \RR}$ of $\E G$ are cofinal. In particular, $\chi \in \TopS^n(G)$ if and only if $((G\cdot \E C)_s)_{(C,s) \in \C(G)\times \RR}$ is essentially $(n-1)$-connected.
\end{prop}
\begin{proof}
	Given $C\in \C(G)$, one can choose $s = \min \chi(C)$ and then it follows  that $G_\chi \cdot \E C \subseteq (G\cdot\E C)_s$.
	
	For the other direction, assume $\chi \neq 0$; otherwise the statement is trivial. Given $(C,s) \in \C(G)\times \RR$, choose $t\in G$ with $\chi(t) \le s - \max \chi(C)$; we claim
	$(G\cdot \E C)_s \subseteq G_\chi \cdot \E (t C)$. Indeed, consider a simplex $g\cdot (x_0, \dots, x_k)$ with $g\in G$~and~$x_i\in C$ such that for all~$x_i$ we have $\chi(gx_i) \ge s$. Expressing it as $gt^{-1}\cdot (t x_0, \dots, t x_k)$, we see $gt^{-1}\in G_\chi$ by noting that 
	\[\chi(gt^{-1}) = \chi (g) - \chi(t) \ge \chi(g) - s +\max \chi(C) \ge \chi(g) - s + \chi(x_0)  \ge 0. \qedhere\]
\end{proof}

\begin{cor}[Probing at the $0$-level]\label{cor:boundeddrop}
	A character $\chi \colon G\to \RR$ lies in~$\TopS^n(G)$ if and only if for every $C\in \C(G)$ there are $D\in \C(G)_{\supseteq C}$ and $s \le 0$ such that for every $k\le n-1$ the inclusion
	\[(G\cdot \E C)_\chi^{(n-1)} \into (G\cdot \E D)_s\]
	is $\pi_k$-trivial.
\end{cor}
\begin{proof}
	The ``only if'' direction follows immediately from Proposition~\ref{prop:doublefiltration}.
	For the converse, note that since the inclusion of the $(n-1)$-skeleton is $\pi_k$-surjective, the hypothesis implies  that also the inclusion
	\[(G\cdot \E C)_\chi \into (G\cdot \E D)_s\]
	is $\pi_k$-trivial.
	The $\chi = 0$ case is then immediate, so assume $\chi \neq 0$. 
	
	In order to apply Proposition~\ref{prop:doublefiltration}, let $(C,s_0)\in \C(G) \times \RR$ be an index and let $D\in \C(G)_{\supseteq C}$, $s\le 0$ be as in the assumption.
	Choose $g\in G$ with $\chi(g) \le s_0$. Writing $r:= \chi(g)$, we have that also the translated inclusion
	\[g \cdot  (G\cdot \E C)_\chi \into g \cdot (G\cdot \E D)_s =(G\cdot \E D)_{s+r}\]
	is $\pi_k$-trivial, whence so is the composition
	$$(G\cdot \E C)_{s_0} \into g \cdot  (G\cdot \E C)_\chi \into (G\cdot \E D)_{s+r}.$$ This shows $\chi \in \TopS^n(G)$.	
\end{proof}

The usefulness of this corollary is that it allows one to establish that $\chi \in \TopS^n(G)$ by proving only that the stages of the double filtration indexed by~$(C,0)$ (with $C\in \C(G)$) are annihilated. In Proposition~\ref{prop:alonsocriterion}, we will see that if $G$~is of type~$\mathrm C_{n-1}$, then it is actually enough to consider a single stage $(K,0)$, with $K$~a suitable compact set that is even independent of~$\chi$. 

Finally, we consider filtrations of a somewhat more algebraic flavor, given by powers of a fixed compact generating set. In this context, it is convenient to work with subsets $X \subseteq G$ that are  \textbf{centered}, that is, with $1\in X$ and $X=X^{-1}$. Any subset $X$ can be made centered by replacing it with $\dot X := X \cup X^{-1} \cup \{1\}$, and if $X$~is compact, so is $\dot X$. 

\begin{lem}[Compact exhaustion]\label{lem.compactexhaustion}
	Let $C \in \C(G)$ be a centered generating set, and let $D \in \C(G)$. Then for some $m\in \NN$ we have $D\subseteq \Int_G(C^m)$. In particular, the directed system $(C^m)_{m\in \NN}$ is cofinal in~$\C(G)$.
\end{lem}

\begin{proof}
	First, use the fact that $G$, being locally compact and Hausdorff, is a Baire space~\cite[Chapter~IX, Theorem~1]{Bou98}, which means that $G$~is not a countable union of closed subsets with empty interior. In particular, since $G = \bigcup_{k\in \NN} C^k$, some~$C^k$ has nonempty interior, so let $U\subseteq C^k$ be a nonempty open subset of~$G$, and let~$u\in U$. Then for every $l\in \NN$ we have $C^l \subseteq C^{l}Uu^{-1} \subseteq C^{2k+l}$.
	
	Now, the set $C^{l}Uu^{-1}$ may be expressed as the union of open sets
	$\bigcup_{x\in C^l} xUu^{-1}$, so it is open in~$G$. Therefore, we have in fact $C^l \subseteq \Int_G C^{2k+l}$, and since $G=\bigcup_{l \in \NN} C^l$, we deduce $G$~is the ascending union of open subsets $G = \bigcup_{m \in \NN} \Int_G(C^m)$. As $D$~is compact, it is contained in some~$\Int_G(C^m)$.
	
	The second statement is immediate from $D\subseteq C^m \in \C(G)$.
\end{proof}

A version of Lemma~\ref{lem.compactexhaustion} is already present in the work of Abels and Tiemeyer  \cite[proof of Theorem~2.3.3, 1st paragraph]{AT97}.
A generalization of this statement (and indeed of the proof just supplied) will be given in Lemma~\ref{lem.compactexhaustionchi}.

\begin{cor}[$\TopS^n$ via compact generating sets]\label{cor.powersofC}
	Suppose $G$~is compactly generated, with $C$~a compact centered generating set, let $n\in \NN$, and let $\chi \colon G\to \RR$ be a character. The condition $\chi \in \TopS^n(G)$ is equivalent to essential $(n-1)$-connectedness of some / all the following filtrations:
	\begin{enumerate}
		\item $(G_\chi \cdot \E C^m)_{m\in \NN}$ of~$\E G$,
		\item $((G\cdot \E C^m)_\chi)_{m\in \NN}$ of~$\E G_\chi$,
		\item $((G_\chi \cdot \E C^m)_\chi)_{m\in \NN}$ of~$\E G_\chi$,
		\item $((G\cdot \E C^m)_s)_{(m,s) \in \NN \times \RR}$ of~$\E G$.
	\end{enumerate}
\end{cor}

\begin{proof}
	Lemma~\ref{lem.compactexhaustion} directly implies that the filtration in Condition~1 is cofinal to $(G_\chi \cdot \E C)_{C\in \C(G)}$, the defining filtration of~$\TopS^n(G)$. And obviously the latter is cofinal to the former, so we conclude from Lemma~\ref{lem:homotopyequiv} that these two filtrations have the same connectedness properties. Thus Condition~1 holds if and only if $\chi \in \TopS^n(G)$.
	
	The same argument, applied to the descriptions of~$\TopS^n(G)$ given by Propositions~\ref{prop.trimmedsigma} and~\ref{prop:doublefiltration}, yields the remaining statements.
\end{proof}

\begin{rem}(Functorial filtrations)\label{rem:functfilt}
	Suppose we are given a continuous group homomorphism $f\colon H\to G$ and characters~$\chi, \hat\chi$ on $H, G$, respectively, fitting into a commutative triangle
	$$\begin{tikzcd}
		H \ar[rr,"f"] \ar[dr,"\hat\chi"]&& G \ar[dl,"\chi"']\\
		&\RR
	\end{tikzcd}.$$
	Then, we have $f(H_{\hat\chi}) \subseteq G_\chi$, so for every $C\in\C(H)$, the map of simplicial sets $\E H \to \E G$ induced by~$f$ restricts to a map $H_{\hat\chi} \cdot \E C \to G_\chi \cdot \E f(C)$. Moreover, since $f$~is continuous, the set $f(C)$~is compact. Thus, taken over all $C\in\C(H)$, these maps assemble to a commuting family of maps from the filtration $(H_{\hat\chi} \cdot \E C)_{C\in\C(H)}$ to the filtration $(G_\chi\cdot \E D)_{D\in \C(G)}$. The analogous statement holds also for the filtrations in Proposition~\ref{prop.trimmedsigma}, Proposition~\ref{prop:doublefiltration}, and Corollary~\ref{cor.powersofC}.
\end{rem}

\section{Relation to the classical $\Sigma$-sets}\label{sec:classicalsigma}
	Throughout this section, assume $G$ is discrete and fix $n\in \NN$. Our goal is to show that the set~$\TopS^n(G)$ generalizes the classical homotopical $\Sigma$-set~$\Sigma^n(G)$, defined for abstract groups.

	\subsection{A recap of {$\Sigma^n$}}
	We begin by recalling the classical construction of~$\Sigma^n(G)$ \cite[Section~II.3]{Ren88}\footnote{Renz writes ``$^*\!\Sigma^n(G)$'', reserving the notation $\Sigma^n(G)$ for the ``homological $\Sigma$-set''~$\Sigma^n(G; \ZZ)$, which we focus on in a separate article \cite{BHQ}.}. We also remind the reader that a \textbf{$G$-CW complex} is a CW complex with a $G$-action by homeomorphisms that map cells to cells.
	
	\begin{dfn}\label{dfn:valuation}
			Let $Y$~be a free $G$-CW complex and $\chi\colon G \to \RR$ a character. A \textbf{valuation} on~$Y$ associated to $\chi$ is a continuous map $v\colon Y \to \RR$ such that
		\begin{enumerate}
			\item for every $g \in G$ and $y \in Y$, we have $v(g\cdot y) = v(y) + \chi(g)$,
			\item $v(Y^{(0)}) = \chi(G)$, and
			\item for each open cell~$e$ of~$Y$, the maximum and minimum values of $v|_{\overline{e}}$ are attained in $\partial \overline{e}$ (and hence, by induction, in $\overline{e} \cap Y^{(0)}$).
		\end{enumerate}
	\end{dfn}
	
	It is always possible to find a valuation~$v$ for given~$Y$~and~$\chi$ \cite[Section II.2]{Ren88}. We then denote by~$Y_s$ the full sub-complex of $Y$ spanned by the $0$-cells~$y$ with $v(y) \ge s$ (the valuation~$v$ being implicit in the notation). For a closed cell~$\overline e$ of~$Y$, we will also use the shorthand $v(\overline{e}) := \min(v|_{\overline{e}})$.
	
	\begin{dfn} \label{dfn.sigma}
		The $n$-th \textbf{homotopical $\Sigma$-set} $\Sigma^n(G)$ of~$G$ is the subset of~$\Hom(G, \RR)$ comprised of the characters $\chi\colon G \to \RR$ with the following property:
		there exists a free $G$-CW complex~$Y$ with $G$-finite $n$-skeleton, together with a valuation associated to~$\chi$, such that $Y_0$~is $(n-1)$-connected.
	\end{dfn}
	
	\begin{rem}[Tweaking~$Y$]
		Using an argument similar to the proof of Lemma~\ref{lem:CnfromSigma}, one sees that every complex~$Y$ as in this definition is automatically $(n-1)$-connected. By attaching (possibly infinitely many) $G$-orbits of cells in dimension $\ge n+1$, one can therefore arrange for $Y$~to be contractible, and thus for $G\backslash Y$~to be a $\K(G,1)$ with finite $n$-skeleton. On the other hand, if one prefers smaller complexes, one can instead discard from~$Y$ all cells of dimension $\ge n+1$ without ruining $(n-1)$-connectedness, thus making it into a $G$-finite complex of dimension $\le n$.
	\end{rem}
	
	The above definition of $\Sigma^n(G)$ is not very practical because it is unclear for which complexes~$Y$ the condition ``$Y_0$~is $(n-1)$-connected'' detects whether $\chi \in \Sigma^n(G)$. However, we have the following alternative characterization \cite[Chapter~IV, Satz~3.4]{Ren88}:
	
	\begin{thm}[$\Sigma^n$ via essential $(n-1)$-connectedness]\label{thm.renz}
		Let $\chi \in \Hom(G,\RR)$, and let~$Y$ be an $(n-1)$-connected free $G$-CW complex with $G$-finite $n$-skeleton and $\chi$-associated valuation. Then $\chi \in \Sigma^n(G)$ if and only if the filtration $(Y_s)_{s\in \RR}$ is essentially $(n-1)$-connected.
	\end{thm}

For the sake of completeness, and because Renz's doctoral dissertation is only available in German, we give a sketch of the proof. One of the directions requires using a result of Whitehead \cite{Whi50}, which we first state\footnote{We present Whitehead's Theorem as phrased by Renz \cite[Satz~3.1]{Ren88}; in Whitehead's paper, the result is a consequence of Theorems 13~and~14.}.

Recall that a cellular map $f\colon K \to L$ between CW complexes is called an \textbf{$n$-equi\-valence} if there is a cellular map $g\colon L \to K$ such that the restricted composition $g\circ f \colon K^{(n-1)} \to K$ is homotopic to the inclusion $K^{(n-1)} \into K$, and $f\circ g \colon L^{(n-1)} \to L$ is homotopic to $L^{(n-1)} \into L$. The map~$g$ is then called an \textbf{$n$-inverse} of~$f$. By a different theorem of Whitehead, $f$~is an $n$-equivalence if and only if it induces isomorphisms on~$\pi_k$ for all $k\le n-1$ \cite[Section~4.4, Exercise~4]{Geo08}. If such $f$~exists, we say $K$~and~$L$ have the same \textbf{$n$-type}. A \textbf{simple homotopy deformation} of a CW complex consists of the attachment of a ball along a hemisphere of its boundary, or the inverse of this operation. In other words, one attaches or removes a canceling pair of a $k$-cell and a $(k+1)$-cell.

\begin{thm}[Whitehead]\label{thm.Whitehead}
	Let $K, L$ be finite CW complexes of the same $n$-type. Then $L$ is obtained from~$K$ by a sequence of simple homotopy deformations and the attachment/removal of cells in dimension $\ge n+1$. 
\end{thm}

\begin{proof}[Proof sketch of Theorem~\ref{thm.renz}]
	($\Rightarrow$) Let $X$ be a free $G$-CW complex with $G$-finite $n$-skeleton and $\chi$-associated valuation, such that $X_0$~is $(n-1)$-connected. Since discarding cells in dimension $\ge n+1$ has no effect on $(n-1)$-connectedness, we may assume that $X$~and~$Y$ have dimension at most~$n$; in particular, they are $G$-finite.
	
	Put $K:=G\backslash X$ and $L:= G\backslash Y$.
	We first construct an $n$-equivalence $f \colon K \to L$ by a standard induction on skeleta (assuming $n\ge 2$; for the low-dimensional cases just stop the construction at a suitable stage): start by choosing a spanning tree of~$K^{(1)}$ containing the base-point~$x$, and map it to the base-point~$y$ of~$L$. Each of the remaining edges of~$K^{(1)}$ now corresponds to an element of $\pi_1(K, x) \cong G \cong \pi_1(L,y)$, so map it to a corresponding loop in~$L^{(1)}$.
	The map can be extended to $K^{(2)}$~because loops in $K^{(1)}$ that are trivial on $\pi_1(K)$ got mapped to trivial loops of~$Y$. We then continue inductively: if $f$ has been defined on the $k$-skeleton, for $2\le k<n$, the fact that $\pi_k(L) =0$ allows one to extend~$f$ to each $(k+1)$-cell. In the end we get a map $f \colon K\to L$ inducing isomorphisms on~$\pi_k$ for all $k\le n-1$, so $f$~is an $n$-equivalence. 
	
	Now, as $K$~and~$L$ are finite CW complexes, we conclude by Theorem~\ref{thm.Whitehead} that $L$~differs from $K$~by a sequence of simple homotopy moves and the attachment or deletion of high-dimensional cells. This sequence can be $G$-equivariantly lifted to the universal covers -- each move now consists of attaching/removing $G$-orbits of canceling pairs of cells, or of attaching/removing $G$-orbits of cells in dimension $\ge n+1$. 
	
	With each such move, one can extend the original $\chi$-valuation on~$X$ in such a way that the successive filtered $G$-CW complexes remain essentially $(n-1)$-connected \cite[Chapter~IV, Satz~1.3]{Ren88}. Ultimately, we produce a $\chi$-associated valuation on~$Y$ whose corresponding filtration is essentially $(n-1)$-connected.
	This implies that the filtration determined by the given $\chi$-valuation on~$Y$ is also essentially $(n-1)$-connected, since any two $\chi$-associated valuations on a $G$-finite $G$-CW complex give rise to cofinal filtrations.
	
	($\Leftarrow$) Given an essentially $(n-1)$-connected $G$-CW complex~$Y$ with $G$-finite $n$-skeleton and $\chi$-associated valuation~$v$, Renz explains how to attach finitely many $G$-orbits of cells to produce a $G$-CW complex $X$~with a $\chi$-associated valuation~$v'$ extending~$v$, such that $X$~supports a ``$v'$-homotopy in dimension~$n$'' \cite[Chapter~IV, Satz~2.4]{Ren88}. He then shows that this implies $X_0$~is $(n-1)$-connected \cite[Chapter~IV, Satz~2.5]{Ren88}.
\end{proof}

\subsection{Brown's Criterion for characters}
	
	Next, we  aim to describe~$\Sigma^n(G)$ in terms of the filtration $((G \cdot \E C)_s)_{(s,C) \in \RR\times \C(G)}$ of Proposition~\ref{prop:doublefiltration}. To that end, we establish an analogue of (the free version of) Brown's Criterion, which we first recall \cite[Theorem~7.4.1]{Geo08}: 
	
	\begin{thm}[Brown's Criterion, free version]\label{thm.freebrown}
		Let $Y$ be a free $(n-1)$-connected $G$-CW complex with a filtration $(Y_\alpha)_{\alpha\in A}$ by $G$-subcomplexes with $G$-finite $n$-skeleta.
		Then $G$ is of type $\mathrm{F}_n$ if and only if $(Y_\alpha)_{\alpha\in A}$ is essentially $(n-1)$-connected.
	\end{thm}
	
	The analogous result for characters features a double filtration obtained by introducing a $\chi$-value parameter: 
	
	\begin{prop}[Brown's Criterion for characters, free version]
	\label{prop:brownforchars_free}
	Let $Y$ be a free $(n-1)$-connected $G$-CW complex, and let $(Y_\alpha)_{\alpha \in A}$ be a filtration of~$Y$ by $G$-subcomplexes with $G$-finite $n$-skeleta. Let $\chi\colon G \to \RR$ be a character and $v$~a valuation on~$Y$ associated to~$\chi$.
	
	Then $\chi \in \Sigma^n(G)$ if and only if the filtration $(Y_{\alpha, s})_{(\alpha, s)\in A\times \RR}$ of~$Y$ is essentially $(n-1)$-connected.
	\end{prop}

	Note that the $\chi=0$ case recovers Theorem~\ref{thm.freebrown}. Similar results appear in the doctoral dissertation of Meinert \cite[Satz~3.3.14, Satz~3.4.1]{Mei93}. In Theorem~\ref{thm:brownforchars_general}, we will present a generalization of Proposition~\ref{prop:brownforchars_free}, where the hypothesis that $G$~acts freely is weakened.

	Before proving Proposition~\ref{prop:brownforchars_free}, we establish two auxiliary results.

	\begin{dfn}
	If $X,Y$ are $G$-CW complexes, then a $G$-equivariant cellular map $f\colon X \to Y$ is a \textbf{$G$-$n$-equivalence} if there exists a $G$-equivariant cellular map $g\colon Y \to X$ and cellular $G$-equivariant homotopies:
	\begin{itemize}
		\item $H_{fg}$ from the composition $Y^{(n-1)} \overset{g}{\to} X \overset{f}{\to} Y$ to the inclusion $Y^{(n-1)} \into Y$,
		\item $H_{gf}$ from the composition $X^{(n-1)} \overset{f}{\to} Y \overset{g}{\to} X$ to the inclusion $X^{(n-1)} \into X$.
	\end{itemize}
\end{dfn}

\begin{lem}[$G$-$n$-equivalences]\label{lem.Gnequiv}
	If $X,Y$ are $(n-1)$-connected free $G$-CW complexes, then there is a $G$-$n$-equivalence $X \to Y$.
\end{lem}
\begin{proof}
	Write $K := G\backslash X, L := G\backslash Y$, so $\pi_1(L) \cong G\cong \pi_1(K)$. 
	Using the fact that $\pi_k(L) = \pi_k(Y) = 0$ for $2\le k \le n-1$, we inductively construct an $n$-equivalence $f_0 \colon K^{(n)} \to L$ (as in the proof of Theorem~\ref{thm.renz}). The inclusion $K^{(n)} \into K$ is also an $n$-equivalence, so let $j \colon K \to K^{(n)}$ be an $n$-inverse. The composition $f\colon K \overset{j}{\to} K^{(n)} \overset{f_0}{\to}  L$ is therefore an $n$-equivalence, and choosing a lift $\tilde f \colon X \to Y$ yields the desired $G$-$n$-equivalence: 
	this is certified by the lifts of cellular homotopies witnessing that $f$~is an $n$-equivalence.
\end{proof}

\begin{lem}[Equivariant maps respect lower bounds]\label{lem.equivlowerbound}
	Let $\chi \colon G \to \RR$ be a character, let $Y,Z$ be free $G$-CW complexes with valuations $v,u$, respectively, associated to~$\chi$. Assume moreover that $Y$~is $G$-finite, and let $f \colon Y \to Z$ be a $G$-equivariant cellular map. Then for every $s\in \RR$, there exists $t \in \RR$ such that
	\[f(Y_s) \subseteq Z_t.\]
\end{lem}
\begin{proof}
	If $\overline e$~is a closed cell in $Y$, then its $(u\circ f)$-image is a closed bounded interval; let us denote by~$uf(\bar e)$ the minimum of this interval. Since $f$~is $G$-equivariant,
	we see that for every $g\in G$, we have $uf(g\cdot \bar e) = \chi (g) + uf(\bar e)$.  If, in addition, $g\cdot \bar e$~lies in~$Y_s$, we have $s \le v(g\cdot \bar e) = \chi(g) + v(\bar e)$, and so $$uf(g\cdot \bar e)\ge s - v(\bar e) + uf(\bar e) =: t_e.$$
	
	Note this lower bound~$t_e$ is independent of the chosen $G$-translate of~$\bar e$. Considering all cells in the $G$-orbit $G\cdot \bar e$ that lie in~$Y_s$, we thus conclude 
	\[f(Y_s \cap G\cdot \bar e) \subseteq Z_{t_e}.\]
	If we now choose representatives $e_1, \dots, e_m$ for the (finitely many) $G$-orbits of cells in~$Y$ and put $t:= \min \{t_{e_i} \mid 1\le i \le m \}$, it follows that $f(Y_s)\subseteq Z_t$.
\end{proof}

A special case worth noting is the following:

\begin{cor}[Equivariant homotopies respect lower bounds]\label{cor.htpylowerbound}
	Let $\chi \colon G \to \RR$ be a character, let $Y,Z$ be free $G$-CW complexes with valuations $v,u$, respectively, associated to~$\chi$. Assume moreover that $Y$~is $G$-finite, and let $H \colon Y \times [0,1] \to Z$ be a $G$-equivariant cellular homotopy. Then for every $s\in \RR$, there exists $t \in \RR$ such that
	\[H(Y_s \times [0,1]) \subseteq Z_t.\]
\end{cor}

\begin{proof}
	The natural CW structure on $Y \times [0,1]$, together with the  $G$-action on the first factor, makes it free and cocompact. Moreover, $v$~induces a valuation on $Y \times [0,1]$ associated to~$\chi$ obtained by precomposing~$v$ with the projection $Y \times [0,1] \onto Y$. From Lemma~\ref{lem.equivlowerbound}, we now get $t\in \RR$ such that $Z_t$ contains the $H$-image of the sub-complex $(Y \times [0,1])_s = Y_s \times [0,1]$.
\end{proof}

\begin{proof}[Proof of Proposition~\ref{prop:brownforchars_free}]
	$(\Rightarrow)$ We assume $\chi \in \Sigma^n(G)$, so let $X$ be a free $G$-CW complex with $\chi$-associated valuation and $G$-finite $n$-skeleton, such that $X_0$ is $(n-1)$-connected (and hence so is~$X$). For simplicity, assume also that $X$~has no cells of dimension $\ge n+1$ and thus is $G$-finite.
	
	By Lemma~\ref{lem.Gnequiv}, there is an inverse pair of $G$-$n$-equivalences $f\colon Y \to X$, $g\colon X \to Y$, and a cellular $G$-homotopy~$H$ from the composition $Y^{(n-1)} \overset{f}{\to} X \overset{g}{\to} Y$
	to the inclusion $Y^{(n-1)} \into Y$.
	
	We will show that given $(\alpha, s) \in A\times \RR$, there are $\beta \ge \alpha$ and $t\le s$ such that
	\begin{itemize}
		\item the restriction of~$H$ to~$Y_{\alpha,s}^{(n-1)} \times [0,1]$ has image in $Y_{\beta, t}$,
		\item the composition $g\circ f$, when restricted to $Y_{\alpha,s}^{(n-1)} \to Y_{\beta,t}$, is $\pi_k$-trivial for all $k \le n-1$.			
	\end{itemize}
	Once these two points are established, the conclusion will easily follow: they imply that the composition $Y_{\alpha,s}^{(n-1)} \into Y_{\alpha,s} \into Y_{\beta,t}$, is $\pi_k$-trivial for $k \le n-1$, and since the first of these inclusions is $\pi_k$-surjective, it follows that $Y_{\alpha,s} \into Y_{\beta,t}$ is $\pi_k$-trivial.
	
	Let us address the first point: Since $H$~is $G$-equivariant and $Y_\alpha^{(n-1)}$~is $G$-finite, the $H$-image of $Y_\alpha^{(n-1)} \times [0,1]$ in~$Y$ is cocompact. Thus this image is contained in $Y_{\beta_0}$ for some $\beta_0 \ge \alpha$.
	Applying Corollary~\ref{cor.htpylowerbound} to the (restricted) homotopy $H \colon Y_\alpha^{(n-1)} \times [0,1] \to Y_{\beta_0}$ yields $t_0\le s$ such that one can further restrict to
	$H \colon Y_{\alpha,s}^{(n-1)} \times [0,1] \to Y_{\beta_0,t_0}$.
	
	To the second point: 
	Applying Lemma~\ref{lem.equivlowerbound} to the restricted map $f\colon Y^{(n-1)}_\alpha \to X$, we find $s'\in \RR$ such that $f(Y^{(n-1)}_{\alpha,s}) \subseteq X_{s'}$, and by essential $(n-1)$-connectedness of $(X_r)_{r\in \RR}$, there is $s'' \le s'$ making the inclusion $X_{s'} \into X_{s''}$ trivial on~$\pi_k$ for all $k\le n-1$. Since $X$~is $G$-finite, its $g$-image is contained in some $Y_{\beta_1}$, and with another application of Lemma~\ref{lem.equivlowerbound}, this time to the co-restricted map $g \colon X \to Y_{\beta_1}$, we obtain $t_1$ giving a co-restriction
	$g \colon X_{s''} \to Y_{\beta_1,t_1}$. Overall, we see the composition $Y^{(n-1)}_{s,\alpha} \overset{f}{\to} X_{s'} \into X_{s''} \overset{g}{\to} Y_{\beta_1, t_1}$ is $\pi_k$-trivial for $k\le n-1$.
	
	Choosing $\beta := \max\{\beta_0, \beta_1\}$ and $t:=\min \{t_0, t_1\}$ finishes the proof of this implication.
	
	$(\Leftarrow)$ Assume, as we may, that $Y$ has no cells of dimension $\ge n+1$, so each $Y_\alpha$ is $G$-finite. The hypothesis that $(Y_{\alpha,s})_{(\alpha,s)\in A \times \RR}$ is essentially $(n-1)$-connected implies that also  $(Y_\alpha)_{\alpha\in A}$ is essentially $(n-1)$-connected: indeed, given $\alpha \in A$, choose $(\beta,s) \in A\times \RR$ such that the inclusion $Y_{\alpha,0} \into Y_{\beta,s}$ is $\pi_k$-trivial for $k\le n-1$.
	Then, for every map $\eta \colon \Sph^k \to Y_\alpha$, one can find $t\in G$ such that the $t$-translate $t\cdot \eta$ (that is, the post-composition of~$\eta$ with the action by~$t$) has image in $Y_{\alpha,0}$. This translate is thus null-homotopic in~$Y_{\beta, s}$, and a null-homotopy, once regarded in~$Y_\beta$, can be translated by~$t^{-1}$ to produce a null-homotopy of~$\eta$.
	
	Brown's Criterion (Theorem~\ref{thm.freebrown}) now tells us that $G$~is of type~$\mathrm{F}_n$, so let $X$~be an $(n-1)$-connected free $G$-CW complex with $G$-finite $n$-skeleton, and construct a valuation on~$X$ associated to~$\chi$.
	
	From this point, the proof follows almost automatically by applying the same ideas as in the previous implication.	 	
	First, use Lemma~\ref{lem.Gnequiv} to construct an inverse pair of $G$-$n$-equivalences $f\colon Y\to X$, $g\colon X \to Y$, so there is a cellular $G$-equivariant homotopy~$H$ from the composition $X^{(n-1)} \xrightarrow{g} Y \xrightarrow {f} X$
	to the inclusion $ X^{(n-1)} \into  X$.
	
	Fix $s\in \RR$.
	As $X ^{(n-1)}$~is $G$-finite,
	Corollary~\ref{cor.htpylowerbound} finds $t_0\le s$ such that $H(X^{(n-1)}_s\times [0,1]) \subseteq X_{t_0}$. 	
	On the other hand,  $g(X^{(n-1)}) \subseteq Y_\alpha$ for some $\alpha\in A$, and by Lemma~\ref{lem.equivlowerbound} we may find $s'\in \RR$ with $g(X_s^{(n-1)}) \subseteq Y_{\alpha,s'}$. From essential $(n-1)$-connectedness of $(Y_{\alpha, r})_{(\alpha,r) \in A\times \RR}$, we find $\beta \ge \alpha$ and $s''\le s'$ such that the inclusion $Y_{\alpha,s'} \into Y_{\beta,s''}$ is $\pi_k$-trivial for $k\le n-1$. Since $Y_\beta$~is $G$-finite, one final application of Lemma~\ref{lem.equivlowerbound} yields $t_1\in \RR$ such that $f(Y_{\beta,s''}) \subseteq X_{t_1}$.
	
	Altogether, choosing $t:=\min \{t_0 , t_1\}$, we conclude that the inclusion $X^{(n-1)}_s \into X_t$ is homotopic to the composition $X^{(n-1)}_s \overset{g}{\to} Y_{\alpha, s'}\into Y_{\beta, s''} \overset{f}{\to} X_t$, which in turn is $\pi_k$-trivial for $k \le n-1$. As $X_s^{(n-1)} \into X_s$ is $\pi_k$-surjective, it follows that $X_s \into X_t$ is $\pi_k$-trivial.
	Therefore $(X_r)_{r \in \RR}$ is essentially $(n-1)$-connected and $\chi \in \Sigma^n(G)$ by the characterization given in Theorem~\ref{thm.renz}.
\end{proof}

\begin{thm}[The classical $\Sigma^n$] \label{thm.sigmasmatch}
	If $G$~is discrete, then for every $n\in \NN$, we have
	\[\Sigma^n(G) = \TopS^n(G).\]
\end{thm}
	
\begin{proof} Let $\chi \in \Hom(G,\RR)$ and equip $|\E G|$, regarded as a free $G$-CW complex, with the $\chi$-associated valuation given on vertices by $v(g) = \chi(g)$ and on higher cells by piecewise-linear interpolation. For every $C\in \C(G)$, that is, for every finite subset $C\subseteq G$, we see that $|\E C|$~is a finite CW complex, and so $|G\cdot \E C|$ is $G$-finite. Therefore, Proposition~\ref{prop:brownforchars_free} applies to the filtration $(|G\cdot \E C|)_{C\in \C(G)}$, telling us that $(n-1)$-connectedness of the filtration $((G\cdot \E C)_s)_{(C,s)\in \C(G)\times \RR}$ is equivalent to $\chi \in \Sigma^n(G)$. But by Proposition~\ref{prop:doublefiltration}, this is also equivalent to $\chi \in \TopS^n(G)$.
\end{proof}

We close this section by giving a generalization of Proposition~\ref{prop:brownforchars_free} that extends the complete version of Brown's Criterion \cite[exercise on p.~179]{Geo08} to the setting of characters. This more general version relaxes the assumption that the action of~$G$ on~$Y$ is free, to cell stabilizers satisfying suitable finiteness properties. We still need the action to be \textbf{rigid}, meaning that the stabilizer~$G_e$ of each cell~$e$ acts trivially on~$e$.

The proof will use the Borel construction and stack rebuilding, as explained in Geoghegan's book \cite[Section~6.1]{Geo08}, and one of the relevant notions is that of a \textbf{stack}. The precise definition of a stack is somewhat involved~\cite[p.~146]{Geo08}, but roughly speaking, it is a cellular map $q \colon A\to C$ of CW complexes, where for each open cell~$e$ of~$C$, the restriction to~$q^{-1}(e)$ looks like a projection $e\times F_e \to e$, where $F_e$ is a CW complex that might vary with~$e$.

\begin{thm}[Brown's Criterion for characters]
	\label{thm:brownforchars_general}
	Let $X$ be an $(n-1)$-connected rigid $G$-CW complex, and $(X_\alpha)_{\alpha \in A}$ a filtration of~$X$ by $G$-subcomplexes with $G$-finite $n$-skeleta. Assume that for each $k$-cell~$e$, its stabilizer $G_e$ is of type $\mathrm{F}_{n-k}$. Let $\chi\colon G 	\to \RR$ be a character and $v$~a valuation on~$X$ associated to~$\chi$.
	
	Then $\chi \in \Sigma^n(G)$ if and only if the filtration $(X_{\alpha, s})_{(\alpha, s)\in A\times \RR}$ is essentially $(n-1)$-connected.
\end{thm}
\begin{proof}
	We will construct a \emph{free} $G$-CW complex $Y$ and a $G$-equivariant
	map $\pi \colon Y \to X$ satisfying the following two
	conditions:
	\begin{enumerate}
		\item The map~$\pi$ is a \textbf{hereditary} homotopy equivalence, that is, for each subcomplex $X' \subseteq X$, it restricts to a homotopy equivalence $\pi^{-1}(X') \to X'$.
		\item If a $G$-subcomplex~$X' \subseteq X$ has $G$-finite $n$-skeleton, then so does $\pi^{-1}(X')$.
	\end{enumerate}
	Assuming for the moment that we have produced such $Y$ and~$\pi$, consider the filtration of~$Y$ by the preimages $Y_{\alpha,	s}:=\pi^{-1}(X_{\alpha,s})$, where $(\alpha, s)\in A\times \RR$. By Condition~2, each $Y_\alpha := \pi^{-1}(X_\alpha)$ has $G$-finite $n$-skeleton. Thus, by the ``free version'' of the statement, Proposition~\ref{prop:brownforchars_free}, we see that $\chi \in \Sigma^n(G)$ if and only if $(Y_{\alpha, s})_{(\alpha, s)\in A\times \RR}$ is essentially $(n-1)$-connected. 
	On the other hand, Condition~1 implies that this is equivalent to $(X_{\alpha, s})_{(\alpha, s)\in A\times \RR}$ being essentially $(n-1)$-connected, whence the theorem follows.
	
	It thus remains to construct $Y$ and $\pi\colon Y \to X$. Let $K$ be a classifying space for $G$ and
	$\tilde{K}$~its universal cover, and endow the product $\tilde{K}\times
	X$ with the diagonal $G$-action. The quotient $Z:=G\backslash(\tilde K \times X)$ is then the total space of a stack
	\[
	q\colon Z \to G\backslash X,
	\]
	where the fiber $F_{e}$ over a cell~$e$ of~$G\backslash X$ is $G_e \backslash \tilde{K}$. In particular, $F_e$~is a classifying space~$K_e$ for~$G_e$.
	
	Stack rebuilding~\cite[Proposition~6.1.4]{Geo08} now allows us to improve~$q$ to a stack $q'\colon Z' \to G\backslash X$ such that for every $k\le n$ and $k$-cell~$e$, the fiber $F'_e$ over~$e$ is a classifying space~$K_e'$ for~$G_e$ \emph{with finite $(n-k)$-skeleton}. Moreover, the construction yields a diagram
	\[
	\begin{tikzcd}
		Z \arrow[rr,"h"] \arrow[dr, "q" swap]
		& & Z' \arrow[dl,"q'"]
		\\
		& G\backslash X
	\end{tikzcd}
	\]
	that commutes up to homotopy over each cell. We take~$Y$ to be the universal cover of~$Z'$ and $\pi\colon Y \to X$ a $G$-equivariant lift of~$q'$. Then $\pi \colon Y \to X$ is a $G$-equivariant stack, and it satisfies Condition~1 \cite[exercise on p.~148]{Geo08}. Moreover, the fiber over each $k$-cell~$e$ is the universal cover of $K_e'$, so its $(n-k)$-skeleton $G_e$-finite. It follows that Condition~2 is also satisfied.
	\end{proof}


\section{Compact generation and {$\TopS^1$}}\label{sec:sigma1}
	In this section we connect classical characterizations of $\Sigma^1$, in terms of the Cayley graph with respect to a finite generating set, to the setting of topological groups. From now on, $G$~will always be a locally compact Hausdorff group. The first step in crossing the bridge from the characterizations given Section~\ref{sec:basics}, in terms of filtrations of $\E G_\chi$, to a more classical setting is the following.
		
	\begin{prop}[Connectedness at a finite stage]\label{prop.sigma1stable}
		A character~$\chi\colon G \to \RR$ is in $\TopS^1(G)$ if and only if for some centered $C \in \C(G)$, the simplicial set $(G_\chi \cdot \E C)_\chi$ is connected. 
	\end{prop}

	\begin{proof}
		($\Rightarrow$) If $\chi \in \TopS^1(G)$, by Proposition~\ref{prop.trimmedsigma} there is $C \in \C(G)_{\supseteq \{1\}}$, which we may assume is centered, such that the inclusion $(G_\chi\cdot \E \{1\})_\chi \into (G_\chi\cdot \E C)_\chi$ is $\pi_0$-trivial. But $(G_\chi \cdot \E \{1\})_\chi^{(0)} = (G_\chi\cdot \E C)_\chi^{(0)}$, so that inclusion is also $\pi_0$-surjective. Thus, $(G_\chi\cdot \E C)_\chi$~is connected. 
		
		($\Leftarrow$) If $(G_\chi\cdot \E C)_\chi$ is connected, then as all maps in the filtration $((G_\chi \cdot \E D)_\chi)_{D\in\C(G)}$ are $\pi_0$-surjective, we conclude that all $(G_\chi\cdot \E D)_\chi$ with $C\subseteq D$ are connected. The filtration is then essentially $0$-connected and so $\chi \in \TopS^1(G)$.
	\end{proof}
	
	The characterization of characters~$\chi$ being in $\TopS^1(G)$ given by Proposition~\ref{prop.sigma1stable} can be framed in terms of the Cayley graph, recovering much of the classical $\Sigma^1$-theory.

	\begin{dfn} The \textbf{label} of an edge $(g_0,g_1)$ of~$\E G$ is the element $g_0^{-1}g_1 \in G$. It is easy to see that for every subset $X\subseteq G$, the edges of $G \cdot \E X$ are precisely the edges of~$\E G$ with labels in~$X^{-1}X$.
		
		The \textbf{Cayley graph} of~$G$ with respect to a subset~$X \subseteq G$ is the 
		 simplicial subset $\Gamma(G, X) \subseteq \E G^{(1)}$ consisting of all vertices of $\E G$, and having as $1$-simplices the edges that are labeled by an element of~$X$.
		
	Given a character~$\chi$ on~$G$ and $s \in \RR$, we will denote
	\[\Gamma(G, X)_s := \Gamma(G,X) \cap \E(\chi^{-1}([s, +\infty[)).\]
	The character $\chi$ is implicit in this notation, but for the special case $s=0$, we make it explicit by writing $\Gamma(G,X)_\chi := \Gamma(G,X)_0$.
\end{dfn}
	
	With this definition, $\Gamma(G, X)$ is connected precisely if $X$~generates~$G$.
	
	Having fixed a character $\chi \colon G \to \RR$ and a centered subset $X\subseteq G$, we consider, for each $m\in \NN$, the elements of~$G$ that admit a word in~$X$ of length~$m$ such that all initial segments have non-negative $\chi$-value:
	\[X^{m,\chi} := \{ x_1 \cdots x_m \in G \mid \forall i \le m\colon  x_i \in X \text{ and } \chi (x_1 \cdots x_i) \ge 0\}.\]
	This definition may also be inductively expressed as
	\[X^{0,\chi} = \{1\}, \qquad X^{m+1,\chi} = (X^{m,\chi} X) \cap G_\chi.\]
	For a word $w = x_1 \cdots x_m$, we call the sequence of values $\chi(x_1 \cdots x_i)$ (with $0 \le i \le m$) the \textbf{$\chi$-track} of~$w$. 
	
	Given a centered subset $X\subseteq G$, we thus obtain an ascending chain of subspaces of~$G_\chi$:
	\[X^{0,\chi} \subseteq X^{1,\chi} \subseteq X^{2,\chi} \subseteq \ldots\]
	Their union $\bigcup_{m \in \NN}X^{m,\chi}$, which consists of all elements of~$G_\chi$ represented by $X$-words of non-negative $\chi$-track, will be denoted by $X^{\infty,\chi}$.
	
	Recall from Section~\ref{sec:altfilt} that given a subset $X\subseteq G$, we denote by~$\dot X$ the smallest centered subset of~$G$ containing~$X$.
		
	\begin{dfn}
		Given a character $\chi \colon G\to \RR$ and a subset $X\subseteq G$, we say \textbf{$X$~generates $G$ along~$\chi$} if 
		\[\dot X^{\infty, \chi} = G_\chi.\]
	\end{dfn}
	
	When $\chi=0$, we recover the familiar notion of $X$~generating~$G$. It is also straightforward to see that if $X$~generates~$G$ along some~$\chi$, then $X$~is a generating set in the usual sense.
	
	We now describe this notion in a geometric fashion, namely through paths in $\Gamma(G, X)_\chi$.	
	Given two vertices $u,v$ in a simplicial set,  an \textbf{edge-path}, or simply a \textbf{path}, from $u$~to~$v$ shall mean an alternating sequence of vertices and edges $u=v_0, e_1,v_1, e_2,\dots, e_k, v_k=v$, such that the vertices of each edge~$e_i$ are $v_{i-1}$~and~$v_i$. An \textbf{oriented path} is one where each~$e_i$ is, moreover, oriented from~$v_{i-1}$ to~$v_i$.
	
	\begin{prop}[Words and paths]\label{prop.wordspaths}
		Let $X\subseteq G$, let $\chi\colon G \to \RR$ be a character, and let $g\in G_\chi$. The following conditions are equivalent:
		\begin{enumerate}
			\item $g \in \dot X^{\infty,\chi}$,
			\item $g$~can be represented by an $\dot X$-word of non-negative $\chi$-track,
			\item there is a path from~$1$ to~$g$ in $\Gamma(G,X)_\chi$,
			\item there is a path from $1$ to $g$ in $(G_\chi \cdot \E \dot X)_\chi$.
		\end{enumerate}
	\end{prop}
	
	\begin{proof}
		Observe first that the veracity of Condition~3 remains unaltered when $X$~is replaced with~$\dot X$, so we may assume $X = \dot X$.
		
		$(1 \Leftrightarrow 2)$ This equivalence is immediate from the definition of the~$X^{m,\chi}$.
		
		$(2 \Rightarrow 3)$ A word $x_1 \cdots x_m$ in~$X$ of non-negative $\chi$-track representing~$g$ specifies a path of length~$m$ in $\Gamma(G, X)_\chi$ from $1$ to~$g$, whose $j$-th edge is $x_1 \cdots x_{j-1}\cdot (1, x_j)$.
%
		
		$(3 \Rightarrow 4)$ We observe each edge~$(h,hx)$ of $\Gamma(G,X)_\chi$, with $h\in G_\chi, x\in X$,  can be written as $h\cdot (1,x)$. Since $1\in X$, this edge is in $(G_\chi \cdot \E X)_\chi$, and the statement follows.
		
		$(4 \Rightarrow  2)$ Since $X$~is centered, we may assume that the given path from $1$ to~$g$ is oriented. Each of its edges has the form $e_i = g_i\cdot (x_i, y_i)$ with $g_i\in G_\chi, x_i, y_i \in X$, so $e_i$~is labeled by $x_i^{-1}y_i$. Thus $g$~is represented by the word
		\[w=(x_1^{-1}y_1) (x_2^{-1}y_2) \cdots (x_m^{-1}y_m),\]
		and because the vertices of given path are in~$G_\chi$, it is clear that each initial sub-word of~$w$ of even length represents an element of~$G_\chi$. To conclude $w$~has non-negative $\chi$-track, we need to check this holds also for the initial words of odd length. But the $i$-th such word represents the element $(g_ix_i) x_i^{-1} = g_i \in G_\chi$.
	\end{proof}

	By quantifying the conditions in Proposition~\ref{prop.wordspaths} over all $g\in G_\chi$, we deduce:

	\begin{cor}[Generating along~$\chi$]\label{cor.passagetocayley}
		For every $X\subseteq G$ and character $\chi\colon G\to \RR$, the following conditions are equivalent:
		\begin{enumerate}
			\item X generates $G$ along~$\chi$,
			\item Every $g\in G_\chi$ can be represented by a word in~$\dot X$ of non-negative $\chi$-track,
			\item $\Gamma(G,X)_\chi$~is connected,
			\item $(G_\chi \cdot \E \dot X)_\chi$~is connected.
		\end{enumerate}
	\end{cor}

	\begin{prop}[Connectedness of $\Gamma(G,X)_\chi$]\label{prop.connectedcayley}
		Let $\chi\colon G \to \RR$ be a character and let $t\in G$ with $\chi(t)>0$. The following conditions on a nonempty $X \subseteq G$ are equivalent:
		\begin{enumerate}
			\item for every $s\in \RR$, the graph $\Gamma(G,X)_s$ is connected,
			\item for some $s \in \RR_{\le 0}$, the inclusion $\Gamma(G,X)_\chi \into \Gamma(G,X)_s$ is $\pi_0$-trivial,
			\item $X$~generates $G$ and the inclusion $[G,G] \into \Gamma(G,X)_\chi$ of the commutator subgroup is $\pi_0$-trivial.
			\item $X$~generates $G$ and $t^{-1} \dot X^{1,\chi} t \subseteq \dot X^{\infty,\chi}$.
		\end{enumerate}
	\end{prop}

\begin{proof}
	Since all conditions are equivalent to their counterparts with $X$~replaced by~$\dot X$, we may assume $X = \dot X$.
	
	
%
%

	$(1 \Rightarrow 4)$ By Corollary~\ref{cor.passagetocayley}, Condition 1 with $s=0$ means that $X$~generates~$G$ along~$\chi$, so in particular $X$~is a generating set for~$G$. Moreover, we see $t^{-1} X^{1,\chi} t \subseteq G_\chi = X^{\infty, \chi}$.
	
	$(4 \Rightarrow 2)$ Let $g\in G_\chi$, which we wish to show can be represented by an $X$-word of non-negative $\chi$-track. Since $X$~generates~$G$, let $w=x_1 \cdots x_l$ be an $X$-word representing~$g$, and let $s$~be the minimum of its $\chi$-track.
	
	Condition~4 allows us to choose, for each $x_i$~with $\chi(x_i)\ge 0$, an $X$-word~$w_i$ of non-negative $\chi$-track representing $t^{-1} x_i t$. If $\chi(x_i)<0$, then choose as $w_i$ the inverse of a word of non-negative $\chi$-track representing $t^{-1} x_i^{-1} t$ (so again $w_i$~represents~$t^{-1}x_it$). It then follows that the word $w' := t  w_1  \cdots  w_l t^{-1}$ represents~$g$; we claim that the minimum of the $\chi$-track of~$w'$ is $\min\{0, s+\chi(t)\}$. Once this assertion is established, the proof is complete, for one may repeat the procedure with $w'$ in place of~$w$, eventually reaching an $X$-word for~$g$ of non-negative $\chi$-track. This will show, by Corollary~\ref{cor.passagetocayley}, that $\Gamma(G, X)_\chi$ is connected, so Condition 2 holds with $s=0$.
	
	To establish this claim, the crucial observation is that for each word $w_i$, the minimum of its $\chi$-track is attained at one of its endpoints: indeed, if $\chi(x_i) \ge 0$, then $w_i$~has non-negative $\chi$-track, so its first initial sub-word, the empty word, having $\chi$-value $0$ represents this minimum. If $\chi(x_i) <0$, then the fact that the $\chi$-track of~$w_i^{-1}$ attains its minimum at the empty initial sub-word implies that the $\chi$-track of~$w_i$ attains its minimum at the last initial sub-word, which is $w_i$~itself.
	
	Hence, the minimum of the $\chi$-track of~$w'$ is attained at some of the initial sub-words 
	\[\emptyset, \quad t w_1,\quad t w_1  w_2,\quad \dots, \quad t  w_1  \cdots  w_{l-1}, \quad t  w_1  \cdots  w_l  t^{-1}.\] But these words represent the elements 
	\[1,\quad x_1 t,\quad  x_1 x_2 t, \quad \ldots, \quad x_1\cdots x_{l-1} t, \quad x_1\cdots x_l =g.\]
	Therefore, for each point in the $\chi$-track of~$w$, the corresponding $\chi$-value for~$w'$ exceeds it by~$\chi(t)$ or equals~$0$.
	
	$(2 \Rightarrow 3)$  $X$~generates $G$ because all elements of $G_\chi$ are connected to~$1$ in $\Gamma(G,X)_s$.
	
	Let $k\in \NN$ with $\chi(t^k) \ge -s$, and fix $g\in [G,G]$, which we aim to show is connected to~$1$ in $\Gamma(G,X)_\chi$ (actually, we will only use the fact that $\chi(g)=0$). By Condition~2, there is a path in~$\Gamma(G,X)_s$ from~$1$ to~$t^{-k} g t^{k}$, say specified by the $X$-word~$w$. Then the word~$t^k  w  t^{-k}$ specifies a path from~$1$ to~$g$ whose $\chi$-track is non-negative.
	
	$(3 \Rightarrow 1)$   Let $g$ be a vertex of $\Gamma(G,X)_s$, which we will show is connected to~$1$. Since $X$~generates $G$, the Cayley graph $\Gamma(G, X)$ is connected and contains an oriented path from $1$ to $g$, which specifies a word~$w$ in~$X$ representing $g$. Let $w'$ be some word obtained from~$w$ by re-ordering the letters so that the ones with non-negative $\chi$-value appear first. Then $w'$ has $\chi$-track above~$s$ and for the element~$g'$ represented by~$w'$, we have $h:=g(g')^{-1} \in [G,G]$. Now the assumption that $[G,G]\into \Gamma(G,X)_\chi$ is $\pi_0$-trivial provides a word~$u$ in~$X$ with non-negative $\chi$-track representing~$h$. Since $\chi(h)=0$, the concatenated word $u w'$ has $\chi$-track above~$s$, and it represents $h g' = g$, providing a path in $\Gamma(G,X)_s$ from~$1$ to~$g$.
\end{proof}

We now focus on the case where the generating set is compact. Namely, if there is $C\in \C(G)$ that generates~$G$ along~$\chi$, we shall say that $G$~is \textbf{compactly generated along~$\chi$}.
This condition is a rephrasing of Kochloukova's definition of ``$\chi \in \TopS^1(G)$'' \cite[p.541]{Ko04}. The equivalence to our definition of~$\TopS^1(G)$ is contained in the following theorem.
	
\begin{thm}[$\TopS^1$ and compact generation]\label{thm:sigma1_cptgen}
	For every character $\chi\colon G\to \RR$, the following conditions are equivalent:
	\begin{enumerate}
		\item $\chi \in \TopS^1(G)$,
		\item $G$~is compactly generated along~$\chi$,
		\item $G$~is compactly generated and every compact generating set generates~$G$ along~$\chi$.
	\end{enumerate}
\end{thm}

Note that for $\chi=0$, using Proposition~\ref{prop.zerocharacter}, we recover the fact that $G$~is of type~$\mathrm C_1$ if and only if $G$~is compactly generated.

\begin{proof}
	(1 $\Leftrightarrow$ 2) By Proposition~\ref{prop.sigma1stable}, we know $\chi \in  \TopS^1(G)$ if and only if $(G_\chi \cdot \E C)_\chi$~is connected for some compact centered $C\in \C(G)$. By Corollary~\ref{cor.passagetocayley}, this means $C$~generates~$G$ along~$\chi$.
	
	(3 $\Rightarrow$ 2) This implication is trivial.
	
	(2 $\Rightarrow$ 3) Let $C$~be a compact centered generating set with $G_\chi=C^{\infty, \chi}$, and let $D$~be a second compact centered generating set. We wish to show that $G_\chi \subseteq D^{\infty, \chi}$. By Lemma~\ref{lem.compactexhaustion}, we have $C \subseteq D^n$ for some $n\in \NN$. Since $\chi (D^n)$~is compact, it has a minimum $z \le 0$; in particular, every element of~$C$ can be expressed as a word in~$D$ with $\chi$-track above~$z$ (and length~$n$).
	
	Now let $g\in G_\chi$, and choose $t\in D$ and $k \in \NN$ with $\chi(t^k)\ge -z$. The element $t^{-k} g t^k$ is in~$G_\chi$, and thus in some $C^{m, \chi}$. Writing $t^{-k} g t^k$ as a word in~$C$ of non-negative $\chi$-track and replacing each letter with a word in~$D$ of $\chi$-track above~$z$, we obtain a word~$w$ in~$D$ representing $t^{-k} g t^k$ with $\chi$-track also above~$z$. Hence the word $t^k w t^{-k}$ in~$D$ represents~$g$ and has non-negative $\chi$-track, showing $g\in D^{nm +2k, \chi}$.
\end{proof}

Given the descriptions of ``generation along~$\chi$'' provided by Proposition~\ref{prop.connectedcayley}, we also have characterizations of~$\TopS^1$ in terms of Cayley graphs:

\begin{prop}[$\TopS^1$-criteria]\label{prop:sigma1_crit}
	  Let $\chi\colon G \to \RR$ be a character and let $t\in G$ with $\chi(t)> 0$. If $\mathcal{P}$ is one of the properties of a subset of~$G$ listed in Proposition~\ref{prop.connectedcayley}, the following conditions are equivalent:
	  	\begin{enumerate}
	  	\item $\chi \in \TopS^1(G)$,
	  	\item Some $C\in \mathcal C(G)_{\supseteq\{1\}}$ satisfies $\mathcal P$.
	  	\item $G$~is compactly generated and every compact generating set satisfies $\mathcal P$.
	  \end{enumerate}
\end{prop}

We finish this section by collecting a pair of results that will be used in the next one.

\begin{lem}[Compactness of $C^{m, \chi}$]
	For every centered $C\in \C(G)$ and $m\in \NN$, the set $C^{m, \chi}$ is compact.
\end{lem}

\begin{proof}
	$C^{0, \chi} = \{1\}$ is compact. If we assume $C^{m, \chi}$~is compact, it follows that $C^{m, \chi} C$, being the image of the map $C^{m, \chi} \times C \to G, (g,h) \mapsto gh$, is compact.  Then $C^{m+1, \chi}$ is the preimage of~$\RR_{\ge 0}$ under the restriction~$\chi|_{C^{m+1, \chi} C}$. Therefore it is closed in~$C^{m, \chi} C$, and hence compact.
\end{proof}

The following lemma generalizes Lemma~\ref{lem.compactexhaustion}.

\begin{lem}[Compact exhaustion along~$\chi$]\label{lem.compactexhaustionchi}
	Let $\chi\colon G \to \RR$ be a character, let $C \in \C(G)$ be a centered subset generating~$G$ along~$\chi$, and let $D\subseteq G_\chi$ be compact. Then for some $m\in \NN$ we have $D \subseteq \Int_{G_\chi}(C^{m, \chi})$.  
\end{lem}
\begin{proof}
	Since $G_\chi$ is locally compact and Hausdorff,
	Baire's Category Theorem tells us it is a Baire space  \cite[Chapter~IX, Theorem~1]{Bou98}. This means $G_\chi$ is not a countable union of closed subsets with empty interior. Hence, for some $k \in \NN$, the set~$C^{k, \chi}$ has nonempty interior relative to~$G_\chi$, that is, there is a nonempty subset $W\subseteq C^{k, \chi}$ that is open in~$G_\chi$.
	
	We wish to modify~$W$ to a nonempty subset~$U \subseteq C^{k, \chi}$ that is open \emph{as a subset of~$G$}. To this end, we write $W = W_0 \cap G_\chi$, where $W_0$~is an open subset of~$G$, and consider two cases:
	\begin{itemize}
		\item If $W_0$ contains an element of strictly positive $\chi$-value, we take $U:= W_0 \cap \chi^{-1}(\RR_{>0})$. In this case it is clear that $U$ is nonempty and open in~$G$, and we also see that $U\subseteq W_0 \cap G_\chi = W \subseteq C^{k, \chi}$.
		\item If $\chi(W_0) \subseteq \RR_{\le0}$, then as $W$~is nonempty, we see $W_0$~contains an element $w$~of~$\ker \chi$. We put $U:= W_0 \cap wW_0^{-1}w$. Again it is clear that $U$ is open in~$G$, and it is also easy to see that $w\in U$, so $U$~is nonempty. Moreover, we have $\chi(w W_0^{-1}w) \subseteq \RR_{\ge0}$, so $\chi(U) = \{0\}$ and thus $U \subseteq W_0 \cap G_\chi =W \subseteq C^{k, \chi}$. 
	\end{itemize}
	
	Armed with this set~$U$, our next goal is to show that for every~$l\in \NN$, we have
	\[C^{l, \chi} \subseteq \Int_{G_\chi} (C^{2k+l, \chi}).\]
	Once this inclusion is established, the result can be obtained easily: together with the assumption that $G_\chi = \bigcup_{l\in \NN}C^{l, \chi}$, we see $G_\chi = \bigcup_{m\in \NN} \Int_{G_\chi} (C^{m, \chi})$. As this is an exhaustion by an ascending sequence of open subsets of~$G_\chi$, compactness of~$D$ implies that $D$~is contained in some $\Int_{G_\chi} (C^{m, \chi})$.
	
	To prove the above inclusion, choose any $u\in U$. We have $C^{l, \chi} \subseteq (C^{l, \chi} Uu^{-1}) \cap G_\chi$, so it suffices to show that $(C^{l, \chi} Uu^{-1}) \cap G_\chi \subseteq \Int_{G_\chi} (C^{2k+l, \chi})$.
	But $C^{l, \chi} Uu^{-1}$ is open in~$G$, so $(C^{l, \chi} Uu^{-1})\cap G_\chi$ is open in~$G_\chi$, and thus it suffices that we verify $(C_\chi^l Uu^{-1}) \cap G_\chi\subseteq C_\chi^{2k+l}$.
	Now, as $U\subseteq C^{k, \chi}$, we see $C^{l, \chi} U \subseteq C_\chi^{l+k,\chi}$.
	Each element of $(C^{l, \chi} Uu^{-1}) \cap G_\chi$ is thus a product of the form $yu^{-1}$ where $y$~is represented by a word~$w_y$ in~$C$ of length $\le k+l$ with non-negative $\chi$-track, and $\chi(y) \ge \chi(u)$.
	Moreover, $u\in C^{k, \chi}$, so consider a word~$w_u$ in~$C$ for~$u$ of length $\le k$ and non-negative $\chi$-track. The inverse word~$w_u^{-1}$ represents~$u^{-1}$ and has $\chi$-track above $-\chi(u)$, so the word $w_y  w_u^{-1}$ has non-negative $\chi$-track. This exhibits~$yu^{-1}$ as an element of~$C^{2k+l,\chi}$.
\end{proof}

\section{Compact presentability and {$\TopS^2$}}\label{sec:sigma2}

We begin by introducing some notation convenient for studying group presentations.
	Given a generating set $X \subseteq G$,
	we define $\bar X$ to be a disjoint copy of~$X$, whose elements are to be thought of as formal inverses to the elements in~$X$. For each $x\in X$, the corresponding element in $\bar X$ is denoted by~$\bar x$, and we also write $\bar{\bar x}:=x$. Every element of~$G$ can thus be represented by a word in the disjoint union $X^{\pm} := X\sqcup \bar X$, where each $\bar x \in \bar X$ is to be read as~$x^{-1}$. Given an $X^{\pm}$-word~$w=y_1\cdots y_k$, we write $\bar w := \bar y_k\cdots \bar y_1$. Two $X^{\pm}$-words are \textbf{freely equivalent} if they differ by a sequence of cancellations and insertions of syllables $y\bar y$, with $y\in X^{\pm}$, and these moves are called \textbf{elementary deletions / insertions}, respectively.
	 If an $X^{\pm}$-word represents~$1$, we call it an  \textbf{$X$-relation}. For a set of $X$-relations~$R$, we write $\dot R := R \cup \{\bar w \mid w\in R\}$.
	
	Given a character $\chi \colon G \to \RR$, we say that an $X$-relation $\rho$~of non-negative $\chi$-track \textbf{follows from~$R$ along~$\chi$} if $\rho$~is freely equivalent to a relation of the form $w_1 r_1 \bar w_1 \cdots w_l r_l \bar w_l$ with non-negative $\chi$-track, where $r_1, \dots, r_l \in \dot R$ and $w_1, \dots, w_l$ are $X^{\pm}$-words. Then the~$w_i$ necessarily have non-negative $\chi$-track. Sometimes, it is also convenient for the~$r_i$ to have non-negative $\chi$-track, which may often be attained with the following trick:
	
	\begin{lem}[Non-negative relators]\label{lem:posrels}
		If $\rho$ follows from $R$ along~$\chi$, then $\rho$~is freely equivalent to a relation of the form $u_1 r'_1 \bar u_1 \cdots u_l r'_l \bar u_l$ with non-negative $\chi$-track, such that all words $r'_1, \dots, r'_l$  are cyclic permutations of relators in~$\dot R$ and have non-negative $\chi$-track.
	\end{lem}
	
	\begin{proof}
		We know $\rho$~is freely equivalent to a relation $w_1 r_1 \bar w_1 \cdots w_l r_l \bar w_l$ of non-negative $\chi$-track and with $r_i\in \dot R$.
		For each $i\in\{1, \dots, l\}$, choose a prefix~$v_i$ of~$r_i$ for which $\chi(v_i)$ is the minimum of its $\chi$-track. In particular, since the empty prefix has $\chi$-value~$0$, we have $\chi(v_i)\le 0$.
		We consider the relator~$r_i'$ obtained from~$r_i$ by moving~$v_i$ to the end of~$r_i$, and claim that $r_i'$~has non-negative $\chi$-track. Indeed, writing $r_i = y_1 \cdots y_k$ and $v_i= y_1 \cdots y_m$ with $m\le k$, the $\chi$-track of $r_i'$ is the sequence
		\[0, \chi (y_{m+1}), \ldots, \chi (y_{m+1} \cdots y_k), \chi (y_{m+1} \cdots y_k y_1),\ldots, \chi(y_{m+1} \cdots y_k y_1\cdots y_m).\]
		Minimality of $\chi(v_i)$ yields
		\begin{align*}
			\chi(y_{m+1}) & = \chi(y_1 \cdots y_{m+1}) - \chi (v_i) \ge 0,\\
			\dots\\
			\chi(y_{m+1} \cdots y_k) &= \chi(y_1 \cdots y_k) - \chi (v_i) = -\chi(v_i) \ge 0,\\
			\chi(y_{m+1} \cdots y_k y_1) &= -\chi(v_i) + \chi (y_1) \ge 0,\\
			\dots\\
			\chi(y_{m+1} \cdots y_k y_1\cdots y_m) &= -\chi(v_i) + \chi (y_1\cdots y_m) = 0.
		\end{align*}
		
		Next, observe that since $\rho$~is assumed to have non-negative $\chi$-track, and the sub-word $w_1 r_1 \bar w_1 \ldots w_{i-1} r_{i-1} \bar w_{i-1}$ represents the trivial element (which has $\chi$-value $0$), the word $u_i:=w_i v_i$ also has non-negative $\chi$-track. It is straightforward to see that the word $u_i r_i' \bar u_i$ is freely equivalent to $w_i r_i \bar w_i$, and so the product $u_1 r'_1 \bar u_1 \cdots u_l r'_l \bar u_l$ is as required.
	\end{proof}
	
	We say $\langle X \mid R\rangle$~is a \textbf{presentation of $G$ along~$\chi$} if $X$~generates~$G$ along~$\chi$ and all $X$-relations in~$G$ of non-negative $\chi$-track follow from~$R$ along~$\chi$. Since every $X$-relation is conjugate to one of non-negative $\chi$-track, this implies in particular that $\langle X \mid R \rangle$~is a presentation of~$G$. For $\chi = 0$ we recover the usual notion of a presentation.

Observe that in the condition of $\rho$~being freely equivalent to $w_1 r_1 \bar w_1 \cdots w_l r_l \bar w_l$, we do not explicitly demand that the equivalence be through relations of non-negative $\chi$-track. However, the existence of such an equivalence follows:

\begin{lem}[Free equivalence of non-negative words]\label{lem.positiveequiv}
	Let $\chi\colon G \to \RR$ be a character, let $X\subseteq G$, and let $\rho,\rho'$ be $X^{\pm}$-words of non-negative $\chi$-track. If $\rho$~is freely equivalent to~$\rho'$, then there is a sequence of elementary insertions and deletions from~$\rho$ to~$\rho'$ such that each intermediate word has non-negative $\chi$-track.
\end{lem}

\begin{proof}
	Let $\rho_0$ be the common free reduction of $\rho$ and $\rho'$, that is, $\rho_0$~is the unique word obtained from $\rho$ by performing elementary deletions until no other is possible. An elementary deletion on a word of non-negative $\chi$-track yields again a word of non-negative $\chi$-track, so all words from~$\rho$ to $\rho_0$ have non-negative $\chi$-track. By the same argument, we connect~$\rho'$ to $\rho_0$ through words of non-negative $\chi$-track. Then the sequence from $\rho$, to $\rho_0$, to $\rho'$ is comprised of words of non-negative $\chi$-track.
\end{proof}

Given $X\subseteq G$ and an $X^{\pm}$-word~$w = y_1 \cdots y_k$, an edge-path $(g_0, e_1, g_1, \dots, e_k, g_k)$ in $\Gamma (G,X)$ \textbf{reads}~$w$ if for every $i\in \{1,\dots, k\}$, we have:
\begin{itemize}
	\item if $y_i\in X$, then $e_i =(g_{i-1}, g_i)$ and $g_{i-1} y_i = g_{i}$,
	\item if $y_i\in \bar X$, then $e_i =(g_i, g_{i-1})$ and $g_i \bar y_i = g_{i-1}$.
\end{itemize}
Every path determines an $X^{\pm}$-word, and conversely, specifying a starting vertex $g_0\in G$ and an $X^{\pm}$-word~$w$ determines a path reading~$w$.

Let $X\subseteq G$ and let $R$~be a set of $X$-relators. The \textbf{Cayley complex} $\Gamma(G, X\st R)$ is the $2$-dimensional CW complex obtained from the geometric realization~$|\Gamma(G,X)|$ of the Cayley graph by attaching a $2$-cell along every path reading a relator in~$R$. 

Given a character $\chi \colon G\to \RR$ and $s \in \RR$, we denote by $\Gamma(G, X\st R)_s$ the sub-complex spanned by the vertices~$g$ with $\chi(g) \ge s$. We also write $\Gamma(G, X\mid R)_\chi := \Gamma(G, X\mid R)_0$.

\begin{lem}[Relations of non-negative $\chi$-track]\label{lem.wordsincayleycomplex}
	Let $\chi\colon G \to \RR$ be a character, let $X\subseteq G$ generate~$G$ along~$\chi$, and let $R$~be a set of $X$-relations. Then an $X$-relation~$\rho$ of non-negative $\chi$-track follows from~$R$ along~$\chi$ if and only if the edge-loop~$\gamma$ in~$|\Gamma(G,X)_\chi|$ based at~$1$ reading~$\rho$ is null-homotopic in~$\Gamma(G, X \mid R)_\chi$. 
\end{lem}
\begin{proof}
	($\Rightarrow$) We know $\rho$ is freely equivalent to a word $w=w_1 r_1 \bar w_1 \cdots w_l r_l \bar w_l$ of non-negative $\chi$-track, with $r_i\in \dot R$. By Lemma~\ref{lem.positiveequiv}, this free equivalence can be realized through words of non-negative $\chi$-track, which provide a homotopy in $|\Gamma(G,X)_\chi|$ from~$\gamma$ to the loop reading~$w$. Hence we need only show that this loop is null-homotopic in $\Gamma(G, X \mid R)_\chi$. But it is a sequence of ``whiskered loops'' given by the sub-words $w_i r_i \bar w_i$, where the $r_i$-portion corresponds to a loop along which a $2$-cell is attached. Thus each whiskered loop is null-homotopic, and so is their concatenation.
	
	($\Leftarrow$) By Seifert-van Kampen's Theorem, the fact that~$\gamma$ is null-homotopic in~$\Gamma(G,X\mid R)_\chi$ means that it represents an element of $\pi_1(\Gamma(G,X)_\chi, 1)$ that lies in the normal closure of the set of loops giving the attaching maps of the $2$-cells, each suitably connected to the base point~$1$ through some edge-path. These whiskered loops read $X^{\pm}$-words of the form $w_r r \bar w_r$, with non-negative $\chi$-track and $r \in R$. Thus, $\gamma$~is pointed-homotopic in $\Gamma(G,X)_\chi$ to a loop reading a product of certain conjugates of these words, or their inverses.
	Absorbing the conjugators into the~$w_r$ yields a word $w=w_1 r_1 \bar w_1 \cdots w_l r_l \bar w_l$ of non-negative $\chi$-track and $r_i \in \dot R$. Now, as the path at~$1$ reading~$w$ is pointed-homotopic to~$\gamma$, both these paths have lifts with the same endpoint in the universal cover, which is the cut-off Cayley graph of the free group~$\Gamma(\mathrm F(X), X)_\chi$. This implies $\rho$~is freely equivalent to~$w$.
\end{proof}

\begin{cor}[$1$-connectedness of $\Gamma(G,X\mid R)_\chi$]\label{cor:cptpres_1conn}
	Let $\chi\colon G \to \RR$ be a character, let $X\subseteq G$ be a generating set, and let $R$~be a set of $X$-relations. Then $\langle X \mid R \rangle$ is a presentation of~$G$ along~$\chi$ if and only if $\Gamma(G,X\mid R)_\chi$~is simply connected.
\end{cor}
\begin{proof}
	Since the $1$-skeleton of $\Gamma(G,X\mid R)_\chi$ is $\Gamma(G, X)_\chi$, 
	we see from Corollary~\ref{cor.passagetocayley} that $\langle X \mid R\rangle$ is a presentation of~$G$ along~$\chi$ if and only if $\Gamma(G, X \mid R)_\chi$ is connected and all $X$-words of non-negative $\chi$-track follow from~$R$ along~$\chi$.
	By Lemma~\ref{lem.wordsincayleycomplex}, this is equivalent to $\Gamma(G, X \mid R)_\chi$ being connected and the inclusion $\Gamma(G, X)_\chi \into \Gamma(G, X\mid R)_\chi$ being~$\pi_1$-trivial. But since the inclusion of the $1$-skeleton is always $\pi_1$-surjective, these conditions combined mean precisely that $\Gamma(G,X\mid R)$ is simply connected.
\end{proof}

Given $m \in \NN$ and $X\subseteq G$, the set of $X$-relations of length at most~$m$ will be denoted by~$R_m^X$, or simply~$R_m$ if no confusion can arise. Observe that if $m\ge 2$, then $\langle X \mid R_m\rangle$ is a presentation of~$G$ along~$\chi$ if and only if $\langle \dot X \mid R_m\rangle$ is. This often allows one to simplify arguments by assuming generating sets are centered.

With the following proposition, we start connecting the characterizations of finite presentability along a character to the simplicial sets from our theory of $\Sigma$-sets. The description of $\TopS^2(G)$ that will be convenient to us will be in terms of the filtration $((G\chi \cdot \E C^m)_\chi)_{m\in \NN}$  from Corollary~\ref{cor.powersofC} (3). Recall that if $X$~is any centered generating set and $\chi$~is a character of~$G$, we have $\Gamma(G, X)_\chi \subseteq (G_\chi\cdot \E X)_\chi$.

\begin{prop}[Null-homotopic loops]\label{prop.trivialloops}
	Let $\chi\colon G \to \RR$ be a character, let $X\subseteq G$~be a centered generating set, and $\rho$ an $X$-relation of non-negative $\chi$-track. Fix also $m \in \NN$. Then:
	\begin{enumerate}
		\item Suppose that the edge-path in $\Gamma(G,X)_\chi$ starting at~$1$ and reading $\rho$ is null-homotopic in $(G_\chi \cdot \E X^m)_\chi$. Moreover, suppose $k \in \NN_{\ge1}$ is such that $X^{2m} \cap G_\chi \subseteq X^{k, \chi}$. Then $\rho$~follows from $R_{3k}$ along~$\chi$.
		
		\item If $\rho$~follows from $R_m$ along~$\chi$, then every edge-path in~$\Gamma(G,X)_\chi$ reading~$\rho$ is null-homotopic in~$(G_\chi \cdot \E X^{\lfloor\tfrac m2\rfloor})_\chi$.
	\end{enumerate}
\end{prop}

\begin{proof}
	(1)
	We first represent the homotopy class of the given path in~$\Gamma(G, X)_\chi$ by a simplicial map $\eta \colon \SD^{n_0}(\partial \Delta^2) \to \Gamma(G,X)_\chi$, for some $n_0 \in \NN$, where~$\SD$ is the subdivision operator from Section~\ref{sec:prelim}.
	By sequentially reading the nondegenerate edges in the image of~$\eta$, we obtain a path reading an $X^{\pm}$-word~$\rho'$ that is freely equivalent to~$\rho$.
	 
	By Lemma~\ref{lem:combfilling}, there are $n\ge n_0$ and a map $\mu$ fitting into a commutative square
	\[\begin{tikzcd}
		\SD^{n}(\partial \Delta^2) \arrow[r,"\eta \circ \Phi^{n-n_0}"] \arrow[d, hook]& \Gamma(G,X)_\chi \arrow[d,hook]\\
		\SD^{n} \arrow[r,"\mu"](\Delta^2) & (G_\chi \cdot \E X^m)_\chi
	\end{tikzcd},\]
	and the upper horizontal map still defines an edge-path reading the word~$\rho'$.
	
	Now, abbreviate $\mathcal D:= \SD^n(\Delta^2)$, $\mathcal S := \SD^n(\partial \Delta^2)$, and note that for each nondegenerate edge~$e$ of~$\mathcal D$ that is not in~$\mathcal S$, its image~$\mu(e)$ is labeled by an element $y\in X^{2m}$ (because $X$~is centered). If $\chi(y) \ge 0$, so $y\in G_\chi \cap X^{2m} \subseteq X^{k,\chi}$, choose an $X$-word~$w_y$ of non-negative $\chi$-track representing~$y$, and label $e$ with~$w_y$. In case $\chi(y)<0$, put instead $w_y= w^{-1}$, where $w$~is an $X$-word of non-negative $\chi$-track for $y^{-1}$.

	The upshot of this definition is that for each such edge $e = (g_0,g_1)$, the path in $\Gamma(G,X)$ starting at~$g_0$ and reading~$w_y$ stays within $\Gamma(G,X)_\chi$, and so these labels define a map $|\mathcal D^{(1)}|\to |\Gamma(G,X)_\chi|$. Moreover, the boundary of each nondegenerate triangle of~$\mathcal D$ reads an $X$-relation of length at most~$3k$, so the preceding map extends to a map $|\mathcal D| \to \Gamma(G, X\mid R_{3k})_\chi$, which exhibits the loop reading $\rho'$ as null-homotopic; see Figure~\ref{fig:trivialloops_a}. By Lemma~\ref{lem.wordsincayleycomplex}, we conclude $\rho'$ follows from $R_{3k}$ along~$\chi$, and hence so does~$\rho$.
	
	\begin{figure}[h]
		\centering
		\def \svgwidth{0.6\linewidth}
\begingroup%
  \makeatletter%
  \providecommand\color[2][]{%
    \errmessage{(Inkscape) Color is used for the text in Inkscape, but the package 'color.sty' is not loaded}%
    \renewcommand\color[2][]{}%
  }%
  \providecommand\transparent[1]{%
    \errmessage{(Inkscape) Transparency is used (non-zero) for the text in Inkscape, but the package 'transparent.sty' is not loaded}%
    \renewcommand\transparent[1]{}%
  }%
  \providecommand\rotatebox[2]{#2}%
  \newcommand*\fsize{\dimexpr\f@size pt\relax}%
  \newcommand*\lineheight[1]{\fontsize{\fsize}{#1\fsize}\selectfont}%
  \ifx\svgwidth\undefined%
    \setlength{\unitlength}{444.38356415bp}%
    \ifx\svgscale\undefined%
      \relax%
    \else%
      \setlength{\unitlength}{\unitlength * \real{\svgscale}}%
    \fi%
  \else%
    \setlength{\unitlength}{\svgwidth}%
  \fi%
  \global\let\svgwidth\undefined%
  \global\let\svgscale\undefined%
  \makeatother%
  \begin{picture}(1,0.42016865)%
    \lineheight{1}%
    \setlength\tabcolsep{0pt}%
    \put(0,0){\includegraphics[width=\unitlength,page=1]{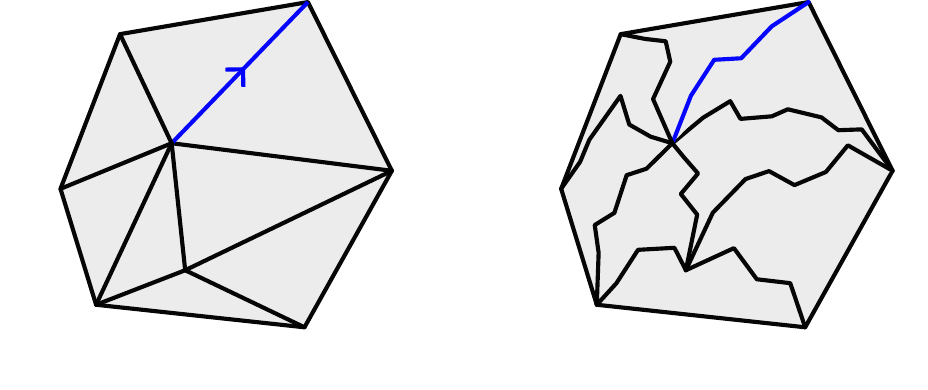}}%
    \put(0.50775879,0.2261261){\color[rgb]{0,0,0}\makebox(0,0)[t]{\lineheight{0.69999999}\smash{\begin{tabular}[t]{c}$\rightsquigarrow$\end{tabular}}}}%
    \put(0.30440954,0.31900467){\color[rgb]{0,0,1}\makebox(0,0)[t]{\lineheight{0.69999999}\smash{\begin{tabular}[t]{c}$y$\end{tabular}}}}%
    \put(0.83435693,0.32238014){\color[rgb]{0,0,1}\makebox(0,0)[t]{\lineheight{0.69999999}\smash{\begin{tabular}[t]{c}$w_y$\end{tabular}}}}%
    \put(0.25164977,0.00560379){\color[rgb]{0,0,0}\makebox(0,0)[t]{\lineheight{0.69999999}\smash{\begin{tabular}[t]{c}$\mathcal D \xrightarrow{\mu} (G_\chi\cdot \E X^m)_\chi$\end{tabular}}}}%
    \put(0.79172364,0.00560379){\color[rgb]{0,0,0}\makebox(0,0)[t]{\lineheight{0.69999999}\smash{\begin{tabular}[t]{c}$|\mathcal D| \to \Gamma(G,X\mid R_{3k})_\chi$\end{tabular}}}}%
    \put(0,0){\includegraphics[width=\unitlength,page=2]{trivialloops_a.pdf}}%
  \end{picture}%
\endgroup%

		\caption{Converting a disk in $(G_\chi \cdot \E X^m)_\chi$ into a disk in $\Gamma(G,X \mid R_{3k})_\chi$. Triangles on the left give rise to disks whose boundary read relators in $R_{3k}$.}
		\label{fig:trivialloops_a}
	\end{figure}
	
	(2)
	The assumption on~$\rho$ tells us it is freely equivalent to a word of non-negative $\chi$-track $\rho'= w_1 r_1  \bar w_1 \cdots  w_l r_l \bar w_l$ with $r_i \in R_m$. By Lemma~\ref{lem.positiveequiv}, the paths at~$1$ reading $\rho$ and~$\rho'$ are pointed-homotopic within~$\Gamma(G,X)_\chi$. Therefore, once we show that for every $r\in R_m$, edge loops in $\Gamma(G, X)_\chi$ reading~$r$ are null-homotopic in $(G_\chi \cdot \E X^{\lfloor\tfrac m2\rfloor})_\chi$, it will follow that $\rho'$ too is null-homotopic, and thus also $\rho$.
	
	One may construct a filling for a loop reading $r\in R_m$ (say, starting at $g_0\in G_\chi$) rather explicitly: if $r$~has length~$l\le m$, consider the $1$-dimensional simplicial set~$\mathcal S$ with $l$~vertices~$v_1,\dots, v_l$ and $l$~nondegenerate edges~$e_1, \dots, e_l$ arranged in a loop, with $e_i$ oriented according to whether the $i$-th symbol~$y_i$ of~$r$ comes from $X$ or $\bar X$. We claim that the morphism of simplicial sets $\mathcal S \to \Gamma(G, X)_\chi$ determined by~$r$ and~$g_0$ extends to a morphism $\mathcal D \to (G_\chi \cdot \E X^{\lfloor\tfrac m2\rfloor})_\chi$, where $\mathcal D$ is the cone of~$\mathcal S$, by sending the cone point to~$g_0$. This is illustrated in Figure~\ref{fig:trivialloops_b}. To see this, we observe that for every $i\in \{1, \dots, l\}$, the element~$g_i$ represented by the word $y_1\cdots y_i$ lies in~$X^{i}$, and also in $X^{m-i}$. Therefore, the triangles $g_0\cdot (1, g_{i-1}, g_{i})$ and $g_0\cdot (1, g_i, g_{i-1})$ are simplices of $(G_\chi \cdot \E X^{\lfloor\tfrac m2\rfloor})_\chi$, allowing us to map as desired, regardless of whether $y\in X$ or $y \in \bar X$.\qedhere
	
	\begin{figure}[h]
		\centering
		\def \svgwidth{0.4\linewidth}
		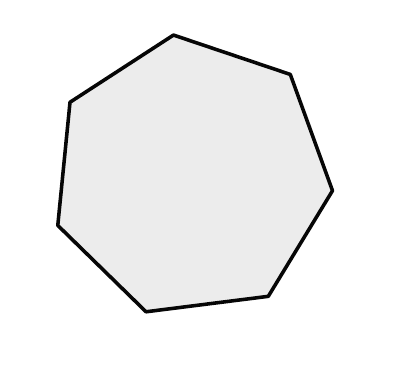
		\caption{The cone~$\mathcal D$ of~$\mathcal S$. Edges of~$\mathcal S$ get mapped to edges of $\Gamma(G,X)_\chi$ as dictated by~$r$. Edges emanating from the cone point get mapped to edges with labels in $X^{\lfloor \tfrac m2 \rfloor}$.}
		\label{fig:trivialloops_b}
	\end{figure}
\end{proof}

\begin{lem}[$\pi_1$-surjectivity]\label{lem.pi1onto}
    For every character $\chi \colon G\to \RR$ and every centered subset $X\subseteq G$,
	the inclusion $\Gamma(G,X)_\chi \into(G_\chi \cdot \E X)_\chi$ is $\pi_1$-surjective on every connected component.
\end{lem}

\begin{proof}
	It suffices to show that every loop in the $1$-skeleton of~$|(G_\chi \cdot \E X)_\chi|$ can be homotoped to a loop in~$|\Gamma(G,X)_\chi|$. In fact, we show that the geometric realization of every edge of~$(G_\chi \cdot \E X)_\chi$ can be homotoped relative endpoints to~$|\Gamma(G,X)_\chi|$.
	
	So let $e:=g\cdot (x_0,x_1)$ with $x_0, x_1 \in X$ and $g, gx_0, gx_1 \in G_\chi$. The triangle $g\cdot(1,x_0, x_1) \in (G_\chi \cdot \E X)_\chi$ shows that $|e|$~is homotopic relative endpoints to the path formed by $|g\cdot (1, x_0)|$ and $|g\cdot (1, x_1)|$. The edges $g\cdot (1, x_0)$ and $g\cdot (1,x_1)$, being labeled by $x_0$~and~$x_1$ respectively, lie in $\Gamma(G, X)_\chi$.
\end{proof}

Let us bring the topology of~$G$ into the picture.

\begin{dfn}\label{dfn:cpt_pres}
	A presentation $\langle C \mid R \rangle$ of~$G$ is \textbf{compact} if the generating set~$C$ is compact and there is $m\in \NN$ such that all relators in~$R$ have length at most~$m$. Given a character $\chi \colon G \to \RR$, we say $G$~is \textbf{compactly presented along $\chi$} if it has a compact presentation along $\chi$.
\end{dfn}

We now see that the condition ``$G$ is compactly presented along $\chi$'' is a rephrasing of Kochloukova's definition of ``$\chi \in \TopS^2(G)$''.
Given a subset $X\subseteq G$, denote by $\F(X)$ the free group on~$X$, and by $\pi\colon \F(X)\to G$ the map induced by the inclusion of~$X$. Moreover, for a character $\chi\colon G\to\RR$, denote by $\F_\chi(X) \subseteq \F(X)$ the submonoid of elements that are represented by $X^\pm$-words of non-negative $\chi$-track (not to be confused with the monoid $\F(X)_{\chi\circ\pi}$ of elements whose $\pi$-image is in~$G_\chi$!). We also consider the map of monoids $\pi_\chi \colon \F_\chi(X) \to G_\chi$. Then Kochloukova defines~$\TopS^2(G)$ to be the set of characters~$\chi$ for which the following condition holds \cite[p.~541]{Ko04}:
\emph{\begin{quote}
	There exist $C\subseteq G$ and $R\subseteq \ker(\pi_\chi)$ such that:
	\begin{enumerate}
		\item $C$ is a compact generating set for~$G$,
		\item the monoid homomorphism $\pi_\chi\colon \F_\chi(C) \to G_\chi$ is surjective,
		\item there is a bound on the length of the elements in $R$,
		\item the monoid $\ker(\pi_\chi)$ is generated by $\bigcup_{w\in \F_\chi(C)} wRw^{-1}$.
	\end{enumerate}
\end{quote}}

\begin{lem}[Kochloukova's definition of~$\TopS^2$]\label{lem:Kochdef}
	The character~$\chi$ lies in Kochloukova's $\TopS^2(G)$ if and only if $G$~is compactly presented along~$\chi$.
\end{lem}
\begin{proof}
	In our usual terminology, Conditions 1~and~2 combined state that $C$~is a compact generating for~$G$ along~$\chi$. The monoid $\ker(\pi_\chi)$ is the set of reduced words that are $C$-relations of non-negative $\chi$-track, so Condition 4 says that every reduced $C$-relation of non-negative $\chi$-track follows from~$R$ along~$\chi$ (and hence so do the unreduced ones). Thus, if $C, R$ satisfy Conditions 1 to~4, then $\langle C \mid R \rangle$ is a compact presentation of~$G$ along~$\chi$.
	
	For the converse direction, we need only take care of the additional detail that our definition of a compact presentation along~$\chi$ does not require the set of relators~$R$ to have non-negative $\chi$-track, which Kochloukova's does by imposing $R\subseteq \ker (\pi_\chi)$ (rather than merely $R\subseteq \ker(\pi))$. But if we have a compact presentation $\langle C \mid R \rangle$ of~$G$ along~$\chi$, we can use Lemma~\ref{lem:posrels} to modify the relators in~$R$ to have non-negative $\chi$-track.
\end{proof}

Our next main goal is to show that compact presentability along~$\chi$ is also equivalent to our defining condition of $\chi$~lying in~$\TopS^2(G)$.

 We begin by establishing a proposition whose case ``$\chi=0$'' contains the ``folklore'' fact that for every finite generating set~$X$ of a finitely presented group, there is a set~$R$ of relators such that $\langle X \mid R \rangle$ is a finite presentation. It also generalizes a lemma of Kochloukova \cite[Lemma~1]{Ko04}.

\begin{prop}[Adding relators]\label{prop.anygensetchi}
	Let $\chi\colon G \to \RR$ be a character and let $G$~be compactly presented along~$\chi$. Then for every compact generating set $C\subseteq G$, there is $m\in \NN$ such that $\langle C \mid R_m\rangle$ is a compact presentation of~$G$ along~$\chi$.
\end{prop}

\begin{proof}
	Suppose $\langle D\mid R_n \rangle$~is a compact presentation of~$G$ along~$\chi$. Assume, as we may, that both $C$~and~$D$ are centered. By Theorem~\ref{thm:sigma1_cptgen}, the fact that $D$~generates~$G$ along~$\chi$ implies that so does~$C$. We are left to prove that there exists~$m\in \NN$ such all $C^\pm$-words of non-negative $\chi$-track follow from~$R_m^C$ along~$\chi$.
	
	Note that it suffices to prove this assertion under the assumption that $D\subseteq C$ or $C\subseteq D$: indeed, once established under each of these conditions, we can prove it for every~$C$ by first showing it for the compact generating set $C\cup D$. So let us address each case separately.
	
	($D \subseteq C$) Since $\chi\in \TopS^1(G)$, Lemma~\ref{lem.compactexhaustionchi} tells us that there exists $k \in \NN$ such that $C_\chi \subseteq D^{k,\chi}$, that is, every $x\in C_\chi$ can be expressed as a $D$-word~$w_x$ of non-negative $\chi$-track with length at most~$k$. For $x \in C$ with $\chi(x)<0$, we put $w_x := (w_{x^{-1}})^{-1}$, and for $y\in \bar C$, we define $w_y = \bar w_{\bar y}$. Thus, for all $y\in C^\pm$, the $D$-word~$w_y$ represents~$y$, has length at most~$k$, and the minimum of its $\chi$-track is attained at one of its endpoints.
	
	We claim that by choosing $m := \max\{n, k+1\}$, it follows that $\langle C \mid R_m \rangle$ is a presentation of~$G$ along~$\chi$.
	Consider, for each $y\in C^{\pm}$, the relator $r_y:=y \bar w_y$ in $R_{k+1}^{C\cup D} \subseteq R_m^C$.
	If $\rho = y_1\cdots y_l$ is a $C$-relator of non-negative $\chi$-track, then $\rho$~is freely equivalent to
	\[r_{y_1}  (w_{y_1} r_{y_2} \bar w_{y_1}) (w_{y_1} w_{y_2}  r_{y_3} \bar w_{y_2} \bar w_{y_1}) \cdots  (w_{y_1} \cdots  w_{y_{l-1}} r_{y_l} \bar w_{y_{l-1}} \cdots  \bar w_{y_1}) (w_{y_1} \cdots  w_{y_l}).\]
	To show that $\rho$ follows from $R_m^C$ along~$\chi$, it suffices to show that each of the bracketed factors has non-negative $\chi$-track, and that the last factor $w_{y_1} \cdots  w_{y_l}$~follows from~$R_m^C$ along~$\chi$.
		
	That each factor~$w_{y_1}\cdots w_{y_{j-1}}   r_{y_j} \bar w_{y_{j-1}} \cdots \bar w_{y_1}$ has non-negative $\chi$-track follows from the choice of  words~$w_x$ as having $\chi$-track that attains its minimum at the endpoints -- explicitly, the minimal value of its $\chi$-track is the same as that of $y_1\cdots y_{j-1} y_j \bar y_j \bar y_{j-1}\cdots \bar y_1$. The $\chi$-track of this word, in turn, takes the same values as that of~$y_1\cdots y_j$. But this word, being an initial sub-word of~$\rho$, has non-negative $\chi$-track. By a similar argument, the factor $w_{y_1} \cdots  w_{y_l}$ has non-negative $\chi$-track.
	
	Finally, since $w_{y_1} \cdots  w_{y_l}$~is a $D$-relator of non-negative $\chi$-track, it follows from~$R_n^D\subseteq R_m^C$ along~$\chi$.
	
	($C \subseteq D$) Again using Lemma~\ref{lem.compactexhaustionchi}, let $k \in \NN$ be such that $D_\chi \subseteq C^{k,\chi}$ (note the reversal of the roles of $C$~and~$D$ from the previous case). We will show that if $m:= nk$, then $\langle C \mid R_m^C\rangle$ is a compact presentation of~$G$ along~$\chi$. 
	
	For every $y \in D^\pm$, we choose as before a word $w_y$ in~$C$ of length at most~$k$ representing~$y$, and whose $\chi$-track attains its minimum at an endpoint. We also insist that for every $y\in D^\pm$, we have $w_{\bar y} =\bar w_y$ and if $y\in C^\pm$, then $w_y=y$. Let $\rho$~be a  $C$-relator of non-negative $\chi$-track. Since $C \subseteq D$, we see $\rho$~is a word in~$D^\pm$ and thus freely equivalent (as $D^\pm$-word) to a word of non-negative $\chi$-track
	\[u:= u_1 s_1 \bar u_1 \cdots u_l s_l \bar u_l,\]
	where each $s_j$~is a relator in~$R_n^D$ and $u_j$ is a $D^\pm$-word. By replacing each letter~$y$ in~$u$~with the corresponding $C^\pm$-word $w_y$, we obtain a $C$-relator
	\[w := w_1 r_1 \bar w_1 \cdots w_l r_l \bar w_l.\]
	As before, the fact that the words~$w_y$ have $\chi$-tracks with minimum at the endpoints, and that $u$~has non-negative $\chi$-track, shows that $w$~has non-negative $\chi$-track. Moreover, since the~$w_y$ have length at most~$k$, the relators $r_j$ have length at most~$nk$.
	
	To conclude that $\rho$~follows from $R_m^C$ along~$\chi$, we need only verify that $w$~is freely equivalent to~$\rho$ as a $C^\pm$-word. Roughly speaking, this follows from the fact that each elementary insertion or cancellation of $D^\pm$-words used in passing from~$u$ to~$\rho$ can be mirrored in terms of $C^\pm$-words.
	Formally, given a $D^\pm$-word $v_0$, let $v_0'$ denote the $C^\pm$-word obtained from~$v_0$ by replacing each letter~$y$ with~$w_y$. The key observation is that if $v_1$~is obtained from~$v_0$ by the insertion or deletion of a sub-word $y\bar y$, then $v_1'$~is obtained from~$v_0'$ by the insertion or deletion of a sub-word $w_y \bar w_y$, respectively. This is true because of our choosing $w_{\bar y} = \bar w_y$. Since $u$~is freely equivalent as a $D^\pm$-word to~$\rho$, we conclude that $w$~is freely equivalent as a $C^\pm$-word to $\rho'$. But since $\rho$ is actually a $C^\pm$-word and we also assumed $w_y =y$ for $y\in C^\pm$, we have $\rho'=\rho$.
\end{proof}

Together with Lemma~\ref{lem:Kochdef}, the next theorem establishes that our definition of $\TopS^2(G)$ matches the one given by Kochloukova:

\begin{thm}[$\TopS^2$ and compact presentability]\label{thm:sigma2_cptpres}
	A character $\chi \colon G\to \RR$ lies in $\TopS^2(G)$ if and only if $G$~is compactly presented along~$\chi$. 
\end{thm}
\begin{proof}
	($\Rightarrow$) Since $\chi \in \TopS^2(G)$, 
	we have in particular that $\chi \in \TopS^1(G)$,  so by Theorem~\ref{thm:sigma1_cptgen} there is a centered $C\in \C(G)$ that generates~$G$ along~$\chi$.
	
	Now, one of the characterizations of~$\TopS^2(G)$ given by Corollary~\ref{cor.powersofC} tells us there is $m\in \NN_{\ge 1}$ such that the inclusion $(G_\chi \cdot \E C)_\chi \into (G_\chi \cdot \E C^m)_\chi$ is $\pi_1$-trivial. The composition $\Gamma(G,C)_\chi \into (G_\chi \cdot \E C)_\chi \into (G_\chi \cdot \E C^m)_\chi$ is therefore also $\pi_1$-trivial.
	Using Lemma~\ref{lem.compactexhaustionchi} to choose $k\in \NN$ such that $C^{2m} \cap G_\chi \subseteq C^{k, \chi}$, we deduce from Proposition~\ref{prop.trivialloops}~(1) that all $C$-relations in~$G$ of non-negative $\chi$-track follow from $R_{3k}$ along~$\chi$.
	This shows $\langle C\mid R_{3k}\rangle$ is a compact presentation of~$G$ along~$\chi$.
	
	($\Leftarrow$) We will use the characterization of $\chi \in \TopS^2(G)$ as essential $1$-connectedness of the filtration $((G_\chi \cdot \E C)_\chi)_{C\in \C(G)}$ given by Proposition~\ref{prop.trimmedsigma}. So fix $C \in \C(G)$, which we may assume is a centered generating set for~$G$ along~$\chi$, and thus $(G_\chi \cdot\E C)_\chi$ is connected.
	
	By Proposition~\ref{prop.anygensetchi}, $C$~fits into a compact presentation $\langle C \mid R_m \rangle$ of~$G$ along~$\chi$.
	Since every element of~$\pi_1(\Gamma(G,C)_\chi, 1)$ is represented by a $C$-relator of non-negative $\chi$-track,  Proposition~\ref{prop.trivialloops} (2) tells us that the composition
	\[\Gamma(G,C)_\chi \into (G_\chi \cdot \E C)_\chi \into (G_\chi \cdot \E C^{\lfloor \frac m2\rfloor})_\chi\]
	is $\pi_1$-trivial. As the first map is $\pi_1$-surjective by Lemma~\ref{lem.pi1onto}, we conclude the inclusion
	$(G_\chi \cdot \E C)_\chi \into(G_\chi \cdot \E C^{\lfloor \frac m2\rfloor})_\chi$ is $\pi_1$-trivial, as desired.
\end{proof}

\begin{prop}[$\TopS^2$-criteria]
	Let $\chi\in \TopS^1(G)$ and let $C\subseteq G$ be a compact generating set. The following conditions are equivalent:
	\begin{enumerate}
		\item $\chi \in \TopS^2(G)$,
		\item There exists $m\in \NN$ such that for all $s\in \RR$, the complex~$\Gamma(G, C\mid R_m)_s$ is simply connected.
		\item For some $m\in\NN$, the complex $\Gamma(G, C\mid R_m)_\chi$ is simply connected.
		\item For some $m \in \NN$ and some $s\in \RR_{\le0}$, the inclusion $\Gamma(G,C\mid R_m)_\chi \into \Gamma(G,C \mid R_m)_s$ is $\pi_1$-trivial.
	\end{enumerate}
\end{prop}

\begin{proof}
	It is immediate that $2 \Rightarrow 3 \Rightarrow 4$.
	
	$(1\Leftrightarrow 3)$ By Theorem~\ref{thm:sigma2_cptpres}, we have $\chi \in \TopS^2(G)$ if and only if $G$~is compactly presented along~$\chi$. By Proposition~\ref{prop.anygensetchi}, this is equivalent to the existence of a compact presentation of~$G$ along $\chi$ of the form~$\langle C\mid R_m\rangle$. By Corollary~\ref{cor:cptpres_1conn}, this is equivalent to Condition 3.
	
	$(4\Rightarrow 2)$ This is immediate if $\chi = 0$, so we assume $\chi \neq 0$ and choose $t\in C$ with $\chi(t)>0$. We may also assume $C$~is centered.
	
	Because $\chi \in \TopS^1(G)$, 
	Proposition~\ref{prop:sigma1_crit} tells us that $C$~generates $G$~along $\chi$. As $t^{-1} C_\chi t$ is a compact subset of~$G_\chi$, it follows from Lemma~\ref{lem.compactexhaustionchi} that $t^{-1} C_\chi t \subseteq C^{k, \chi}$ for some $k\in \NN$. In other words, $C_\chi \subseteq t C^{k, \chi} t^{-1}$, and thus every $x\in C_\chi$ may be expressed as~$t w_x \bar t$ with $w_x$~a $C^\pm$-word of non-negative $\chi$-track and length at most~$k$. For $x\in C$ with $\chi(x) <0$, define $w_x := \bar w_{\bar x}$, so again $x$~is represented by $t w_x \bar t$.
	
	Now, let $m, s$ be as in Condition 2 and put $M := \max\{m, k+3\}$. We show that for all $r\le s$, the complex $\Gamma(G, C\mid R_M)_r$ is simply connected. This is sufficient, since for every $s'\in \RR$, the complex $\Gamma(G, C\mid R_M)_{s'}$ has a $G$-translate of the form $\Gamma(G, C\mid R_M)_r$ with $r\le s$.
	
	Since $M\ge m$ and $r\le s$, we have a commutative square of inclusion maps
	\[ \begin{tikzcd}
		\Gamma(G, C \mid R_m)_\chi \arrow[r, hook]\arrow[d, hook]& \Gamma(G, C \mid R_m)_s \arrow[d, hook]\\
		\Gamma(G, C \mid R_M)_\chi \arrow[r, hook] & \Gamma(G, C \mid R_M)_r 
	\end{tikzcd},\]
	where the top map is $\pi_1$-trivial by assumption and the left-hand map is $\pi_1$-surjective because the spaces differ by the attachment of $2$-cells. We thus conclude that the lower map is $\pi_1$-trivial, and so the proof will be complete as soon as we show it is $\pi_1$-surjective.
	
	Observe that, for each $x\in C$, the $C$-relator $x t \bar w_x \bar t$ lies in~$R_M$. These relators allow us to homotope relative endpoints any edge in $\Gamma(G, C \mid R_M)_r$ labeled by~$x$ to the path reading $tw_x \bar t$. Now consider an edge-loop~$\gamma$ in $\Gamma(G,C\mid R_M)_r$, say based at~$g$ and reading the $C$-relator $w=y_1 \cdots y_l$.
	The preceding observation shows that $\gamma$~is pointed-homotopic to the loop reading $(tw_{y_1}\bar t) (tw_{y_2}\bar t) \cdots (tw_{y_l}\bar t)$, which in turn is freely homotopic to the loop $\gamma^+$ starting at $gt$ and reading $w':= w_{y_1} w_{y_2} \cdots w_{y_l}$. Since for each $y\in C^\pm$ the $\chi$-track of~$w_y$ attains its minimum at an endpoint, the minimal value of the $\chi$-track of $w'$ is the same as that of~$w$.
	Hence, as $\gamma$ lies in $\Gamma(G, C \mid R_M)_r$, we conclude $\gamma^+$~lies in $\Gamma(G, C \mid R_M)_{r+\chi(t)}$. Repeating this technique enough times, we can show $\gamma$~is homotopic to a loop in $\Gamma(G, C \mid R_M)_{\chi}$; in other words, the bottom map in the above commutative square is $\pi_1$-surjective.
\end{proof}

We close this section by generalizing the well-known fact that finitely presented groups admit a presentation whose relators have length at most~$3$.

\begin{prop}[Relators of length~$3$ suffice]
	If $G$~is compactly presented along a character~$\chi$, then there is a centered $C \in \C(G)$ such that $\langle C \st R_3\rangle$ is a presentation of~$G$ along~$\chi$.
\end{prop}
 \begin{proof}	
	Let $\langle D \mid R_m\rangle$ be a compact presentation of~$G$ along~$\chi$, where we may assume that $D$~is centered and $m\ge 2$.
	We shall take
	\[C:=D^{m, \chi} \cup \big(D^{m, \chi}\big)^{-1},\]
	which is clearly centered. Since $D \subseteq C$, we see $C$~generates~$G$, so by Theorem~\ref{thm:sigma1_cptgen} it generates~$G$ along~$\chi$. We need to show every $C$-relator $\rho=y_1 \cdots y_l$ of non-negative $\chi$-track follows from~$R_3^C$ along~$\chi$.
	
	Each $x\in C_\chi$ can be expressed by a $D$-word~$w_x$ of non-negative $\chi$-track and length at most~$m$. In case $x\in D_\chi$, we may further assume $w_x=x$. For $x \in C$ with $\chi(x)<0$, put $w_x := w_{x^{-1}}^{-1}$ (that is, the $D$-word obtained from $w_{x^{-1}}$ by reversing the order and swapping each letter with its inverse in~$D$), and for $y\in \bar C$, consider the $\bar D$-word $w_y:=\bar w_{\bar y}$. Thus for every $y\in C^{\pm}$, the $\chi$-track of~$w_y$ attains its minimum at one of its endpoints, so the $D$-relator $\rho' := w_{y_1} \cdots w_{y_l}$ has non-negative $\chi$-track.
	We first show that $\rho'$, interpreted as a $C$-relator, follows from~$R_3^C$ along~$\chi$ if and only if $\rho$~does. Afterwards, we will see that $\rho'$ does indeed follow from~$R_3^C$ along~$\chi$.
	
	For the first assertion, by Lemma~\ref{lem.wordsincayleycomplex}, we need only show that the edge-loops at~$1$ reading $\rho$ and~$\rho'$ are homotopic in~$\Gamma(G, C\mid R_3^C)_\chi$. We will prove a stronger statement: each edge $g\cdot (1,x)$ of $\Gamma(G, C\mid R_r^C)_\chi$ is homotopic relative endpoints to the edge-path at~$g$ reading~$w_x$. 
	
	So write $w_x= x_1\cdots x_k$ with $3\le k\le m$  (for $k\le 2$ the statement is immediate) and for each $i\le k$, let~$g_i := x_1\cdots x_i$. If $\chi(x) \ge 0$, then $g_i \in D^{m, \chi}$; otherwise we have $g_i^{-1} \in D^{m, \chi}$. Either way, we see $g_i \in C$.
	Hence, for every $2 \le i \le k$, the graph $\Gamma(G,C)_\chi$ contains the edges
	$g\cdot (1,g_{i-1})$ and $g \cdot (1, g_i)$. Together with the edge $gg_{i-1}(1,x_i)$, they form a loop reading the $C$-relator $g_{i-1} x_i \bar g_i \in R_3^C$, which is filled by a $2$-cell in $\Gamma(G \mid C , R_3^C)_\chi$. These triangles assemble to give the desired homotopy relative endpoints from $g\cdot (1,x)$ to the path reading $w_x$; see the example in Figure~\ref{fig:three_suffice_a}.
	
	\begin{figure}[h]
		\centering
		\def \svgwidth{0.5\linewidth}
		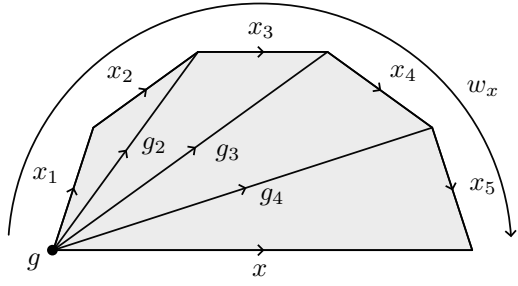
		\caption{A collection of $k-1$ triangles in $\Gamma(G, C \mid R_3^C)_\chi$ exhibits the edge $g\cdot(1,x)$ as homotopic relative endpoints to the path at $g$ reading $w_x = x_1 \cdots x_k$.}
		\label{fig:three_suffice_a}
	\end{figure}
	
	We are left to show that $\rho'$ follows from~$R_3^C$ along~$\chi$. Of course, being a $D$-relator, $\rho'$~follows from $R_m^D$ along~$\chi$, so it is freely equivalent to a product of relators $w r \bar w$ with non-negative $\chi$-track and with $r \in R_m^D$.
	Using Lemma~\ref{lem:posrels}, we may assume $r$~has non-negative $\chi$-track. Thus the proof will be complete once we show that every $r \in R_m^D$ with non-negative $\chi$-track follows from $R_3^C$ along~$\chi$. We may also restrict to the case where $r$~is of the form $z_1 \cdots z_l$ with $z_i \in D^{\pm}$ and $4\le l\le m$.
	
	Again appealing to Lemma~\ref{lem.wordsincayleycomplex}, we will prove this by showing that the edge-loop~$\gamma$ in~$\Gamma(G,C\mid R_3^C)_\chi$ at~$1$ reading~$r$ is null-homotopic. The argument is similar to the previous one: write $r=z_1 \cdots z_l$ with $z_i \in D^{\pm}$. Each initial sub-word $z_1 \cdots z_i$ represents an element~$g_i \in  D^{m, \chi} \subseteq C$, so $\Gamma(G,C)_\chi$ contains all edges of the form $(1, g_{i})$. Together with the edges of~$\gamma$, they form length-3 loops reading $C$-relators, which are thus filled by $2$-cells in $\Gamma(G,C\mid R_3^C)_\chi$. These triangles assemble to a filling of~$\gamma$, as illustrated in Figure~\ref{fig:three_suffice_b}.\qedhere
	
	\begin{figure}[h]
		\centering
		\def \svgwidth{0.4\linewidth}
		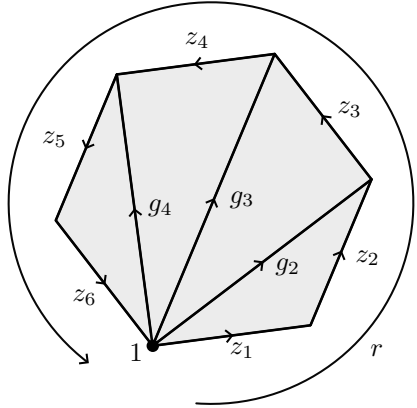
		\caption{Triangles in $\Gamma(G, C \mid R_3^C)_\chi$ exhibiting the loop~$\gamma$ reading~$r$ as null-homotopic. The orientation of edges labeled by $z_i$ should be understood to be reversed when $z_i\in \bar C$.}
		\label{fig:three_suffice_b}
	\end{figure}
 \end{proof}
\section{$\TopS^n$ via proper actions}\label{sec:properaction}

We now present a characterization of the $\Sigma$-sets in terms of an action on a locally compact Hausdorff space~$X$, extending results in Section~3.2 of Abels--Tiemeyer's article \cite{AT97}. Actions $\alpha \colon G\times X  \to X$ are always assumed to be continuous and denoted by $\alpha(g,x) =: g\cdot x$. We remind the reader that $\alpha$~is \textbf{proper} if the map
\begin{align*}
	\alpha_+ \colon G\times X &\to X \times X \\
	(g,x)& \mapsto (x, g\cdot x)
\end{align*}
is proper. Another useful description is the following:

\begin{lem}[Properness via $\Ret$]\label{lem:ret}
	The action $\alpha \colon G\times X \to X$ is proper if and only if for every $K\in \C(X)$ the set \[\Ret(K):=\{g \in G \mid  g\cdot K \cap K \neq \emptyset\}\]
	is compact.
\end{lem}

\begin{proof}
	($\Rightarrow$) Properness of~$\alpha$ implies that for every $K\in \C(X)$, the set
	$$\alpha_+^{-1}(K\times K) = \{(g,k) \in G\times K \mid g\cdot k\in K\}$$
	is compact. Its projection to the $G$-factor is precisely $\Ret(K)$, which is therefore compact.
	
	($\Leftarrow$) Let $Q\subseteq X\times X$ be compact; we need to show that $\alpha_+^{-1}(Q)$ is compact.
	Consider the compact set $K = \pr_1(Q) \cup \pr_2(Q) \subseteq X$, and note that
	$$\alpha_+^{-1}(Q) = \{(g,x) \in G\times X \mid x \in \pr_1(Q), g\cdot x \in \pr_2(Q)\}$$
	is a closed subset of $\Ret(K) \times K$.
	Because  $\Ret(K) \times K$ is compact and Hausdorff, we conclude $\alpha_+^{-1}(Q)$ is compact.
\end{proof}

\begin{prop}[$\Sigma^n(G)$ via proper actions]\label{prop:properaction}
	Let $X$~be a nonempty locally compact Hausdorff space with a proper $G$-action, let $n\in \NN$, and let $\chi\colon G \to \RR$ be a character. Then $\chi \in \TopS^n(G)$ if and only if the filtration
	\[(G_\chi \cdot \E K)_{K \in \C(X)}\]
	of $\E X$ is essentially $(n-1)$-connected.
\end{prop}

For $\chi =0$, this fact has been established by Abels--Tiemeyer \cite[Theorem~3.2.2]{AT97}, using essentially the same argument we now present.
\begin{proof}
	Fix any $x_0 \in X$. For each $C\in \C(G)$, note that $Cx_0 \in \C(X)$ and consider the map of simplicial sets
	\begin{align*}
		\rho_C \colon G_\chi \cdot \E C &\to G_\chi \cdot \E (C x_0)\\
		(g_0, \ldots, g_k) & \mapsto (g_0 \cdot x_0, \ldots, g_k \cdot x_0).
	\end{align*}
	The~$\rho_C$, taken over all~$C$, form a commuting family of maps up to simplicial homotopy from $(G_\chi \cdot \E C)_{C\in \C(G)}$ to $(G_\chi \cdot \E K)_{K\in \C(X)}$ (recall Definition~\ref{dfn:hoequiv} -- in fact, the relevant diagrams commute ``on the nose'', not just up to homotopy).
	By Lemma~\ref{lem:homotopyequiv}, the statement will follow once  we show that this family has a simplicial homotopy-inverse. 
	
	Given $K\in \C(X)$ and $y\in G_\chi \cdot K$, we choose an expression $y = g_y \cdot z_y$ with $g_y\in G_\chi$ and $z_y\in K$.
	 We claim there is a map of simplicial sets
	\begin{align*}
		s_K \colon G_\chi \cdot \E K &\to G_\chi \cdot \E \Ret(K)\\
		(y_0, \ldots, y_k) &\mapsto (g_{y_0}, \ldots, g_{y_k}).
	\end{align*}
	Indeed, writing $(y_0, \ldots, y_k) = g\cdot (l_0, \ldots, l_k)$ with $g \in G_\chi$ and $l_i \in K$, we have $g_{y_i} \cdot z_{y_i} = g \cdot l_i$ for each~$i$, so $g^{-1} g_{y_i} \cdot z_{y_i} \in K$, and hence $g^{-1}g_{y_i} \in \Ret (K)$.
	The expression $(g_{y_0}, \ldots, g_{y_k}) = g\cdot (g^{-1}g_{y_0}, \ldots, g^{-1}g_{y_k})$ thus shows this is a simplex in $G_\chi \cdot \E \Ret(K)$.
	
	Moreover, the $(s_K)_{K\in \C(X)}$ commute up to simplicial homotopy. To see this, consider compact subsets $K\subseteq K'$ of~$X$ and
	the (unique) simplicial homotopy $H\colon G_\chi \cdot \E K \times \Delta^1 \to \E G$ from~$s_K$ to~$s_{K'}$. It is given by 
	$$H\left((y_0, \ldots, y_k\right), \tau_i^k) = (g_{y_0}, g_{y_{i-1}}, g_{y_i}',\ldots, g_{y_k}'),$$
	where $g_y$~and~$g_y'$ denote the elements of~$G_\chi$ chosen when constructing $s_K$ and $s_{K'}$, respectively. The same argument as in the previous paragraph shows the simplex on the right hand side lies in $G_\chi \cdot \E \Ret(K')$.
	
	To establish that $(\rho_C)_{C\in\C(G)}$ and $(s_K)_{K\in \C(X)}$ do indeed form an inverse pair of simplicial homotopy equivalences, we claim that for each $C\in \C(G)$ and $K\in \C(X)$, the following diagrams commute up to simplicial homotopy:
	$$\begin{tikzcd}
		G_\chi \cdot \E C \ar[r,hook] \ar[d, "\rho_C"]& G_\chi \cdot \E (C\cup \Ret (Cx_0))\\
		G_\chi \cdot \E(Cx_0) \ar[r,"s_{Cx_0}"] & G_\chi \cdot \E \Ret(C x_0)\ar[u,hook]
	\end{tikzcd} \begin{tikzcd}
		G_\chi \cdot \E K \ar[r,hook] \ar[d, "s_K"]& G_\chi \cdot \E (K\cup \Ret (K)x_0)\\
		G_\chi \cdot \E\Ret(K) \ar[r,"\rho_{\Ret(K)}"] & G_\chi \cdot \E (\Ret(K)x_0)\ar[u,hook]
	\end{tikzcd}$$
	
	For the left diagram, this amounts to checking that the (unique) simplicial homotopy with codomain $\E G$ between the two relevant maps has image in $G_\chi \cdot \E(C \cup \Ret (Cx_0))$.
	This follows if we prove that for each $g\in G_\chi$, both the inclusion and the composition $s_{Cx_0} \circ \rho_C$ send $g\cdot \E C$ to $g\cdot \E(C\cup \Ret(Cx_0))$. And as the latter simplicial set is free, one need only verify the statement on vertices; all other necessary simplices are automatically available. Now, the claim is certainly true for the inclusion; as for $s_{Cx_0} \circ \rho_C$, given $c\in C$, we compute the image of the vertex $gc$:
	$$s_{Cx_0}(\rho_C(gc)) = s_{Cx_0}(gc\cdot x_0) = g_{gc\cdot x_0}.$$
	By the same argument as before, we have $g_{gc\cdot x_0} \in g \Ret(Cx_0)$, as required.
	
	The proof for the second diagram proceeds similarly.
\end{proof}

A subgroup $H\le G$ is called \textbf{cocompact} if the left action of $H$ on~$G$ is cocompact, that is, there is  $C \in \C(G)$ such that $HC=G$.

\begin{thm}[Closed cocompact subgroups]\label{thm:cocptsubgps}
	Let $H\le G$ be a closed cocompact subgroup and let $\chi\colon G \to \RR$ be a character. Then, 
	$$\chi\in\TopS^n(G) \iff \chi|_H\in\TopS^n(H).$$ 
\end{thm}
\begin{proof}
	We first note that the action of~$H$ on~$G$ by left multiplication is proper: indeed, since $H$~is closed in~$G$, for every $C \in \C(G)$, the set
	\[
	\Ret(C)=\{h\in H\mid hC\cap C\not=\emptyset\}=CC^{-1}\cap H
	\]
	is compact, and so the claim follows from Lemma~\ref{lem:ret}.
	Therefore, by Proposition~\ref{prop:properaction}, we have $\chi|_H \in \TopS^n(H)$ if and only if the filtration $(H_\chi\cdot \E C)_{C\in \C(G)}$ of~$\E G$ is essentially $(n-1)$-connected. We are thus reduced to showing that the filtrations
	\[(H_\chi \cdot \E C)_{C\in \C(G)}\qquad \text{and} \qquad (G_\chi\cdot \E C)_{C\in \C(G)}\]
	of $\E G$ are cofinal.

	One of the directions is immediate: given $C\in \C(G)$, we have
	$H_\chi \cdot \E C\subseteq G_\chi \cdot \E C$.
	For the other inclusion, use that $H$ is cocompact to choose $K\in \C(G)$ with $HK =G$, and let $g_0 \in G$ be such that $\chi (g_0) \le - \max (\chi(K))$. We first show that $G_\chi \subseteq H_\chi Kg_0$. Indeed, since $HKg_0 = Gg_0 = G$, every $g\in G_\chi$ may be written as $g=hkg_0$ with $h\in H$ and $k\in K$. But then $\chi(h) = \chi(g) - \chi(kg_0) \ge 0$, so $h\in H_\chi$.
	From here we conclude that, given $C\in \C(G)$, we have
	\[G\chi \cdot \E C \subseteq H_\chi Kg_0 \cdot \E C \subseteq H_\chi \cdot \E (Kg_0C).\qedhere\] 
\end{proof}

\begin{cor}[Abelian groups]\label{cor:abelian}
	If $G$~is a compactly generated abelian group, then for every $n\in \NN$, we have $\TopS^n(G) = \TopHom(G,\RR)$.
\end{cor}
\begin{proof}
	The classification theorem for compactly generated locally compact abelian groups \cite[Theorem~24]{Mor77} tells us that $G \cong \ZZ^{m_1} \times \RR^{m_2} \times Q$, with $m_1, m_2\in \NN$ and $Q$~compact. Thus the subgroup $H:= \ZZ^{m_1}\times \ZZ^{m_2}$ is a closed cocompact subgroup of~$G$. 
	It is well-known that for every $n \in \NN$ we have $\Sigma^n(H) = \Hom(H,\RR)$ (one can see this by first noting that $H$~is of type~$\mathrm F_n$, since it has the $(m_1+ m_2)$-dimensional torus as a classifying space, and then using Renz's Theorem \cite[Satz~C]{Ren88} with $N=1$). By Theorem~\ref{thm.sigmasmatch}, it follows that $\TopS^n(H) = \Hom(H, \RR)$.
	
	Now, every character on~$H$ extends to a character on~$G$ (uniquely, in fact): on the $\RR^{m_2}$ factor we use the only $\RR$-linear map extending the definition on~$\ZZ^{m_2}$, and we send~$Q$ to~$0$. Thus, we conclude from Theorem~\ref{thm:cocptsubgps} that $\TopS^n(G) = \TopHom(G,\RR)$.
\end{proof}
\section{$\TopS^n$ and quotients}\label{sec:quotients}

In this section, we relate the homotopical $\Sigma$-sets of the locally compact Hausdorff groups in an extension
\[1\to N \to G \overset{p}{\to} Q \to 1.\]
Explicitly, this means that the indicated maps form a short exact sequence of the underlying abstract groups and are continuous, and that $N$~and~$Q$ carry, respectively, the subspace and quotient topologies induced from~$G$. The fact that~$N$ is locally compact and $G$~is Hausdorff implies that $N$~includes as a closed subgroup of~$G$ \cite[Chapter~1, Proposition~7]{Mor77}.

We also fix a character $\chi \colon Q \to \RR$ and denote by $\hat{\chi}\colon G \to \RR$ its pull-back~$\chi \circ p$ (which of course vanishes on~$N$). Note that we do not assume $\chi\neq 0$, so the results in this section also apply to the compactness properties~$\mathrm C_n$.

To warm up, we consider the case where $p$~has a \textbf{section}, that is, a continuous group homomorphism $s\colon Q \to G$ with $p\circ s = \id_{Q}$. Equivalently, the above short exact sequence expresses $G$~as a semidirect product $G \cong N \rtimes Q$.

\begin{prop}[Semidirect products]\label{prop:split}
	Suppose the quotient map $p\colon G \to Q$ has a section~$s$, and let $n\in\NN$. Then,
	$$\hat \chi \in \TopS^n(G) \implies \chi \in \TopS^n(Q).$$
\end{prop}
\begin{proof}
	Given $C\in \C(Q)$, the set $s(C)$~is compact by continuity of~$s$. Assuming $\hat\chi\in \TopS^n(G)$, we find $D\in\C(G)_{\supseteq s(C)}$ such that the inclusion $G_{\hat\chi} \cdot \E s(C) \into G_{\hat\chi} \cdot \E D$ is $\pi_k$-trivial for all $k\le n-1$. We now consider the compact set~$p(D)$. Since $s$~is a section, the inclusion $Q_\chi \cdot \E C \into Q_\chi \cdot \E p(D)$ factors as 
	$$Q_\chi \cdot \E C \xrightarrow{s} G_{\hat\chi} \cdot \E s(C) \into G_{\hat\chi} \cdot \E D \xrightarrow{p} Q_\chi \cdot \E p(D),$$
	where we slightly abuse notation in denoting by~$s$ and~$p$ the induced maps on simplicial sets (as explained in Remark~\ref{rem:functfilt}). The composition is therefore $\pi_k$-trivial for all $k\le n-1$, showing $\chi \in \TopS^n(Q)$.
\end{proof}

A converse to the implication in Proposition~\ref{prop:split} would be too much to hope for (take for example $Q$~as the trivial group). However, if one has some control on the compactness properties of the kernel~$N$, it is possible to relate the $\Sigma$-sets of~$G$ and~$Q$, even without the help of a section:

\begin{thm}[Quotients by subgroups of type~$\mathrm C_n$]\label{thm:sigmaofquotient} Let $n\in \NN$ and suppose $N$~is of type~$\mathrm{C}_n$. Then,
	\begin{enumerate}
		\item $\chi \in \TopS^n(Q) \implies \hat\chi \in \TopS^n(G)$,
		\item $\hat\chi \in \TopS^{n+1}(G) \implies \chi \in \TopS^{n+1}(Q)$.
	\end{enumerate}
\end{thm}

The remainder of this section is devoted to the proof of Theorem~\ref{thm:sigmaofquotient}.


\begin{lem}[Compact lifts]\label{lem:cptlifts}
	For each $D\in \C(Q)$, there is $C\in \C(G)$ with $p(C)=D$.
\end{lem}
\begin{proof}
	We show that $D\subseteq p(C_0)$ for some $C_0\in \C(G)$; the lemma then follows by putting $C:= C_0 \cap p^{-1}(D)$. For each $q\in D$, choose a $p$-preimage $g_q\in G$. Since $G$~is locally compact, there is an open subspace $U_q \subseteq G$ with compact closure containing~$g_q$. Moreover, quotients by group actions are open maps, so~$p$ is open. Therefore, we obtain an open cover $\{p(U_q) \cap D\}_{q\in D}$ of~$D$, which has a sub-cover indexed by a finite subset $F\subseteq D$. Now in $G$, the space $C:=\bigcup_{q\in F} \overline{U_q}$ is a finite union of compact sets, and hence compact. Clearly, $D\subseteq p(C)$.
\end{proof}

\begin{cor}[Re-indexing by~$\C(G)$]\label{cor.uniformfiltrations} We have:
	\begin{enumerate}
		\item $\C(N) = \{C\cap N \mid C\in \C(G)\}$,
		\item $\C(Q) = \{p(C) \mid C\in \C(G)\}$.
	\end{enumerate}
\end{cor}
\begin{proof}
	(1) Every compact subspace $D\subseteq N$~is also a compact subspace of~$G$. Conversely, since $N$~is closed in~$G$, it follows that for every $C\in \C (G)$, the intersection $C\cap N$ is compact.
	
	(2) Given $C\in \C(G)$, continuity of~$p$ implies that $p(C)$ is compact. The converse direction follows by Lemma~\ref{lem:cptlifts}.
\end{proof}

The role of the section~$s$ in Proposition~\ref{prop:split} will be replaced in Theorem~\ref{thm:sigmaofquotient} by a lifting property afforded to us by $N$~having type~$\mathrm C_n$. In what follows, we denote by~$p$ also the maps induced by~$p$ on the various filtration stages, as explained in Remark~\ref{rem:functfilt}.

\begin{lem}[Lifting along quotients by subgroups of type~$\mathrm{C}_n$]\label{lem:Cnlift}	
	Suppose $N$~is of type~$\mathrm{C}_n$. Then for every $C\in \C(G)$ there is $D\in \C(G)_{\supseteq C}$ such that the following lifting property is satisfied: for every pair $A\subseteq X$ of simplicial sets of dimension at most~$n$, with $X$~finite relative to~$A$, and every commutative square
	\[\begin{tikzcd}
		A \arrow[r, "\eta"] \arrow[d, hook]& G_{\hat\chi} \cdot \E C \arrow[d, "p"]\\
		X \arrow[r,"\mu"] & Q_\chi \cdot \E p(C)
	\end{tikzcd},\]
	there are $m\in \NN$ and a map $\tilde \mu \colon \SD^m(X) \to G_{\hat\chi}\cdot \E D$ such that in the diagram
	\[\begin{tikzcd}
		\SD^m(A) \ar[d,hook]\ar[r,"\Phi^m"]& A \arrow[r, "\eta"] &  G_{\hat\chi} \cdot \E C  \arrow[r,hook] &  G_{\hat\chi} \cdot \E D \arrow[d, "p"] \\
		\SD^m(X) \ar[r,"\Phi^m"']  \arrow[rrru,"\tilde \mu"] & X \arrow[r,"\mu"'] & Q_\chi \cdot \E p(C) \arrow[r,hook] & Q_\chi \cdot \E p(D)
	\end{tikzcd},\]
		the upper triangle commutes and the lower triangle commutes up to simplicial homotopy.	 
\end{lem}
\begin{proof}
	We will proceed by induction on~$n$. Along the way, we will also show that the maps~$\tilde\mu$ can be constructed so that following condition holds:
		
	\begin{itemize}
		\item[($\star$)] Let $\sigma$ be a simplex of~$X$, and let $h\in G_{\hat\chi}$ be such that $\mu(\sigma)$ lies in~$p(h \cdot \E C)$. Then $\tilde \mu (\SD^m(\sigma))\subseteq hN\cdot \E(CC^{-1}D)$.
	\end{itemize}
	
	For $n = 0$, take $D:=C$, and suppose we are given a commutative square as in the statement. For each vertex~$x$ of~$X$, we define $\tilde \mu(x)$ as follows: if $x\in A$, then we set (as we must) $\tilde \mu(x) = \eta(x)$; if $x\not \in A$, we take $\tilde \mu (x)$ to be any $p$-preimage of~$\mu(x)$. In particular, $\tilde \mu$ is a lift of~$\mu$ and both relevant triangles commute on the nose. To verify~$(\star)$, fix $x\in X^{(0)}=X$ and $h \in G_{\hat\chi}$~with $\mu(x) \in p(h \cdot \E C)$. Regarding~$\mu(x)$ as an element of~$Q$, we have  $\mu(x) \in p(hC)$. Thus, by our construction of $\tilde \mu$ as a lift of~$\mu$, we see $\tilde\mu(x) \in hCN = hN C$ (by normality of~$N$), which makes $\tilde \mu(x)$ a vertex of $hN\cdot \E C$.
	
	Now let $n\ge 1$. Having fixed $C\in \C(G)$,
	the induction hypothesis provides us with $D_0\in \C(G)_{\supseteq C}$ satisfying the stated lifting property on commutative squares with $X$~of dimension at most~$n-1$. Consider the proper action of~$N$ on~$G$ by left multiplication. Since $N$~is of type~$\mathrm C_n$,  Proposition~\ref{prop:properaction} applied to the zero character on~$N$ tells us that the filtration $(N\cdot \E C)_{C \in \C(G)}$ of $\E G$ is essentially $(n-1)$-connected. Thus there exists $D \in \C(G)_{\supseteq CC^{-1}D_0}$ such that the inclusion $N\cdot \E(CC^{-1}D_0) \into N\cdot \E D$ is $\pi_{n-1}$-trivial. We claim such~$D$ is as desired.
	
	To verify the required lifting property, let $X, A, \eta, \mu$ fit in a commutative square as above, with $X$~of dimension at most~$n$. By choice of~$D_0$, there are $m_0\in \NN$ and $\tilde\mu_0 \colon \SD^{m_0}(X^{(n-1)}) \to G_{\hat \chi}\cdot \E D_0$ solving the lifting problem given by the restriction of $\eta$ and~$\mu$ to the $(n-1)$-skeleta $A^{(n-1)}\subseteq X^{(n-1)}$.
	
	The crux of the problem lies in defining~$\tilde \mu$ on the (subdivided) non\-de\-ge\-ne\-rate $n$-simplices~$\sigma$ of~$X$ that are not in~$A$. Given such~$\sigma$, choose $g\in G_{\hat\chi}$ such that $\mu(\sigma)$ lies in $p(g)\cdot \E p(C) = p(g\cdot \E C)$. Property~$(\star)$ on~$\tilde \mu_0$ tells us that for each face $\tau$ of~$\sigma$, we have $\tilde \mu_0(\SD^{m_0}(\tau)) \subseteq gN\cdot \E(CC^{-1}D_0)$. In other words, $\tilde \mu_0(\SD^{m_0}(\partial \sigma)) \subseteq gN\cdot \E(CC^{-1}D_0)$. Therefore, by choice of~$D$ (and Lemma~\ref{lem:combfilling}), there are $m_\sigma \ge m_0$ and a map $\tilde \mu_\sigma \colon \SD^{m_\sigma}(\sigma) \to gN\cdot \E D$ extending~$\tilde \mu_0 \circ \Phi^{m_\sigma - m_0}$.
	
	Now, finiteness of~$X$ relative to~$A$ allows us to take~$m\in \NN$ as the maximum of all the~$m_\sigma$, and we are now ready to construct $\tilde \mu \colon \SD^m(X) \to G_{\hat\chi}\cdot \E D$. The required conditions on~$\tilde \mu$ dictate that the definition on~$\SD^m(A)$ be as $\eta \circ \Phi^m$. On $\SD^m(X^{(n-1)})$, we define~$\tilde \mu$ as $\tilde \mu_0 \circ \Phi^{m-m_0}$, which is consistent with the definition on~$\SD^m(A)$. Finally, on each nondegenerate $n$-simplex~$\sigma$ of~$X$ we define~$\tilde \mu$ as $\tilde \mu_\sigma \circ \Phi^{m-m_\sigma}$, which again matches the definition on $\SD^m(A \cup X^{(n-1)})$. Thus $\tilde \mu$~is defined so that the upper triangle of the above diagram commutes. 
	
	Next, let us establish~$(\star)$. 
	If $\sigma$~lies in~$X^{(n-1)}$, it follows directly by the induction hypothesis.
	
	If $\sigma$~is a nondegenerate $n$-simplex of~$X$ not in~$A$, recall that for the element $g\in G_{\hat\chi}$ chosen when constructing~$\tilde \mu_\sigma$, we have:
	\begin{enumerate}
		\item $\mu(\sigma) \subseteq p(g\cdot \E C)$,
		\item $\tilde \mu(\SD^m(\sigma)) \subseteq gN\cdot \E D$, and 
		\item $\mu(\sigma) \subseteq p(h\cdot \E C)$, the assumption of~$(\star)$.
	\end{enumerate}
	Conditions 1 and 3 imply that in~$Q$ we have $p(gC) \cap p(hC) \neq \emptyset$, which means there are $x, x'\in C$ and $t\in N$ with $gx = thx'$. Therefore $g\in NhCC^{-1} = hNCC^{-1}$ (since $N$~is normal), and so Condition~2 allows us to conclude:
	$$\tilde\mu(\SD^m(\sigma)) \subseteq hNCC^{-1}\cdot \E D \subseteq hN\cdot \E(CC^{-1}D).$$
	
	If $\sigma$~is a nondegenerate $n$-simplex of~$A$, we let $g \in G_{\hat\chi}$ be any element for which $\eta(\sigma) \subseteq g\cdot \E C$, so in particular $g$~satisfies Condition~1. Using the fact that $\tilde \mu(\SD^m(\sigma)) = \eta(\sigma)$ and that $C\subseteq D$, we see that Condition~2 also holds for~$g$. Thus, the same argument as before establishes~$(\star)$.
	
	Finally, commutativity up to homotopy of the lower triangle amounts to showing that the (unique) simplicial homotopy
	$H\colon \SD^m(X) \times \Delta^1 \to \E Q$ from~$\mu\circ \Phi^m$ to~$p\circ \tilde\mu$ has image in $Q_\chi \cdot \E p(D)$. But in each case of our construction, one readily sees that for every simplex~$\sigma$ of~$X$, there is $g\in G_{\hat\chi}$ such that
	$$\mu\circ \Phi^m(\SD^m(\sigma))\subseteq p(g\cdot \E C) \subseteq p(g\cdot \E D)\quad \text{and} \quad  p\circ \tilde\mu(\SD^m(\sigma))\subseteq p(g\cdot \E D).$$
	In other words, both maps take $\SD^m(\sigma)$ into the translated free simplicial set $p(g)\cdot \E p(D)$, which thus contains all simplices of $H(\SD^m(\sigma) \times \Delta^1)$. The conclusion follows at once.
\end{proof}

\begin{proof}[Proof of Theorem~\ref{thm:sigmaofquotient}]
	(1) Given $C\in \C(G)$, the fact that $\chi \in  \TopS^n(Q)$, together with Corollary~\ref{cor.uniformfiltrations}~(2), produces $C'\in \C(G)_{\supseteq C}$ such that the inclusion $Q_\chi \cdot \E p(C) \into Q_\chi \cdot \E p(C')$ is $\pi_k$-trivial for all $k \le n-1$. Let $D\in \C(G)_{\supseteq C'}$ be given by applying Lemma~\ref{lem:Cnlift} to~$C'$. We claim that $G_{\hat\chi} \cdot \E C$ includes $\pi_k$-trivially in $G_{\hat\chi} \cdot \E D$ for all $k\le n-1$.
	
	For every $\eta \in \Map(\partial \Delta^{k+1}, G_{\hat\chi}\cdot \E C)$, there is a filling $\mu \in \Map(\Delta^{k+1}, Q_\chi \cdot \E p(C'))$ of $p\circ \eta$ (by Lemma~\ref{lem:combfilling}). We apply the defining property of~$D$ to find a lift $\tilde\mu\in\Map(\Delta^{k+1}, G_{\hat\chi}\cdot\E D)$ of~$\mu$, thus exhibiting~$\eta$ as nullhomotopic in $G_{\hat\chi}\cdot \E D$.
	
	(2) Start with a compact subset of~$Q$, which by Corollary~\ref{cor.uniformfiltrations}~(2) is of the form~$p(C)$ with $C\in\C(G)$. Applying Lemma~\ref{lem:Cnlift} to~$C$, we produce $D\in \C(G)_{\supseteq C}$ with the indicated lifting property. Using that $\hat\chi \in \TopS^{n+1}(G)$, let $D'\in \C(G)_{\supseteq D}$ be such that $G_{\hat\chi}\cdot \E D \into G_{\hat\chi} \cdot \E D'$ is $\pi_k$-trivial for $k\le n$.
	
	Now, given $\eta \in \Map(\partial \Delta^{k+1}, Q_\chi\cdot \E p(C))$, with $k\le n$, apply the defining property of~$D$ to the square
	\[\begin{tikzcd}
		\emptyset \arrow[r] \arrow[d, hook]& G_{\hat\chi} \cdot \E C \arrow[d, "p"]\\
		\SD^l(\partial \Delta^{k+1}) \arrow[r,"\eta"] & Q_\chi \cdot \E p(C)
	\end{tikzcd}\]
	(for appropriate~$l$), producing a lift $\tilde \eta \in \Map(\partial \Delta^{k+1}, G_{\hat\chi}\cdot \E D)$. By definition of~$D'$, we then find a filling $\mu\in \Map(\Delta^{k+1}, G_{\hat\chi}\cdot \E D')$ of~$\tilde\eta$, whose projection $p\circ \mu$ shows~$\eta$ is null-homotopic.
\end{proof}

It is worth spelling out explicitly that the discrepancy in dimensions between items (1) and (2) of Theorem~\ref{thm:sigmaofquotient} is due to the fact that Lemma~\ref{lem:Cnlift} allows us to lift maps out of complexes of dimension at most $n$. In (2) this is applied to spheres, whereas in (1) it is applied to their fillings.

\section{Technical interlude}\label{sec:technical}

\subsection{The raising principle}\label{sec:raisingprinciple}

In using the criterion in Proposition~\ref{prop.trimmedsigma}~(2) to show that a character $\chi\colon G\to \RR$ lies in~$\TopS^n(G)$, one seeks to find, for each given $C\in \C(G)$, a larger compact set~$D$ such that for~$k\leq n-1$ the inclusion $(G\cdot \E C)_\chi\into (G\cdot\E D)_\chi$ is $\pi_k$-trivial. The goal of this brief section is to reduce such a task to finding $K\in\C(G)_{\supseteq C}$ with $\pi_k$-trivial inclusion~$G\cdot \E C \into G\cdot \E K$, and a $G$-equivariant map $\varphi \in \Map (G\cdot \E K^{(n)}, G\cdot \E K)$ that raises $\chi$-values. We establish this principle as Lemma~\ref{lem:pushfilingup}. Its main usage is in the coming sections, but we also collect some immediate dividends by deducing one of our main results, Theorem~\ref{thm:center} below.

\begin{lem}[Homotopy to a map on a subdivision]\label{lem:interpol_homotopy}
	Let $K\in \C(G)$, let~$X$ be a simplicial set, and suppose we are given simplicial maps $\varphi_0\colon X \to G\cdot \E K$ and $\varphi_1\colon \SD^m(X) \to  G\cdot \E K$ (for some $m\in \NN$). We consider the (unique) simplicial homotopy \[H\colon \SD^m(X) \times \Delta^1 \to \E G\]
	from $\varphi_0 \circ \Phi^m$ to~$\varphi_1$.
		
	Suppose there is $B\in \C(G)_{\supseteq \{1\}}$  such that for every vertex $x\in X^{(0)}$, we have $\varphi_1(x) \in \varphi_0(x) B$. Then there is $D\in \C(G)_{\supseteq K}$ such that $H$~has image contained in~$G\cdot \E D$.
\end{lem}

Of particular interest for us is the case where $X$~carries a $G$-action that is transitive on vertices, and $\varphi_0, \varphi_1$ are $G$-equivariant. Then the existence of a suitable~$B$ is automatic:  choosing any $x\in X^{(0)}$ and writing $g_0:=\varphi_0(x), g_1:=\varphi_1(x)$, it follows from $G$-equivariance that for every $y\in X$ we have $\varphi_1(y) = g_1 g_0^{-1}\varphi_0(y)$. Hence, $B:=\{1, g_1 g_0^{-1}\}$ is as required.

\begin{proof}
	Let $\sigma$ be a simplex of~$X$ and let $g\cdot \E K$ be a translate of~$\E K$ containing the simplex~$\varphi_0 (\sigma)$.
	It follows that each vertex of~$\sigma$ is sent by $\varphi_1$ to $gKB$.
	Now recall that each vertex~$v$ of~$\SD^m(\sigma)$ is within edge-distance at most~$2m$ from a vertex of~$\sigma$. 
	As $\varphi_1$ maps each such edge to an edge with label in~$K^{-1} K$, we see that 
	\[\varphi_1(v) \in g K B (K^{-1} K)^{2m}.\]
	
	Thus, by taking $D:= K B (K^{-1} K)^{2m}$, we see that $g\cdot \E D$ contains the $H$-image of $(\SD^m(\sigma) \times \Delta^1)^{(0)}$. Since $g\cdot \E D$ is a free simplicial set, it contains $H(\SD^m(\sigma) \times \Delta^1)$, and thus $G\cdot \E D$ contains the image of~$H$.
\end{proof}

	Here and in the coming sections, it will be convenient to lighten the notation by establishing the following conventions.
	Given a character~$\chi$ on~$G$ and a simplicial subset~$Y \in \E G$,
	we write
	$$\chi(Y):=\inf \{\chi(g) \mid g\in Y^{(0)}\}.$$
	Moreover, given simplicial sets $A\subseteq X$ and $\varphi \in \Map(X, \E G)$, say supported on $\SD^m(X)$, we abbreviate $\chi(\varphi(A)) := \chi(\varphi(\SD^m(A)))$, and $\chi(\varphi):=\chi(\varphi(X))$.
	
	\begin{dfn}
		Let  $X,Y$ be simplicial subsets of $\E G$, let $\chi \colon G\to \RR$ be a character, and let $\varphi \in \Map(X, Y)$ be a map, say defined on $\SD^m(X)$. We say that $\varphi$ \textbf{raises $\chi$-values} if there is $\epsilon >0$ such that for every simplex $\sigma$ of~$X$, we have
		$$\chi(\varphi(\sigma)) \ge \chi(\sigma) + \epsilon.$$
	\end{dfn}
	
	We emphasize that this inequality is a more stringent condition than requiring each vertex $v\in X^{(0)}$ to satisfy $\chi(\varphi(v))\ge \chi(v) + \epsilon$. We also need the vertices~$v$ of $\SD^m(X)$ not in~$X$, say in the subdivision $\SD^m(\sigma)$ of a simplex~$\sigma$ of~$X$, to satisfy $\chi(\varphi(v)) \ge \chi(\sigma) + \epsilon$.
	
	Maps raising $\chi$-values will play a role similar to Renz's $v$-homotopies. The following principle is reminiscent of his result establishing $v$-homotopies as witnesses for essential $(n-1)$-connectedness \cite[Satz~2.5]{Ren88}.
	
\begin{lem}[The raising principle]\label{lem:pushfilingup}
	Fix $n\in \NN$ and a character $\chi\colon G\to \RR$. Let $C\subseteq K$ be compact subsets of~$G$ such that the inclusion $G\cdot \E C \into G\cdot \E K$ is $\pi_k$-trivial for all $k \le n-1$.
	If there is a $G$-equivariant $\varphi \in \Map (G\cdot \E K^{(n)}, G\cdot \E K)$ raising $\chi$-values, then there is $D\in \C(G)_{\supseteq K}$ depending on~$\varphi$ but not on~$\chi$, such that the inclusion
	$(G\cdot \E C)_\chi \into (G\cdot \E D)_\chi$ is $\pi_k$-trivial for $k\le n-1$.
\end{lem}
\begin{proof}
	Say $\varphi$ is defined on $\SD^m(G\cdot E K^{(n)})$. Since $\varphi$~is $G$-equivariant, we may apply Lemma~\ref{lem:interpol_homotopy} with $B:=\{1,\varphi(1)\}$ to produce $D\in \C(G)_{\supseteq K}$ such that the simplicial homotopy
	\[H\colon \SD^m(G\cdot \E K^{(n)}) \times \Delta^1 \to G\cdot \E D.\]
	from the composition $\SD^m(G\cdot \E K^{(n)}) \xrightarrow{\Phi^m} G\cdot \E K^{(n)} \into G\cdot \E K$ to~$\varphi$ is well-defined.
	Note that $D$~and~$H$ are independent of~$\chi$.
	Since $\varphi$ raises $\chi$-values, we see that $H$~restricts to a homotopy
	$\SD^m((G\cdot \E K)_\chi^{(n)})\times \Delta^1 \to (G\cdot \E D)_\chi$.
	
	We will prove that the inclusion $(G\cdot \E C)_\chi \into (G\cdot \E D)_\chi$ is $\pi_k$-trivial, so let $\eta \in \Map (\partial \Delta^{k+1}, (G\cdot \E C)_\chi)$. Since $G\cdot \E C \into G\cdot \E K$ is $\pi_k$-trivial, we can use Lemma~\ref{lem:combfilling} to fill~$\eta$ with a map $\mu \in  \Map(\Delta^{k+1}, G\cdot \E K)$; we need to find a filling with image in~$(G\cdot \E D)_\chi$.
	Writing $s:= \chi(\mu)$, we see $\mu \in \Map(\Delta^{k+1}, (G\cdot \E K)_{s})$.
	
	If $s\ge 0$, there is nothing left to show.
	Otherwise, the composition
	\[\partial \Delta^{k+1} \times\Delta^1 \xrightarrow{\eta \times \id} (G\cdot \E C)^{(n)}_\chi \times \Delta^1 \xrightarrow{H}(G\cdot \E D)_\chi \]
	(where we suppress occurrences of $\SD$ and $\Phi$ to lighten notation)
	homotopes~$\eta$ to $\eta' := \varphi\circ \eta \in\Map(\partial \Delta^{k+1}, (G\cdot \E K)_\chi)$, which is filled by
	$\mu':= \varphi \circ \mu \in \Map(\Delta^{k+1}, (G\cdot \E K)_{s+\epsilon})$.
	In particular, $\chi(\mu') \ge s+\epsilon$.
	
	If still $\chi(\mu') < 0$, we may repeat the procedure, again homotoping $\eta'$ to a map~$\eta''$, with a filling~$\mu''$ satisfying $\chi(\mu'') \ge s + 2\epsilon$. After enough iterations, we eventually conclude $\eta$~is homotopic in $(G\cdot \E D)_\chi$ to a map that has a filling in $(G\cdot \E K)_\chi \subseteq (G\cdot \E D)_\chi$.
\end{proof}

\begin{thm}[Non-vanishing on the center]\label{thm:center}
	Let $n \in \NN$ and suppose $G$~is of type~$\mathrm C_n$. Then $\TopS^n(G)$~contains all characters that do not vanish on the center~$\operatorname{Z}(G)$.
\end{thm}

\begin{proof}
	Suppose a character $\chi\colon G\to \RR$ does not vanish on~$\operatorname{Z} (G)$, so there is  $t\in \operatorname{Z}(G)$ with $\chi(t) > 0$, and let $C \in \C(G)$. Since $G$~is of type~$\mathrm C_n$, there is $K \in \C(G)_{\supseteq C}$ such that the inclusion $G\cdot \E C \into G\cdot \E K$ is $\pi_k$-trivial for all $k\le n-1$.
	
	We define a simplicial map $\varphi \colon G\cdot \E K \to G\cdot \E K$ by specifying it on vertices as $g\mapsto gt$. This is clearly $G$-equivariant, but the fact that its images lies in $G\cdot \E K$ uses centrality of~$t$ in a crucial way: given a simplex $\sigma = g\cdot (x_0, \ldots, x_m)$ of~$G\cdot \E K$, with $g\in G$ and $x_0, \ldots, x_m\in K$, we have
	\[\varphi(\sigma) = g\cdot (x_0t, \ldots, x_mt) = gt(x_0, \ldots ,x_m) \in G\cdot \E K.\]
	
	Since we also have $\chi(\varphi(\sigma)) = \chi(\sigma) + \chi(t)$, Lemma~\ref{lem:pushfilingup} yields $D\in \C(G)_{\supseteq K}$ such that the inclusion $(G\cdot \E C)_\chi \into (G\cdot \E D)_\chi$ is $\pi_k$-trivial for $k\le n-1$. This shows $\chi \in \TopS^n(G)$ by Proposition~\ref{prop.trimmedsigma}.
\end{proof}

\subsection{Topologising spaces of maps}\label{sec:topologise}
Adapting results of discrete $\Sigma$-theory to the setting of locally compact groups will often require translating finiteness arguments from the classical proofs into compactness arguments.
To this end, we will now add meaningful topologies to the sets of maps between relevant simplicial sets.

Let $X$~be a topological space and $S$ a simplicial set. We topologise $\HomSS(S, \E X)$  using the canonical bijections
\[\HomSS(S, \E X) \cong \HomSet(S^{(0)}, X) \cong \prod_{S^{(0)}} X.\]
Given a sub-simplicial set~$Y \subseteq \E X$, we endow $\HomSS(S,Y)$ with the sub\-space topology induced from $\HomSS(S, \E X)$.

\begin{lem}[Closed and open $G$-subspaces]\label{lem:closedandopensubspaces}
	Let $X$ be a locally compact Hausdorff $G$-space with proper action. Then:
	\begin{enumerate}
		\item For every compact $C\subseteq X$, the subspace $G C \subseteq X$ is closed. 
		\item For every open  $U\subseteq X$, the subspace $G U \subseteq X$ is open.
	\end{enumerate}
\end{lem}

\begin{proof}
	(1)
	Let $x_\infty$ be a point of the closure~$\overline{G C}$, so $x_\infty$~is the limit of a net $(g_i c_i)_{i\in I}$ with $g_i\in G$ and $c_i \in C$. Since $C$~is compact, we may, after passing to a cofinal subnet, assume $(c_i)_{i\in I}$ converges to some $c_\infty \in C$ (unique, since $X$~is Hausdorff). Thus in $X\times X$ we have $(c_i, g_i c_i) \to (c_\infty, x_\infty)$.
	
	Choose a compact neighborhood~$K$ of~$(c_\infty,x_\infty)$ in $X\times X$. Properness of the map $f\colon G\times X \to X\times X$ sending $(g,x)\mapsto(x,g  x)$ ensures that $\tilde K := f^{-1}(K)$ is a compact subset of $G\times X$, and $\tilde K$ contains a tail of the net $(g_i, c_i)_{i\in I}$. Thus $\tilde K$ contains a cluster point $(g_\infty, c_\infty)$, which $f$~maps to $(c_\infty, x_\infty)$ by continuity. Hence $x_\infty = g_\infty c_\infty \in G C$ and we conclude $\overline {G C} \subseteq G C$, as required.
	
	(2)
	Openness of $G U \subseteq X$ is immediate from its expression as the union of open sets $G U= \bigcup_{g\in G} g U$.
\end{proof}

\begin{lem}[Closed and open Hom-subspaces]
	Let $S$ be a simplicial set and $X$~a locally compact Hausdorff $G$-space with proper action. Then:
	\begin{enumerate}
		\item For every $C\in \C(X)$, the subspace $\HomSS(S, G\cdot \E C) \subseteq \HomSS(S, \E X)$ is closed. 
		\item If $S$~is finite and $U\subseteq X$ is open, then the subspace $\HomSS(S, G\cdot \E U) \subseteq \HomSS(S, \E X)$ is open.
	\end{enumerate}
\end{lem}
\begin{proof}
	(1) Recall that a simplicial map $\eta\colon  S \to \E X$ is determined by its restriction to the $0$-skeleton~$S^{(0)}$ -- specifically, the image of a simplex~$\sigma$ with vertices $v_0, \ldots, v_k$ is~$(\eta(v_0), \ldots, \eta(v_k))$. Thus, the condition of $\eta(\sigma)$~being a simplex of~$G \cdot \E C$ means that there are $g\in G, c_0, \ldots, c_k \in C$ with $\eta(v_i) = g\cdot c_i$ for each $i\in \{0, \ldots, k\}$.
	Regarding~$\eta$ as an element of~$\prod_{S^{(0)}} X$, this translates to $$\eta \in \big(G\cdot \prod_{\sigma^{(0)}} C\big) \times  \prod_{S^{(0)} \setminus \sigma^{(0)}} X,$$ where the left factor is the union of $G$-translates with respect to the diagonal action on~$\prod_{\sigma^{(0)}} X$.
	
	We first claim this product is closed; equivalently, that $G\cdot \prod_{\sigma^{(0)}} C$ is closed in $\prod_{\sigma^{(0)}} X$. This is a consequence of Lemma~\ref{lem:closedandopensubspaces}~(1), because, as one can easily verify, the diagonal action of~$G$ on $\prod_{\sigma^{(0)}} X$ is again proper, this product is locally compact and Hausdorff, and $\prod_{\sigma^{(0)}} C$ is compact.
	
	We have thus shown that for each simplex~$\sigma$ of~$S$, the subset of $\HomSS(S,\E X)$ mapping~$\sigma$ to~$G\cdot \E C$ is closed. $\HomSS(S, G\cdot \E C)$ is the intersection of the closed subsets given by all~$\sigma$, and thus closed.
	
	(2) As in the previous part, for each simplex~$\sigma$ of~$S$, the subset of~$\HomSS(S,\E X)$ mapping~$\sigma$ to $G\cdot \E U$~is the product $\left(G\cdot \prod_{\sigma^{(0)}} U\right) \times  \prod_{S^{(0)} \setminus \sigma^{(0)}} X$. Openness of this set is equivalent to that of $G\cdot \prod_{\sigma^{(0)}} U \subseteq \prod_{\sigma^{(0)}} X$, which in turn follows from Lemma~\ref{lem:closedandopensubspaces}~(2).
	
	Now, the subset $\HomSS(S, G\cdot \E U) \subseteq \HomSS(S, \E X)$ is an intersection of open subsets taken over all simplices of~$S$. But this (possibly infinite) intersection is the same as the one taken only over the nondegenerate simplices, of which there are only finitely many by assumption. Thus $\HomSS(S, G\cdot \E U)$ is open.
\end{proof}
\section{Openness and connecting sets}\label{sec:stability}
The main goal of this section is to establish the following fact:

\begin{thm}[Openness]\label{thm:openness}
	For every $n\in \NN$, if $G$~is of type~$\mathrm C_n$, then $\TopS^n(G)$ is the cone at~$0$ over an open subset of $\TopHom(G,\RR)\setminus\{0\}$.
\end{thm}

The emphasis here is on the word ``open''; everything else follows directly from Proposition~\ref{prop.zerocharacter}.

We now introduce one of the most important ingredients of the proof of Theorem~\ref{thm:openness}, whose properties will also have ramifications in Section~\ref{sec:coabelian}. 

\begin{dfn}\label{def:connectingset}
Let $\chi\colon G \to \RR$ be a character and let~$K\in \C(G)$.
\begin{itemize}
	\item We say~$K$ is a \textbf{$0$-connecting set for~$\chi$} if $1 \in K$.
	\item For $n\ge 1$, we say~$K$ is an~\textbf{$n$-connecting set for~$\chi$} if there is an $(n-1)$-connecting set~$K_0 \subseteq \mathring K$ for~$\chi$ such that the inclusion $(G\cdot \E K_0)_\chi \into (G\cdot \E \mathring K)_\chi$ is $\pi_{k}$-trivial for all $k\le n-1$.
\end{itemize}
\end{dfn}

Note that every $n$-connecting set for a character $\chi$ is $n$-connecting for the zero character, and that if $K$~is $n$-connecting for~$\chi$, then so is every $K'\in \C(G)_{\supseteq K}$.

For the moment, this definition might seem obscure. We will see later (Proposition~\ref{prop:alonsoandsigma}) that the existence of an $n$-connecting set for a character~$\chi$ is equivalent to $\chi\in \TopS^n(G)$. Before that, however, we establish the 
main result that justifies our interest in connecting sets.

\begin{prop}[Fundamental property of connecting sets]\label{prop:alonso_main}
	Let $n\in \NN$, let $\chi\colon G \to \RR$ be a character, and let $K$~be an $n$-connecting set for~$\chi$.
	Then for every~$C \in \C(G)$, there is a $G$-equivariant $\varphi \in \Map(G\cdot \E C^{(n)},  G\cdot \E K)$.
	
	Moreover, if $\chi \neq 0$, then $\varphi$~may be chosen so that there is $\epsilon >0$ and a neighborhood $V\subseteq \TopHom(G,\RR) \setminus \{0\}$ of~$\chi$ such that for every $\psi \in V$ and every simplex~$\sigma$ of $G\cdot \E C^{(n)}$, we have
		\[\psi(\varphi(\sigma)) \ge \psi(\sigma) + \epsilon.\]
\end{prop}

\begin{proof}
	\setcounter{claim}{0}
	The proof will focus on the $\chi \neq 0$ case. For $\chi = 0$, only the first assertion is claimed and the argument is much simpler.
	[We will point out in parenthetical comments like this one the changes in this easier case.]
	
	We proceed by induction over~$n$. If $n=0$, choose $t\in G$ with $\chi(t) >0$ and define~$\varphi$ by $g\mapsto gt$, which is clearly $G$-equivariant [if $\chi=0$ simply define $\varphi$ as the identity].
	One can then take $\epsilon := \tfrac {\chi(t)} 2$ and $V:=\{\psi\colon G\to \RR \mid \psi(t) >\epsilon \}$.
	
	Now let $n\ge 1$, with $K$~and~$C$ as in the statement. By definition, there is an $(n-1)$-connecting set $K_0 \subseteq \mathring K$ for~$\chi$ such that $(G\cdot \E K_0)_\chi \into (G\cdot \E \mathring K)_\chi$ is $\pi_{n-1}$-trivial. Let $\varphi_0 \in \Map(G\cdot \E C^{(n-1)}, G\cdot \E K_0)$ be as given by the inductive hypothesis, with corresponding constant~$\epsilon_0$, and say $\varphi_0$~is defined on the subdivision $\SD^{m_0}(G\cdot \E C^{(n-1)})$.
	We will find a suitable $m\in \NN$ and define~$\varphi$ on $\SD^{m_0+m} (G\cdot \E C^{(n)})$, so that it extends {$\varphi_0\circ \Phi^m\colon \SD^{m_0+m}(G\cdot \E C^{(n-1)}) \to G\cdot \E K$}, where $\Phi$~is the natural map from the semisimplicial approximation theorem (Theorem~\ref{thm:approx}).
	
	Write $\mathcal S := \SD^{m_0} (\partial \Delta ^n)$. For each nondegenerate $n$-simplex~$\sigma$ of~$G\cdot \E C$, the given~$\varphi_0$ specifies a map $\eta_\sigma \colon \mathcal S \to G\cdot \E K_0$, and our goal is to
	produce a map $\mu_\sigma \colon \SD^{m_0 + m} (\Delta^n) \to  G \cdot \E K$ extending $\eta_\sigma \circ \Phi^{m}$. These $\mu_\sigma$ will then assemble into the desired~$\varphi$. 
	We also need to control the drop in $\psi$-value of these fillings for characters~$\psi$ near~$\chi$, as required by the second part of the proposition.
	
	Writing $t:=\varphi_0(1)$, we first observe that $G$-equivariance implies that $\varphi_0$~is given on~$G\cdot \E C^{(0)}$ by $g\mapsto gt$. Therefore, the inequality in the theorem, which we seek to establish, can be rephrased in terms of~$\mu_\sigma$ as
	\[\psi(\mu_\sigma) \ge \min_{0\le i\le n}\psi(\mu_\sigma(e_i)t^{-1}) + \epsilon.\tag{$\dagger$}\]
	
	The maps $\eta_\sigma$ with $\sigma$ the nondegenerate $n$-simplices of~$G\cdot \E C$ form a $G$-set, and we will represent each orbit by the translate sending the $0$-th vertex~$e_0$ of~$\Delta^n$ to~$1$:
	\[\eta_\sigma^\bullet := \eta_\sigma(e_0)^{-1} \cdot \eta_\sigma.\]
	We collect these representatives in the transversal
	\[\mathcal M := \{ \eta_\sigma^{\bullet} \colon  \mathcal S \to  G\cdot \E K_0 \mid \text{$\sigma$~is a nondegenerate $n$-simplex of $G \cdot \E C$} \}.\]

	\begin{claim}~\label{claim:compactproblems}
		The closure~$\overline{\mathcal M}$ of~$\mathcal M$ in $\HomSS (\mathcal S, G\cdot \E K_0)$~is compact, and every $\eta \in \overline{\mathcal M}$ satisfies
		\[\chi(\eta) \ge \min_{0\le i\le n} \chi (\eta(e_i)t^{-1}) + \epsilon_0 \tag{$\star$}.\]
		[This inequality is suppressed if $\chi = 0$.]
	\end{claim}
	
	\begin{proof}
		We first show~$\overline{\mathcal M}$ is compact.
		By Lemma~\ref{lem:closedandopensubspaces}~(1), the inclusion 
		\[\HomSS(\mathcal S, G\cdot \E K_0) \subseteq \HomSS(\mathcal S, \E G) = \prod_{\mathcal S^{(0)}} G\]
		is closed, so it is enough to show $\mathcal M$~has compact closure in~$\prod_{\mathcal S^{(0)}} G$. And since this is a Hausdorff space, that follows once we find a compact~$\mathcal B\subseteq \prod_{\mathcal S^{(0)}} G$ containing~$\mathcal M$. To find such~$\mathcal B$, we consider two cases.
		
		If $n=1$, then $\sigma$~is an edge labeled by some~$g\in C^{-1}C$, so $\eta_\sigma^\bullet$~consists of the two vertices~$1$ and~$t^{-1}gt$. Thus, $\mathcal M \subseteq \{1\} \times t^{-1} C^{-1}C t =:\mathcal B$.
		
		When $n\ge2$, we use the fact that $\partial \Delta^{n}$ is connected. In particular, each vertex $v \in \mathcal S^{(0)}$ is within some finite edge-distance $d_v\in \NN$ from the vertex~$e_0$ in the $1$-skeleton~$\mathcal S^{(1)}$. Now, $\eta_\sigma^\bullet$~maps~$e_0$ to~$1$ and takes each edge of~$\mathcal S$ to an edge labeled by $K_0^{-1}K_0$. Thus, for each vertex~$v$ we see $\eta_\sigma^\bullet(v)\in (K_0^{-1}K_0)^{d_v}$. Hence
		\[\mathcal M \subseteq \prod_{v\in \mathcal S^{(0)}} (K_0^{-1}K_0)^{d_v} =: \mathcal B.\]
		
		We are only left to prove the inequality~$(\star)$, and we focus first on the elements~$\eta_\sigma^\bullet$ of~$\mathcal M$.
		For each $(n-1)$-simplex~$\tau$ of~$\partial \sigma$, we have a map $\mu_\tau\colon \SD^{m_0}(\Delta^{n-1})\to G\cdot \E K_0$ specified by the restriction of $\varphi_0$ to~$\SD^{m_0}(\tau)$,
		and by induction,
		\[\chi(\mu_\tau) \ge \min_{0\le i\le n-1} \chi(\mu_\tau(e_i)t^{-1}) + \epsilon_0.\]
		Therefore,
		\[\chi(\eta_\sigma)  = \min_{\tau}\chi(\mu_\tau) \ge \min_\tau \big(\min_{0 \le i \le n-1} \chi(\mu_\tau(e_i)t^{-1})\big) + \epsilon_0 = \min_{0\le i \le n}\chi( \eta_\sigma(e_i)t^{-1}) + \epsilon_0.\]
		The inequality for $\eta_\sigma^\bullet$ is then a special case.
		
		To see that $(\star)$~is satisfied for all $\eta \in \overline{\mathcal M}$, observe that the set~$Q \subseteq  \HomSS(\mathcal S, G\cdot \E K_0)$ for which it holds is the preimage of the closed set ${[\epsilon_0, +\infty[}\subset \RR$ by the continuous map
		\begin{align*}
			\HomSS(\mathcal S, G\cdot \E K_0) & \to \RR\\
			\eta &\mapsto \min_{v\in \mathcal S^{(0)}}\chi(\eta(v)) - \min_{0\le i \le n} \chi(\eta(e_i)t^{-1}).
		\end{align*}
		Hence, $Q$~is a closed set containing~$\mathcal M$, and thus also~$\overline {\mathcal M}$.
	\end{proof}
	
	Intuitively, Claim~\ref{claim:compactproblems} says that to construct the maps~$\mu_\sigma$, we need only solve a compact set of filling problems. Claim~\ref{claim:copyhomework} below will complement it by establishing that once a solution to one such filling problem is found, it can be re-used on nearby problems. Ultimately, this will result in an open cover of the compact set~$\overline{\mathcal M}$, and a finite sub-cover of it will allow us to produce the desired (finite)~$m$.
		
	The idea of ``re-using the solution to a filling problem on a different filling problem'' will be expressed through the following notation: For each $m'\in\NN$, we consider the subdivisions
	\[ \mathcal D_{m'}:=\SD^{m_0 + m'}(\Delta^n), \qquad \mathcal S_{m'} := \SD^{m'}(\mathcal S) = \SD^{m_0+m'}(\partial \Delta^n).\]
	Given maps $\eta \colon \mathcal S \to \E G$ and $\mu \colon \mathcal D_{m'}\to \E G$, we define the simplicial map $\eta^\mu \colon \mathcal D_{m'} \to \E G$ by specifying it on each vertex~$v\in  \mathcal D_{m'}^{(0)}$ as
	\[\eta^\mu(v) := \begin{cases}
		\eta\circ \Phi^{m'}(v) & \text{if $v\in \mathcal S_{m'}^{(0)}$,}\\
		\mu(v) & \text{otherwise.}
	\end{cases}\]
	In other words, for each fixed~$\mu$, the assignment $\eta \mapsto \eta^\mu$ is the composition
	\[\HomSS(\mathcal S, \E G) \xrightarrow{- \circ \Phi^{m'}} \HomSS(\mathcal S_{m'}, \E G) \into \HomSS(\mathcal D_{m'}, \E G),\]
	where the inclusion on the right is determined by the value of~$\mu$ on~$\mathcal D_{m'}^{(0)} \setminus \mathcal S_{m'}^{(0)}$. Regarding each of these spaces as products of copies of~$G$ indexed by the vertices of the source simplicial sets, it becomes clear that this composition is continuous.
	
	Let us now find \emph{some} filling for each~$\alpha \in \overline{\mathcal M}$. Choose $g\in G$ satisfying $\chi(g) = \chi(\alpha)$, so $\alpha$~has image in~$g \cdot (G\cdot \E K_0)_\chi$.
	By choice of~$K$ (and Lemma~\ref{lem:combfilling}), there are $m_\alpha \in \NN$ and a map $\mu_\alpha \colon \mathcal D_{m_\alpha} \to g\cdot (G\cdot \E \mathring K)_\chi$ extending $\alpha \circ \Phi^{m_\alpha} \colon \mathcal S_{m_\alpha} \to g\cdot (G\cdot \E K_0)_\chi$.
	We have
	\begin{align*}
		\chi(\mu_\alpha) & = \chi(\alpha\circ \Phi^{m_\alpha}) \tag{Construction of~$\mu_\alpha$}\\
		& = \chi(\alpha)\\
		& \ge \min_{0\le i\le n} \chi(\alpha(e_i)t^{-1}) + \epsilon_0. \tag{Claim~\ref{claim:compactproblems}}
	\end{align*}
	We also fix a choice of vertex $e_\alpha \in (\Delta^n)^{(0)} =\{e_0, \ldots, e_n\}$ realizing this minimum, so
	\[\chi(\mu_\alpha) \ge \chi(\mu_\alpha(e_\alpha)t^{-1}) + \epsilon_0.\]
	[If $\chi=0$, this lower bound on $\chi(\mu_\alpha)$ is suppressed.]

	\begin{claim}\label{claim:copyhomework}
		Each $\alpha \in \overline{\mathcal M}$ has an open neighborhood~$U_\alpha \subseteq  \HomSS(\mathcal S, G\cdot \E K_0)$ such that for every $\eta \in U_\alpha$, the map $\eta^{\mu_\alpha}\colon \mathcal D_{m_\alpha} \to \E G$ has image in~$G\cdot \E \mathring K$ and
		\[\chi(\eta^{\mu_\alpha}) > \chi(\eta(e_\alpha)t^{-1}) + \tfrac {\epsilon_0} 2.\]
		[If $\chi=0$, this inequality is suppressed.]
	\end{claim}
	
	\begin{proof}
		By Lemma~\ref{lem:closedandopensubspaces}~(2), the inclusion $\HomSS(\mathcal D_{m_\alpha}, G\cdot \E \mathring K) \subseteq \HomSS(\mathcal D_{m_\alpha}, \E G)$ is open. The subset of maps~$\mu$ satisfying $\chi(\mu) > \chi(\mu(e_\alpha)t^{-1}) +  \frac{\epsilon_0}2$ is also open, because it is the preimage of~$]{\frac{\epsilon_0}2, +\infty} [$ under the continuous map
		\begin{align*}
			\HomSS(\mathcal D_{m_\alpha}, \E G)& \to \RR \\
			\mu & \mapsto \min_{v\in \mathcal D_{m_\alpha}^{(0)}} \chi(\mu(v))- \chi(\mu(e_\alpha)t^{-1}).
		\end{align*}
		Hence, there is an open neighborhood~$U$ of~$\mu_\alpha$ in $\HomSS(\mathcal D_{m_\alpha}, \E G)$ that is contained in~$\HomSS(\mathcal D_{m_\alpha}, G\cdot \E \mathring K)$, and whose elements~$\mu$ satisfy $\chi(\mu) > \chi(\mu(e_\alpha)t^{-1}) + \frac {\epsilon_0}2$.
		
		We take $U_\alpha$~to be the preimage of~$U$ under the map
		\begin{align*}
			\HomSS(\mathcal S, G\cdot \E K_0) &\to \HomSS(\mathcal D_{m_\alpha}, G\cdot \E \mathring K)\\
			\eta &\mapsto \eta^{\mu_\alpha},	
		\end{align*}
		which we have already observed to be continuous. Thus $U_\alpha$~is open and satisfies all required properties.
	\end{proof}
	
	The sets $U_\alpha$ now form an open cover of~$\overline{\mathcal M}$, which is compact by Claim~\ref{claim:compactproblems}, so
	there is a finite subset~$F\subseteq \overline{\mathcal M}$ with $\overline{\mathcal M} \subseteq \bigcup_{\alpha\in F}U_\alpha$ and we may set
	$m := \max\{m_\alpha \mid \alpha \in F\}$.
	For each $\eta\in \mathcal M$, choose $\alpha \in F$ with $\eta \in U_\alpha$ and define, for every non\-degenerate $n$-simplex~$\sigma$ with $\eta_\sigma^\bullet = \eta$, the map $\mu_\sigma \colon \mathcal D_m \to G\cdot \E K$ to be the appropriately translated filling of~$\eta_\sigma^\bullet$ given by $\mu_\alpha$:
	\[\mu_\sigma := \eta_\sigma(e_0) \cdot  ((\eta_\sigma^\bullet)^{\mu_\alpha}\circ \Phi^{m-m_\alpha}).\]
	This is a $G$-equivariant definition extending the maps~$\eta_\sigma \circ \Phi^m$. We assemble~$\varphi$ out of the $\mu_\sigma$, thus extending $\varphi_0 \circ \Phi^m$. [This concludes the proof of Proposition~\ref{prop:alonso_main} in the case $\chi=0$.]

	We are left to find the neighborhood~$V$ of~$\chi$ required by the second part of the proposition. 
	For each $\alpha \in F$, let $\overline{U_\alpha}$~be the closure of~$U_\alpha$ in~$\overline{\mathcal M}$ (in particular, $\overline{U_\alpha}$~is compact).
	We shall take $\epsilon:= \frac {\epsilon_0} 3$ and
	\[V := \{\psi\in \TopHom(G,\RR) \mid \forall\alpha \in F \quad \forall\eta \in\overline{U_\alpha}: \quad \psi(\eta^{\mu_\alpha}) > \psi(\eta(e_\alpha)t^{-1}) + \epsilon\}.\]
	
	To see $V$~is open in~$\TopHom(G,\RR)$, we rewrite its defining condition:
	\[\forall \alpha \in F \quad \forall v \in \mathcal D_{m_\alpha}^{(0)} \quad \forall \eta\in\overline{U_\alpha} :\quad  \psi\left(\eta^{\mu_\alpha}(v) \, \eta(e_\alpha)^{-1} \,t \right) > \epsilon.\]
	Thus $V$~is the finite intersection
	$V = \bigcap_{\alpha \in F} \bigcap_{v\in \mathcal D_{m_\alpha}^{(0)}} V_{\alpha, v}$,
	where
	\[V_{\alpha,v} := \left\{\psi \in \TopHom(G, \RR) \mid \forall \eta \in \overline{U_\alpha}: \quad  \psi\left(\eta^{\mu_\alpha}(v) \, \eta(e_\alpha)^{-1} \,t \right) > \epsilon \right\},\]
	so it suffices we show each $V_{\alpha,v}$ is open.
	Consider the continuous map
	\begin{align*}
		\overline{U_\alpha} &\to G\\
		\eta & \mapsto \eta^{\mu_\alpha}(v)\, \eta(e_\alpha)^{-1}\, t.
	\end{align*}
	Its image $Q_{\alpha,v}$ is compact and allows us to express~$V_{\alpha,v}$ as
	\[V_{\alpha,v}= \{\psi\in\TopHom(G,\RR) \mid \psi(Q_{\alpha,v}) \in {]\epsilon,+\infty[} \}.\]
	Hence $V_{\alpha,v}$~is open by definition of the compact-open topology.

	To see $\chi \in V$, recall that for each $\alpha\in F$ and $\eta \in U_\alpha$ we have
	\[\chi(\eta^{\mu_\alpha}) - \chi(\eta(e_\alpha)t^{-1}) > \tfrac{\epsilon_0}2.\]
	Since the map
	\begin{align*}
		\overline{\mathcal M} &\to \RR\\
		\eta & \mapsto \min_{v \in \mathcal D_{m_\alpha}^{(0)}} \chi(\eta^{\mu_\alpha}(v)) - \chi(\eta(e_\alpha)t^{-1}) 
	\end{align*}
	is continuous, it follows that for $\eta \in \overline{U_\alpha}$ we still have
	\[\chi(\eta^{\mu_\alpha})  - \chi(\eta(e_\alpha)t^{-1}) \ge \tfrac{\epsilon_0}2 > \epsilon,\]
	and hence $\chi \in V$.
	
	Finally, inequality $(\dagger)$, required for every nondegenerate $n$-simplex~$\sigma$ of $G\cdot \E C$ and every $\psi \in V$, may be expressed in terms of the normalized maps:
	\[\psi((\eta_\sigma^\bullet)^{\mu_\alpha}) \ge \min_{0\le i \le n} \psi(\eta_\sigma^\bullet(e_i)t^{-1}) + \epsilon,\]
	where $\alpha \in F$ indexes the open set~$U_\alpha$ used to cover $\eta_\sigma^\bullet$ when defining~$\mu_\sigma$.
	But the definition of~$V$ actually ensures the stricter condition
	\[\psi((\eta_\sigma^\bullet)^{\mu_\alpha}) >  \psi(\eta_\sigma^\bullet(e_\alpha)t^{-1}) + \epsilon.\qedhere \]
\end{proof}

We now establish the three $\TopS^n$-criteria stated in Theorem~\ref{thm:sigmacriteria_intro} of the Introduction. The first is the long-promised refinement of Proposition~\ref{prop:doublefiltration} and Corollary~\ref{cor:boundeddrop}.

\begin{prop}[Testing on connecting sets for~$0$]\label{prop:alonsocriterion}
	Let $\chi\colon G \to \RR$ be a character, let $n\in \NN$, and let $K$~be an $(n-1)$-connecting set for the zero character. Then $\chi \in \TopS^n(G)$ if and only if there are $K^+\in \C(G)_{\supseteq K}$ and $s\le 0$ such that the inclusion
	\[(G\cdot \E K)^{(n-1)}_\chi \into (G\cdot \E K^+)_s\]
	is $\pi_k$-trivial for every $k \le n-1$. 
\end{prop}

\begin{proof}
	The implication ``$\Rightarrow$'' is immediate from Proposition~\ref{prop:doublefiltration}, so we focus on ``$\Leftarrow$''.
	
	By Corollary~\ref{cor:boundeddrop}, it suffices to show that for each $C\in\C(G)$ there are $D\in \C(G)_{\supseteq C}$ and $r\le 0$ such that the inclusion
	\[(G\cdot \E C)^{(n-1)}_\chi \into (G\cdot \E D)_r\]
	is $\pi_k$-trivial for all $k \le n-1$.
		
	Let $\varphi \in \Map(G\cdot \E C^{(n-1)}, G\cdot \E K)$ be as given by applying Proposition~\ref{prop:alonso_main} to the zero character, and suppose $\varphi$~is defined on $\SD^m(G\cdot \E C^{(n-1)})$.
	We wish to apply Lemma~\ref{lem:interpol_homotopy} with $\varphi_0$~the inclusion $G\cdot \E C^{(n-1)}\into G\cdot \E C$ and $\varphi_1 = \varphi$. As both maps are $G$-equivariant, we may take $B:=\{1,\varphi(1)\}$. 
	The lemma then yields $D\in \C(G)_{\supseteq C}$ such that the simplicial homotopy
	\[H \colon \SD^m (G\cdot \E C^{(n-1)}) \times \Delta^1 \to G\cdot \E D\]
	from $\varphi_0 \circ \Phi^m$ to~$\varphi$ is well-defined. Without loss of generality we may assume $K^+\subseteq D$.
	
	For each simplex $\sigma$ of $G\cdot \E C^{(n-1)}$ and vertex~$v$ of $\SD^m(\sigma)$, we see $H$~maps the edge~$\{v\}\times \Delta^1$ to $(\Phi^m(v), \varphi(v))$, which is an edge connecting a vertex of~$\sigma$ to~$\varphi(v)$. Moreover, this edge is labeled by an element of~$D^{-1}D$. Therefore, we have $\chi(\varphi(v)) \ge \chi(\sigma) + \min(\chi(D^{-1}D))$. Hence, setting $s':=\min(\chi(D^{-1}D))$,
	we see that $H$~restricts to a homotopy
	\[\SD^m((G\cdot \E C)_\chi^{(n-1)}) \times \Delta^1 \to (G\cdot \E D)_{s'}.\]
	
	Now, the assumption on $K^+\subseteq D$~and~$s$ implies that the translated inclusion
	\[(G\cdot\E K)_{s'}^{(n-1)} \into (G\cdot \E D)_{s+s'}\]
	is $\pi_k$-trivial for $k\le n-1$. Thus, writing $r:=s+s'$, we see that the composition
	\[\SD^m((G\cdot \E C)_\chi^{(n-1)}) \xrightarrow{\varphi} (G\cdot \E K)_{s'}^{(n-1)} \into (G\cdot \E D)_r\]
	is $\pi_k$-trivial as well. On the other hand, $H$~exhibits it as homotopic to
	\[\SD^m((G\cdot \E C)_\chi^{(n-1)})\xrightarrow{\Phi^m} (G\cdot \E C)_\chi^{(n-1)} \into (G\cdot \E D)_r.\]
	Since $|\Phi^m|$~is homotopic to a homeomorphism, the desired inclusion $(G\cdot \E C)_\chi^{(n-1)} \into (G\cdot \E D)_r$ is $\pi_k$-trivial.
\end{proof}

\begin{prop}[$n$-connecting sets and $\TopS^n$]\label{prop:alonsoandsigma}
	A character $\chi\colon G\to \RR$ lies in~$\TopS^n(G)$ if and only if there is an $n$-connecting set for~$\chi$.
\end{prop}
\begin{proof}
	($\Rightarrow$) For $n=0$ take $K = \{1\}$.
	For $n\ge 1$, since $\TopS^n(G) \subseteq \TopS^{n-1}(G)$, we may inductively assume the existence of an $(n-1)$-connecting set~$K_0$ for~$\chi$. Since $\chi \in \TopS^n(G)$, we know by Proposition~\ref{prop.trimmedsigma} that there is $D \in \C(G)_{\supseteq K_0}$ such that the inclusion $(G\cdot \E K_0)_\chi \into (G\cdot \E D)_\chi$ is $\pi_{k}$-trivial for all $k\le n-1$. The fact that $\TopS^n(G) \neq \emptyset$ implies that $G$~is compactly generated, so by Lemma~\ref{lem.compactexhaustion} some power $K$~of a  compact generating set satisfies $D\subseteq \mathring K$. Thus $K$~is an $n$-connecting set for~$\chi$.
	
	($\Leftarrow$) If $K$~is an $n$-connecting set for~$\chi$, then by definition there is an $(n-1)$-connecting set~$K_0$ for~$\chi$ such that the inclusion $(G\cdot \E K_0)_\chi \into (G\cdot \E K)_\chi$ is $\pi_{k}$-trivial for $k\le n-1$. Because $K_0$~is, in particular, an $(n-1)$-connecting set for the zero character, Proposition~\ref{prop:alonsocriterion} implies $\chi \in \TopS^n(G)$.
\end{proof}

 The third criterion is an analogue of a theorem of Renz \cite[Satz~2.6]{Ren88}.

\begin{prop}[$\TopS^n$ via $\chi$-value raising maps]\label{prop:theraisingcrit}
	Let $\chi\colon G \to \RR$ be a nonzero character and let $K$~be an $n$-connecting set for the zero character. Then $\chi \in\TopS^n(G)$ if and only if there is a $G$-equivariant $\varphi\in \Map(G\cdot \E K^{(n)}, G\cdot \E K)$ raising $\chi$-values.
\end{prop}

\begin{proof}
	($\Leftarrow$) By definition of~$K$, there is an $(n-1)$-connecting set $C\subseteq K$ for the zero character such that the inclusion $G\cdot \E C \into G\cdot \E K$ is $\pi_k$-trivial for all $k\le n-1$. Together with the map~$\varphi$, the raising principle (Lemma~\ref{lem:pushfilingup}) finds $D\in \C(G)_{\supseteq K}$ such that the inclusion $(G\cdot \E C)_\chi \into (G\cdot \E D)_\chi$ is $\pi_k$-trivial. From Proposition~\ref{prop:alonsocriterion}, we deduce $\chi \in \TopS^n(G)$.
	
	($\Rightarrow$) If $\chi \in \TopS^n(G)$, Proposition~\ref{prop:alonsoandsigma} provides an $n$-connecting set $K_+$ for~$\chi$, which we may enlarge to contain~$K$. We apply Proposition~\ref{prop:alonso_main} to the zero character to find a $G$-equivariant map $\varphi_0 \in \Map(G\cdot \E K_+^{(n)}, G\cdot \E K)$. 
	We may assume that $\varphi_0(1)=1$ by acting with $\varphi_0(1)^{-1}$ if necessary, so $\varphi_0$~is the identity on vertices (in fact this is satisfied by the construction provided in that proof). Moreover, let $m\in \NN$ be such that $\varphi_0$~is defined on~$\SD^m(G\cdot \E K_+^{(n)})$, and recall that each vertex~$v$ of $\SD^m(G\cdot \E K_+^{(n)})$ is within edge-distance~$2m$ from a vertex of~$G\cdot \E K_+^{(n)}$. As $\varphi_0$ maps these edges to edges with labels in $K^{-1}K$, we see that for each such~$v$, there is $g\in G$ such that  $\varphi_0(v) \in \varphi(g) (K^{-1}K)^{2m}$. Therefore, setting $s:= 2m\min(\chi(K^{-1}K))\le 0$, it follows that $\varphi_0$ drops $\chi$-values by at most~$s$, that is, it restricts to an element of $\Map((G\cdot \E K_+)^{(n)}_\chi, (G\cdot \E K)_s)$.
	
	Let us once again apply Proposition~\ref{prop:alonso_main}, this time to the character~$\chi$, to find a $G$-equivariant map~$\varphi\in \Map (G\cdot \E K_+^{(n)}, G\cdot \E K_+^{(n)})$ that raises $\chi$-values, say by $\epsilon>0$. If $k\in \NN$ is such that $k\epsilon + s >0$, then the composite $\varphi_0 \circ \varphi^k$, when restricted to the appropriate subdivision of $(G\cdot \E K)^{(n)}_\chi$, raises $\chi$-values by $k\epsilon + s$.
\end{proof}

Propositions \ref{prop:alonsocriterion}, \ref{prop:alonsoandsigma} and \ref{prop:theraisingcrit} form a useful toolkit for showing that a character lies in~$\TopS^n$, or for extracting consequences thereof. With these criteria in hand, we are equipped to prove the main result of this section.

\begin{proof}[Proof of Theorem~\ref{thm:openness}]
	We know $\TopS^n(G)$~is a cone at~$0$, so we are left to show that every $\chi \in \TopS^n(G) \setminus \{0\}$ has an open neighborhood contained in~$\TopS^n(G)\setminus\{0\}$.
	By Proposition~\ref{prop:alonsoandsigma}, there is an $n$-connecting set~$K$ for~$\chi$.  Proposition~\ref{prop:alonso_main} yields a neighborhood~$V\subseteq \TopHom(G,\RR)\setminus\{0\}$ of~$\chi$ and a map $\varphi\in \Map(G\cdot \E K^{(n)}, G\cdot \E K)$ that is $\psi$-value-raising for every $\psi \in V$. As $K$~is also an $n$-connecting set for the zero character, Proposition~\ref{prop:theraisingcrit} shows $V\subseteq \TopS^n(G)$.
\end{proof}

We finish this section with a pair of results that will be used in proving Proposition~\ref{prop:radialalonso} below and, ultimately, Theorem~\ref{thm:coabelian}. The first is a slight refinement of Proposition~\ref{prop:theraisingcrit}.

\begin{prop}[Stable essential triviality]\label{prop:stabletriviality}
	Let $n\in\NN$, let $\chi \in \TopS^n(G) \setminus \{0\}$, and let $C\in\C(G)$. Then there are a neighborhood~$V\subseteq \TopHom(G,\RR) \setminus \{0\}$ of~$\chi$ and $D\in \C(G)_{\supseteq C}$ such that for every $\psi \in V$, the inclusion
	\[(G\cdot \E C)_\psi \into (G\cdot \E D)_\psi\]
	is $\pi_k$-trivial for all $k \le n-1$.
\end{prop}

\begin{proof}		
	First use that $G$~is of type $\mathrm C_n$ to find $K \in \C(G)_{\supseteq C}$ such that the inclusion $G\cdot \E C \into G\cdot \E K$ is $\pi_k$-trivial. In particular, since $\chi \in \TopS^n(G)$ we may assume $K$~is large enough to be an $n$-connecting set for~$\chi$.
	
	Applying Proposition~\ref{prop:alonso_main}, we produce a $G$-equivariant $\varphi\in \Map(G\cdot \E K^{(n)}, G\cdot \E K)$, together with a neighborhood~$V\subseteq \TopHom(G,\RR)\setminus \{0\}$ of~$\chi$ such that $\varphi$~is $\psi$-value raising for every $\psi \in V$. 
	Lemma~\ref{lem:pushfilingup} then yields the required $D\in \C(G)_{\supseteq K}$ uniformly for all $\psi \in V$.
\end{proof}

\begin{cor}[Uniform essential triviality]\label{cor:unifesstriv}	
	Let $n\in\NN$ and let $\Theta$~be a compact subset of~$\TopS^n(G) \setminus \{0\}$. Then for every $C\in \C(G)$ there is $D\in \C(G)_{\supseteq C}$ such that for every $\chi \in \Theta$ and every~$k \le n-1$, the inclusion
	\[(G\cdot \E C)_\chi \into (G\cdot \E D)_\chi\]
	is $\pi_k$-trivial.
\end{cor}
\begin{proof}
	For each $\chi \in \Theta$, use Proposition~\ref{prop:stabletriviality} to find an open neighborhood~$V_\chi \subseteq \TopHom(G,\RR) \setminus \{0\}$ of~$\chi$ and $D_\chi \in \C(G)_{\supseteq C}$ such that for every $\psi \in V_\chi$ the inclusion
	$(G\cdot \E C)_\psi \into (G\cdot \E D_\chi)_\psi$
	is $\pi_k$-trivial.
	Now consider the open cover $\{V_\chi \mid \chi \in \Theta\}$ of~$\Theta$ and use compactness to produce a finite subset $F\subseteq \Theta$ such that $\Theta \subseteq \bigcup_{\chi \in F} V_\chi$. The compact set
	\[D := \bigcup_{\chi \in F} D_\chi\]
	is as required.
\end{proof}
\section{Compactness properties of co-abelian subgroups}\label{sec:coabelian}

This section presents a result that transfers compactness properties from~$G$ to a closed normal subgroup~$N$, provided the quotient~$Q$ is abelian. We retain the notation for the short exact sequence from before:
$$1 \to N \to G \xrightarrow{p} Q \to 1,$$
and we write $N^\perp$ to denote the set of characters $G\to \RR$ that vanish on~$N$; in other words,
\[N^\perp := p^*(\TopHom(Q, \RR)).\]
Our main result, which generalizes a classical theorem of Renz \cite[Satz~C]{Ren88}, is:

\begin{thm}[Co-abelian subgroups]\label{thm:coabelian}
	Suppose $Q$~is abelian and let $n\in\NN$. If $N^\perp \subseteq\TopS^n(G)$, then $N$~is of type~$\mathrm{C}_n$.
\end{thm}

Under the additional assumption that $Q$~is compactly generated, the converse statement also holds, because Corollary~\ref{cor:abelian} implies $\TopS^n(Q) = \TopHom(G,\RR)$, and then Theorem~\ref{thm:sigmaofquotient} (1) yields $N^\perp \subseteq \TopS^n(G)$.

\begin{rem}[Generalizing to $\chi \neq 0$]
	Theorem~\ref{thm:coabelian} is a statement about transferring property~$\mathrm C_n$ from~$G$ to~$N$ in case $Q$~is abelian and $N^\perp \subseteq\TopS^n(G)$. In other words, for the character $\chi =0$, we have $$\chi \in \TopS^n(G) \implies \chi|_{N}\in \TopS^n(N).$$ One might conjecture that this transfer property generalizes to all nonzero characters, but this is not true even in the discrete case.
	
	Consider, for example, the group $G:= \mathrm F_2 \times \ZZ$, where $\F_2$~is the free group on two generators, 
	and the short exact sequence
	\[1 \to \mathrm F_2 \to G \to  \ZZ \to 1\]
	given by the obvious inclusion and projection.
	Define also the character
	\begin{align*}
		\chi \colon \mathrm F_2 \times  \ZZ = \langle a,b \rangle \times \langle t \rangle & \to \RR\\
		a,b,t & \mapsto 1.
	\end{align*}
	A classical criterion for computing~$\Sigma^1$ for right-angled Artin groups \cites[Theorem~4.1]{MV95}[Theorem~1]{Mei95} reveals that $\chi \in \Sigma^1(\mathrm F_2 \times \ZZ)$
	and that $\mathrm F_2^\perp \subseteq \Sigma^1(\mathrm F_2 \times \ZZ)$. However, it is well known (and indeed follows from the same criterion) that $\Sigma^1(\mathrm F_2)=\{0\}$, so $\chi|_{\mathrm F_2} \not \in \Sigma^1(\mathrm F_2)$.
\end{rem}

The remainder of this section is devoted to proving Theorem~\ref{thm:coabelian}, which relies on embedding a cocompact subgroup of~$Q$ into a suitable $\RR^m$, and then exploiting its geometry. From now on, we equip~$\RR^m$ with the standard scalar product and induced norm.

For each point $P\in \RR^m\setminus\{0\}$, let $\chi_P\colon \RR^m \to \RR$ be the character given by
\[\chi_P(T) := \left\langle T, -\frac P{\Vert P\Vert}\right\rangle.\]
In particular, $\chi_P(P) = - \Vert P \Vert$.
Often, an embedding of topological groups $Q\subseteq \RR^m$ will be implicit, and then we use, for elements $g \in G$, the shorthand notation $\Vert g\Vert := \Vert p(g)\Vert$. If $g\not \in N$, so $\Vert g\Vert \neq 0$, we define the character
\begin{align*}
	\chi_g :=  \chi_{p(g)} \circ p \colon G&\to \RR\\
	h & \mapsto \left\langle p(h), -\frac{p(g)}{\Vert g\Vert}\right \rangle.
\end{align*}

Given a subset $X \subseteq G$, we will also write $\Vert X \Vert := \sup_{g\in X}\Vert g\Vert$ and similarly, given a map~$\mu$ from a simplicial set into~$\E G$, we put $\Vert \mu \Vert := \sup\{\Vert g \Vert \mid g\in \im(\mu)^{(0)}\}$.

The main geometrical input used in the proof of Theorem~\ref{thm:coabelian} is the following lemma in euclidean geometry, which, given $P\in \RR^m$ of large enough norm, converts a lower bound on the $\chi_P$-value of a point~$T$ close to~$P$ into an upper bound on~$\Vert T \Vert$.

\begin{lem}[Norm bound from $\chi_P$-bound]\label{lem:normfromchi}
	Let $P,T\in \RR^m$ be points with distance bounded by $\Vert P - T \Vert \le r$, and with $P \neq 0$. Let $h > 0$ be such that $\chi_P (T) \ge \chi_P(P) + h$. Then, given $\nu \in {[0,h[}$, if 
	$\Vert P\Vert \ge \max \left\{ \frac{r^2 - \nu^2}{2(h-\nu)}, h\right\}$, it follows that $\Vert T \Vert \le \Vert P\Vert - \nu$.
\end{lem}

\begin{figure}[h]
	\centering
	\def \svgwidth{0.6\linewidth}
	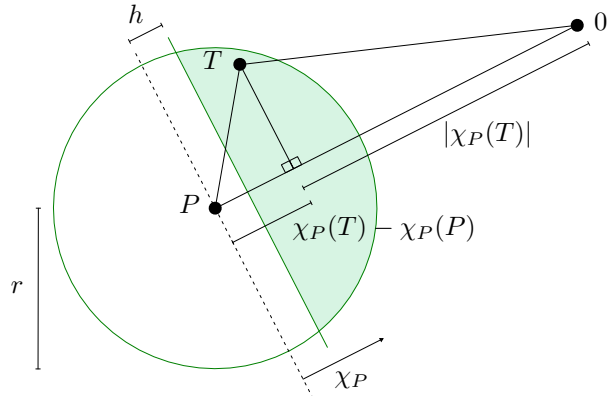
	\caption{Bounding $\Vert T\Vert$. The shaded region indicates where $T$~might lie, once $P, h$~and~$r$ have been fixed.}
	\label{fig:euclideangeometrylemma}
\end{figure}

\begin{proof}
	Since $h-\nu >0$, the first lower bound on~$\Vert P \Vert$ may be rewritten as
	\[ r^2 - 2h\Vert P \Vert \le \nu ^2 - 2 \nu \Vert P\Vert \tag{$\star$}.\]

	Applying Pythagoras's Theorem to the two right triangles indicated in Figure~\ref{fig:euclideangeometrylemma} shows that
	\begin{align*}
		\Vert T \Vert ^2 & = \chi_P(T)^2 + \Vert P-T\Vert ^2 - (\chi_P(T) - \chi_P(P))^2\\
		&\le r^2 + 2\chi_P(T)\chi_P(P) - \chi_P(P)^2  & (\Vert P - T\Vert \le r)\\
		&\le r^2 + 2(\chi_P(P)+h)\chi_P(P) - \chi_P(P)^2  & (\chi_P(T) \ge \chi_P(P) + h, \quad \chi_P(P) \le 0)\\
		& = r^2 - 2h \Vert P \Vert + \Vert P \Vert^2 & (\chi_P(P) = -\Vert P \Vert)\\
		&\le \nu^2 - 2 \nu \Vert P \Vert + \Vert P \Vert ^2 &(\star)  \\
		&= (\Vert P \Vert - \nu)^2.
	\end{align*}
	As $\Vert P \Vert - \nu \ge h - \nu > 0$, this implies $\Vert T \Vert \le \Vert P \Vert - \nu$.
\end{proof}

To state the main technical result of this section, it will be convenient to package some of its hypotheses and input data into the following setup:
	
\begin{setup}\label{setup:coabelian}
	\begin{itemize}
		\item Assume $Q = \ZZ^{m_1} \times \RR^{m_2} \subseteq \RR^m$, for some $m_1, m_2 \in \NN$ and $m := m_1 + m_2$.
		\item For each element $u_1,\ldots, u_m$ of the standard basis of~$\RR^m$, choose $p$-preimages $\tilde u_1, \ldots, \tilde u_m \in G$ and define
		\[B:=\{1, \tilde u_1, \tilde u_1^{-1}, \ldots , \tilde u_m, \tilde u_m^{-1}\}.\]

	\end{itemize}
\end{setup}
	
\begin{prop}[Radial connecting sets]\label{prop:radialalonso}
	Assume that $N^\perp\subseteq \TopS^n(G)$ and that we are in the situation of Setup~\ref{setup:coabelian}. Then there exists $K\in \C(G)_{\supseteq \{1\}}$ such that for every $C\in \C(G)$ there is an $N$-equivariant $\varphi \in \Map(G\cdot \E C^{(n)}, G\cdot \E K)$ such that:
	\begin{enumerate}
		
		\item There are $\mathcal R >0$ and $\epsilon >0$ such that for every simplex~$\sigma$ of $G\cdot \E C^{(n)}$, we have
		\[\Vert \varphi (\sigma)\Vert \le \max\{ \mathcal R, \Vert\sigma\Vert - \epsilon \}.\]
		
		\item For each vertex $g\in G$, we have $\varphi(g)\in gB$.
	\end{enumerate}
\end{prop}


The reader should note the similarities with Proposition~\ref{prop:alonso_main}. In that proposition, the $n$-connecting sets were used in producing a map raising $\chi$-values; here the analogous role is played by~$K$, which we use for producing a map~$\varphi$ that decreases the distance from~$N$ (albeit only for simplices beyond a fixed distance~$\mathcal R$). Condition~2 might seem like a new addition, but the analogous property in the setting of Proposition~\ref{prop:alonso_main} was automatically fulfilled by virtue of the $\chi$-value-raising map being $G$-equivariant. Our~$\varphi$, however, is far from being $G$-equivariant.
As the reader might suspect, Condition~2 will be needed later in order to apply Lemma~\ref{lem:interpol_homotopy}. Lastly, observe that (much like connecting sets) if $K\in\C(G)$~satisfies the conditions in this proposition, then so does every $K'\in \C(G)_{\supseteq K}$.

\begin{proof} \setcounter{claim}{0}
	In the degenerate case $m=0$, we have $N=G$ and $N^\perp = \{0\}$, so the hypothesis says that $G$~is of type $\mathrm C_n$. We take $K$~to be an $n$-connecting set for the zero character (using Lemma~\ref{prop:alonsoandsigma}). Then, given $C\in \C(G)$, the map $\varphi$ given by the case ``$\chi=0$'' of Proposition~\ref{prop:alonso_main}, together with the constants $\mathcal R = \epsilon = 1$ is as required. So now we prove the proposition under the assumption that $m\ge1$ (which is useful because we will soon need to divide by $\sqrt m$).
	
	We proceed by induction over~$n$. For $n=0$, we take
	$K=\{1\}$, and (independently of~$C$) define $\varphi$ as follows:
	If $\Vert g \Vert \le \sqrt{m}$, let $t_g:=1$. For $\Vert g \Vert > \sqrt{m}$, we claim it is possible to choose $t_g\in B$ such that
	\[\chi_g(t_g) \ge \frac 1{\sqrt m}.\]
	To find $t_g$, note that the element $\frac{p(g)}{\Vert g\Vert} \in \RR^{m}$, being of unit norm, has an entry (say, the $i$-th) with absolute value at least $\frac 1{\sqrt{m}}$. Depending on its sign, we may then take $t_g = \pm \tilde u_i$. Since the choice of~$t_g$ is only constrained by $p(g)$, the~$t_g$ may be picked $N$-equivariantly.
	We now define $\varphi(g) := gt_g$.
	
	Condition~2 is then certainly satisfied, and we now check that Condition~1 holds with $\mathcal R := \sqrt{m}$ and $\epsilon := \frac 1{2\sqrt m}$. This is immediate for $0$-simplices~$g$ with $\Vert g\Vert \le \mathcal R$. For the other vertices, we apply
	Lemma~\ref{lem:normfromchi},
	which tells us that
	$\Vert \varphi(g) \Vert \le \Vert g\Vert - \frac1{2\sqrt{m}}$ if
	\[\Vert g \Vert \ge \max \left\{\frac{1^2 -\left(\frac 1{2\sqrt m}\right)^2}{2\left(\frac{1}{\sqrt m} - \frac{1}{2\sqrt m}\right)}, \frac 1 {\sqrt m} \right\}.\]
	This maximum is easily seen to be bounded above by~$\sqrt m = \mathcal R$.
	
	Assume now that $n\ge 1$, let $K_0\in \C(G)_{\supseteq \{1\}}$ be as given by the ``$n-1$'' case of the statement, and
	consider the compact set of characters
	\[\Theta := \{ \chi_u \circ p \mid u \in \Sph^{m-1}\} \subseteq N^\perp \subseteq \TopS^n(G).\]
	By Corollary~\ref{cor:unifesstriv}, there is $K \in \C(G)_{\supseteq K_0}$ such that for every $\chi \in \Theta$ the inclusion $(G\cdot \E K_0)_\chi \into (G\cdot \E \mathring K)_\chi$~is $\pi_{n-1}$-trivial (we may take the interior~$\mathring K$ due to Lemma~\ref{lem.compactexhaustion}). Note that by Lemma~\ref{lem:CnfromSigma}, this implies in particular that $G\cdot \E K_0 \into G\cdot \E \mathring K$ is $\pi_{n-1}$-trivial.
	
	Now, let $C \in \C(G)$ and let $\varphi_0\in \Map(G\cdot \E C^{(n-1)}, G\cdot \E K_0)$ be as given by the inductive hypothesis (say, defined on $\SD^{m_0}(G\cdot \E C^{(n-1)})$), with corresponding constants $\epsilon_0, \mathcal R_0$. As in the proof of Proposition~\ref{prop:alonso_main}, we will find a suitable $m\in \NN$ and define~$\varphi$ on $\SD^{m_0+m}(G\cdot \E C^{(n)})$ as an extension of~$\varphi_0\circ \Phi^m$.
	Writing $\mathcal S := \SD^{m_0}(\partial \Delta^n)$, we consider, for each nondegenerate $n$-simplex~$\sigma$ of $G\cdot \E C$, the map $\eta_\sigma \colon \mathcal S \to G\cdot \E K_0$ specified by~$\varphi_0$. Our goal is to produce maps $\mu_\sigma \colon  \SD^{m_0 + m}(\Delta^n) \to G\cdot \E K$ extending the $\eta_\sigma \circ \Phi^m$, which will then assemble to the desired $\varphi$. We will also ensure these $\mu_\sigma$ are such that Conditions 1~and~2 hold.
	
	Denote the $G$-translate of each~$\eta_\sigma$ sending the $0$-th vertex~$e_0$ of~$\Delta^n$ to~$1$ by
	\[\eta_\sigma^\bullet := \eta_\sigma(e_0)^{-1} \cdot \eta_\sigma.\]
	These maps are collected into the set
	\[\mathcal M := \{\eta_\sigma^\bullet\in \Map(\partial \Delta^n, G\cdot \E K_0) \mid \text{$\sigma$~is a nondegenerate $n$-simplex of~$G\cdot \E C$}  \},\]
	and we define $\overline{\mathcal M}$ as its closure in~$\Map(\partial \Delta^n, G\cdot \E K_0)$.
	
	When proving Proposition~\ref{prop:alonso_main}, the set~$\overline {\mathcal M}$ was thought of as the collection of filling problems to be solved in order to construct the maps~$\mu_\sigma$. This time, 
	filling problems coming from simplices far away from~$N$ require additional data, which consists, roughly speaking, of a character whose value should drop by a controlled amount upon filling, and the vertex with respect to which this drop is to be measured. This is the motivation for defining the set
	\[ \mathcal K := \big\{(\eta, \chi, e, t)\in \overline{\mathcal M} \times \Theta \times (\Delta^n)^{(0)} \times B \, \mid \, \chi(\eta) \ge \chi(\eta(e)t^{-1}) +\epsilon_0\}.\]

	\begin{claim} \label{claim:bothcpt}
		$\overline{\mathcal M}$ and $\mathcal K$ are compact.
	\end{claim}
	
	\begin{proof}
		We first show $\overline {\mathcal M}$ is compact. By Lemma~\ref{lem:closedandopensubspaces}~(1), we know $\HomSS(\mathcal S, G\cdot \E K_0)$ is closed in $\HomSS(\mathcal S, \E G) = \prod_{\mathcal S^{(0)}}G$, so it suffices to find a compact $\mathcal B \subseteq \HomSS(\mathcal S, G\cdot \E K_0)$ containing~$\mathcal M$.
		
		If $n=1$, then each simplex~$\sigma$ is an edge labeled by an element of~$C^{-1}C$. Moreover, each element of~$\mathcal M$ sends~$e_0$ to~$1$, so given Condition~2 on~$\varphi_0$, we may take
		\[\mathcal B := \{1\} \times B^{-1} C^{-1}C B.\]
		
		For $n \ge 2$, we see $\mathcal S$~is connected, so for each vertex $v\in \mathcal S ^{(0)}$, let~$d_v\in \NN$ be its edge-distance to~$e_0$ in the $1$-skeleton $\mathcal S^{(1)}$. As each edge of~$\mathcal S$ is mapped by each $\eta_\sigma^\bullet \in \mathcal M$ to an edge labeled by an element in~$K_0^{-1}K_0$, we may take
		\[\mathcal B= \prod_{v\in \mathcal S^{(0)}}(K_0^{-1}K_0)^{d_v}.\]
		
		This finishes the proof that $\overline{\mathcal{M}}$~is compact, and hence so is $\overline{\mathcal M} \times \Theta \times (\Delta^n)^{(0)} \times B$.
		We establish compactness of~$\mathcal K$ by proving it is closed in this product. Define
		\begin{align*}
			F\colon \overline{\mathcal M} \times \Theta \times (\Delta^n)^{(0)} \times B & \to \RR\\
			(\eta, \chi, e, t)&\mapsto \min_{v\in \mathcal S^{(0)}} \chi(\eta (v)) - \chi(\eta(e)t^{-1}).
		\end{align*}
		Since $G$~is locally compact and~$\TopHom(G,\RR)$ is equipped with the compact-open topology, the evaluation map $G\times \TopHom (G,\RR) \to \RR$ is continuous \cite[§3.4, Corollary~1]{Bou98}, and so $F$~is continuous. Since $\mathcal K = F^{-1}({[\epsilon_0, +\infty[})$, we conclude $\mathcal K$~is closed in the above product, and hence compact.
	\end{proof}
	
		We will now associate to each $n$-simplex~$\sigma$ of~$G\cdot \E C$ the data that will encode the filling problem which we will use for constructing~$\mu_\sigma$, and we do it differently depending on whether $\sigma$ is close to~$N$. The threshold between these two regimes will be
	\[\mathcal R_0^+ := \mathcal R_0 +\Vert C^{-1} C\Vert +\epsilon_0.\]
	If $\Vert\sigma\Vert\le \mathcal R_0^+$, the data of the associated filling problem will simply be the map $\eta_\sigma^\bullet \in \overline{\mathcal M}$.
	For $\Vert\sigma\Vert > \mathcal R_0^+$, we choose an index $i_\sigma \in \{0,\ldots, n\}$ such that the	
	vertex $e_\sigma := e_{i_\sigma} \in (\Delta^n)^{(0)}$ realizes the distance to~$N$ (that is, with $\Vert \sigma[i_\sigma] \Vert = \Vert \sigma\Vert$), and we make this choice $N$-equivariantly. Then we define $t_\sigma := \sigma[i_\sigma]^{-1} \eta_\sigma(e_\sigma) \in B$ and $\chi_\sigma := \chi_{\sigma[i_\sigma]} \in \Theta$. The filling problem associated to~$\sigma$ is then the tuple~$(\eta_\sigma^\bullet, \chi_\sigma, e_\sigma, t_\sigma)$.
	
	\begin{claim}\label{claim:tupleinK}
		If $\Vert \sigma \Vert > \mathcal R_0^+$, then $(\eta_\sigma^\bullet, \chi_\sigma, e_\sigma, t_\sigma) \in \mathcal K$.
	\end{claim}
	
	\begin{proof}
		We need to verify that $\chi_\sigma(\eta_\sigma^\bullet) \ge \chi_\sigma(\eta^\bullet_\sigma(e_\sigma)t_\sigma^{-1}) + \epsilon_0$. Adding $\chi_\sigma(\eta_\sigma(e_0))$ to both sides, we see this is equivalent to
		\[\chi_\sigma(\eta_\sigma) \ge \chi_\sigma(\eta_\sigma(e_\sigma)t_\sigma^{-1}) + \epsilon_0,\]
		which we set out to show.
		
		Since the edges of~$\sigma$ have labels in~$C^{-1}C$, the lower bound $\Vert \sigma \Vert > \mathcal R_0^+$ implies that every vertex~$g$ of~$\sigma$ has $\Vert g \Vert > \mathcal R_0 + \epsilon_0$. In particular, each $(n-1)$-simplex~$\tau$ of~$\partial \sigma$ satisfies $\Vert \tau \Vert > \mathcal R_0 + \epsilon_0$, and so $\max\{\mathcal R_0, \Vert\tau\Vert -\epsilon_0\} = \Vert \tau \Vert - \epsilon_0$. Condition~1 on~$\varphi_0$ then implies that, for the map~$\mu_\tau \colon \SD^{m_0}(\Delta^{n-1}) \to G\cdot \E K_0$ describing $\varphi_0$ on~$\tau$, we have
		\[\Vert \mu_\tau \Vert \le \Vert \tau \Vert - \epsilon_0. \tag{$\star$}\]
		
		For each vertex~$v \in \mathcal S^{(0)}$, we now see
		{\allowdisplaybreaks
			\begin{align*}
				\chi_\sigma(\eta_\sigma(v)) &= \left\langle p(\eta_\sigma(v)), -\frac{p(\sigma[i_\sigma])}{\Vert \sigma[i_\sigma]\Vert} \right\rangle\\
				&\ge - \Vert \eta_\sigma(v) \Vert \tag{Cauchy-Schwarz}\\
				&\ge - \Vert \eta_\sigma \Vert\\
				& =- \max_{\text{{$\tau$~face of~$\sigma$}}} \Vert \mu_\tau \Vert\\
				& \ge  -\max_{\text{{$\tau$~face of~$\sigma$}}} \Vert \tau \Vert + \epsilon_0 \tag{$\star$}\\
				&= - \Vert \sigma \Vert + \epsilon_0\\
				&= - \Vert \sigma[i_\sigma]\Vert + \epsilon_0 \\
				&= \chi_\sigma(\sigma[i_\sigma]) + \epsilon_0\\
				&= \chi_\sigma(\eta_\sigma(e_\sigma) t_\sigma^{-1}) + \epsilon_0 \tag{$\eta_\sigma(e_\sigma)) = \sigma[i_\sigma] t_\sigma$},
		\end{align*}}
		from which we conclude the desired lower bound on $\chi_\sigma(\eta_\sigma)$.
	\end{proof}
	
	As with Proposition~\ref{prop:alonso_main}, the idea of re-using the solution to a filling problem on nearby problems will be essential, so we again make use of the following notation: given $m'\in\NN$, abbreviate
	\[\mathcal D_{m'}:= \SD^{m_0 + m'}(\Delta^n), \qquad \mathcal S_{m'} := \SD^{m_0 + m'}(\partial \Delta^n),\]	
	and given maps $\eta\colon \mathcal S \to \E G$ and $\mu\colon \mathcal D_{m'} \to \E G$, define $\eta^\mu \colon \mathcal D_{m'} \to \E G$ on vertices as
	\[\eta^\mu(v) := \begin{cases}
		\eta\circ \Phi^{m'}(v) & \text{if $v\in \mathcal S_{m'}^{(0)}$,}\\
		\mu(v) & \text{otherwise.}
	\end{cases}\]
	Recall also that the assignment $\eta \mapsto \eta^\mu$ is continuous.
	
	We first focus on finding solutions to the easier filling problems $\alpha \in \overline {\mathcal M}$.
	Since the inclusion $G\cdot \E K_0 \into G\cdot \E \mathring K$ is $\pi_{n-1}$-trivial, we know by Lemma~\ref{lem:combfilling} that there are $m_\alpha \in \NN$ and a map $\mu_\alpha \colon \mathcal D_{m_\alpha} \to G\cdot \E \mathring K$ extending $\alpha \circ \Phi^{m_\alpha} \colon \mathcal S_{m_\alpha}\to G\cdot \E K_0$.
	
	\begin{claim}\label{claim:copyeasyhomework}
		There is an open neighborhood~$U_\alpha \subseteq  \HomSS(\mathcal S, G\cdot \E K_0)$ of~$\alpha$ such that for every $\eta \in U_\alpha$ we have
		$\eta^{\mu_\alpha} \in \HomSS(\mathcal D_{m_\alpha}, G\cdot \E \mathring K)$.
	\end{claim}
	
	The proof of Claim~\ref{claim:copyeasyhomework} is an easier version of that of Claim~\ref{claim:copyhomework} in Proposition~\ref{prop:alonso_main}, and also of Claim~\ref{claim:copyhardhomework} below, so we omit it.

	We now focus on the filling problems $\beta = (\eta_\beta, \chi_\beta, e_\beta, t_\beta) \in\mathcal K$. For each such~$\beta$, choose $g\in G$ with $\chi_\beta (g) = \chi_\beta (\eta_\beta)$, so $\eta_\beta \in \HomSS(\mathcal S, g\cdot (G\cdot \E K_0)_{\chi_\beta})$.
	By choice of~$K$ (and Lemma~\ref{lem:combfilling}), there is $m_\beta\in \NN$ and a map
	\[\mu_\beta \colon \mathcal D_{m_\beta} \to g\cdot (G\cdot \E\mathring K)_{\chi_\beta},\]
	extending $\eta_\beta \circ \Phi^{m_\beta}$. In particular,
	\[ \chi_\beta(\mu_\beta) = \chi_\beta(\eta_\beta)\\
	\ge \chi_\beta(\eta_\beta(e_\beta)t_\beta^{-1}) + \epsilon_0.\]

	\begin{claim}\label{claim:copyhardhomework}
		There is an open neighborhood $U_\beta \subseteq \HomSS(\mathcal S, G\cdot \E K_0) \times \Theta$  of $(\eta_\beta, \chi_\beta)$ such that for every $(\eta, \chi) \in U_\beta$, we have
		$\eta^{\mu_\beta} \in \HomSS(\mathcal D_{m_\beta}, G\cdot \E \mathring K)$ and
		\[\chi(\eta^{\mu_\beta} ) > \chi(\eta(e_\beta)t_\beta^{-1}) + \tfrac{\epsilon_0}2.\]
	\end{claim}

	\begin{proof}
		Since, as already mentioned, the evaluation map $G\times \TopHom(G,\RR) \to \RR$ is continuous, also the map 
		\begin{align*}
			\HomSS(\mathcal D_{m_\beta}, \E G) \times \Theta &\to \RR\\
			 (\mu , \chi) &\mapsto \min_{v\in \mathcal D_{m_\beta}^{(0)}}\chi(\mu(v)) - \chi(\mu(e_\beta)t_\beta^{-1})
		\end{align*}
		is continuous, so the set of pairs $(\mu, \chi)$ for which $\chi (\mu) > \chi(\mu(e_\beta)t_\beta^{-1}) +\tfrac{\epsilon_0}2$ is open. By Lemma~\ref{lem:closedandopensubspaces}(2), the set~$U$ of pairs where, in addition, $\mu$~has image in $G\cdot \E \mathring K$ is also open. We take~$U_\beta$ to be the preimage of~$U$ under the continuous map
		\begin{align*}
			\HomSS(\mathcal S, G\cdot \E K_0) \times \Theta &\to \HomSS(\mathcal D_{m_\beta}, \E G) \times \Theta\\
			(\eta, \chi) &\mapsto (\eta^{\mu_\beta}, \chi).
		\end{align*}
		By the observation preceding the claim, we have $(\eta_\beta, \chi_\beta)\in U_\beta$.
	\end{proof}
	
	Now, the sets $U_\alpha$ provided by Claim~\ref{claim:copyeasyhomework} over all $\alpha \in \overline{\mathcal M}$ form an open cover of~$\overline{\mathcal M}$. Since, by Claim~\ref{claim:bothcpt},  $\overline{\mathcal M}$ is compact, we may choose a finite subset $F_1 \subseteq \overline{\mathcal M}$ with
	$\overline{\mathcal M} \subseteq \bigcup_{\alpha \in F_1}  U_\alpha$.
	Similarly, the open subsets
	\[W_\beta := U_\beta  \times \{e_\beta\} \times \{t_\beta\} \subseteq \HomSS(\mathcal S, G\cdot \E K_0) \times \Theta \times (\Delta^n)^{(0)} \times B,\]
	taken over all $\beta \in \mathcal K$, cover~$\mathcal K$, which is compact by Claim~\ref{claim:bothcpt}. Hence, there is a finite subset $F_2\subseteq \mathcal K$ with
	$\mathcal K \subseteq \bigcup_{\beta \in F_2} W_\beta$.
	We take $m$ as the maximum of the finite set $$\{m_\alpha \mid \alpha \in F_1\} \cup \{m_\beta \mid \beta\in F_2\}.$$
	
	We are now ready to define $\mu_\sigma \colon \mathcal D_m \to G\cdot \E K$ for all nondegenerate $n$-simplices~$\sigma$ of $G\cdot \E C$. First we choose, for each $\eta \in \mathcal M$, an element $\alpha\in F_1$ with $\eta\in U_\alpha$. Then, for every $\sigma$ with $\Vert \sigma\Vert\le \mathcal R_0^+$ such that  $\eta_\sigma^\bullet = \eta$, we define $\mu_\sigma$ as the filling of~$\eta_\sigma^\bullet$ given by~$\mu_\alpha$, suitably translated:
	$$\mu_\sigma := \eta_\sigma(e_0) \cdot ((\eta_\sigma^\bullet)^{\mu_\alpha}\circ \Phi^{m - m_\alpha}).$$
	Similarly, for each tuple $(\eta,\chi, e,t) \in \mathcal K$, choose $\beta \in F_2$ with $W_\beta$ containing this tuple, and whenever $\Vert \sigma \Vert > \mathcal R_0^+$ satisfies $(\eta_\sigma^\bullet, \chi_\sigma, e_\sigma, t_\sigma) = (\eta, \chi, e,t)$, put
	$$\mu_\sigma := \eta_\sigma(e_0) \cdot ((\eta_\sigma^\bullet)^{\mu_\beta}\circ \Phi^{m - m_\beta}).$$
	By Claim~\ref{claim:tupleinK}, this indeed defines $\mu_\sigma$ for all $\sigma$. In this second case, it will be useful to note that since $(\eta_\sigma^\bullet, \chi_\sigma, e_\sigma, t_\sigma)\in W_\beta$, we have
	\[\chi_\sigma ((\eta_\sigma^\bullet)^{\mu_\beta}) > \chi_\sigma (\eta_\sigma^\bullet (e_\sigma)t_\sigma^{-1}) + \tfrac{\epsilon_0}2,\]
	so adding $\chi_\sigma(\eta_\sigma(e_0))$ to both sides and rewriting $\eta_\sigma(e_\sigma)t_\sigma^{-1} = \sigma[i_\sigma]$, we see
	\[\chi_\sigma (\mu_\sigma) > \chi_\sigma (\sigma[i_\sigma]) + \tfrac {\epsilon_0}2. \tag{$\star$}\]

	This clearly results in an $N$-equivariant definition, and the $\mu_\sigma$ assemble to a map~$\varphi$ that extends $\varphi_0\circ \Phi^{m-m_0}$. Condition~2 of the statement is inherited from~$\varphi_0$, so we are left to produce~$\mathcal R$ and~$\epsilon$ as required by Condition~1.	
	
	Let $d\in \NN$ be an upper bound on the edge-distance between vertices of the $1$-skeleton~$\mathcal D_m^{(1)}$.
	The maps~$\mu_\sigma$ send each edge of~$\mathcal D_m$ to an edge with label in $K^{-1}K$, and moreover, $\mu_\sigma(e_\sigma) \in \sigma[i_\sigma]\,  B$. Hence, for every $\mu_\sigma$ and every $v \in \mathcal D_m^{(0)}$, we have $\mu_\sigma(v)\in \sigma[i_\sigma]\, B(K^{-1}K)^d$, and thus
	\[ \Vert p(\mu_\sigma (v)) -  p(\sigma[i_\sigma]) \Vert \le \Vert B \Vert + d \,\Vert K^{-1}K \Vert =:r \tag{$\dagger$}.\]
	
	We now use $(\star)$ and $(\dagger)$ to apply Lemma~\ref{lem:normfromchi} with
	\[P := p(\sigma[i_\sigma]),\qquad  T := p(\mu_\sigma(v)), \qquad h := \tfrac{\epsilon_0}2, \qquad \nu := \tfrac{\epsilon_0}3.\]
	The conclusion of Lemma~\ref{lem:normfromchi} follows if $\Vert \sigma\Vert > \mathcal R^+_0$ (so $(\star)$ holds) and
	$\Vert \sigma \Vert =\Vert p(\sigma[i_\sigma]) \Vert \ge \max\{\frac{r^2 - \nu^2}{2(h-\nu)}, h \}$
	(as per the hypotheses of the lemma), yielding
	\[ \Vert \mu_\sigma(v)\Vert \le \Vert \sigma \Vert - \tfrac{\epsilon_0}3.\]
	Setting $M:=\max\{ \mathcal R_0^+, \frac{r^2 - \nu^2}{2(h-\nu)}, h\}$ we have, more succinctly,
	\[\Vert \sigma\Vert > M \quad \implies \quad \Vert \mu_\sigma \Vert \le \Vert \sigma \Vert - \tfrac{\epsilon_0}3.\]
	
	We now put 
	\[\mathcal R := M + r, \qquad \epsilon := \tfrac{\epsilon_0}3.\]
	By the preceding considerations, Condition~1 follows immediately for simplices~$\sigma$ with $\Vert \sigma\Vert > M$. To see that it also holds for $\Vert \sigma \Vert \le M$, note that $(\dagger)$ implies, for every~$\sigma$,
	\[\Vert \mu_\sigma \Vert \le \Vert \sigma \Vert + r.\]
	Thus, for $\Vert \sigma \Vert \le M$, we have $\Vert \mu_\sigma \Vert \le M+r = \mathcal R$, whence Condition 1 holds.
\end{proof}

\begin{proof}[Proof of Theorem~\ref{thm:coabelian}]
	The hypothesis implies, in particular, that $0\in \TopS^n(G) \subseteq \TopS^1(G)$, so $G$~is compactly generated, and thus so is~$Q$. By the classification theorem for compactly generated locally compact abelian groups \cite[Theorem~24]{Mor77}, one can express~$Q$ as
	\[Q \cong \ZZ^{m_1} \times \RR^{m_2} \times Q_0\]
	for some $m_1, m_2 \in \NN$ and $Q_0$~a compact abelian group. The inclusion $\ZZ^{m_1} \times \RR^{m_2} \into Q$ induces an isomorphism
	$\TopHom(Q, \RR) \cong \TopHom(\ZZ^{m_1} \times \RR^{m_2}, \RR)$, and as $Q_0$~is compact, and hence of type~$\mathrm C_n$, Theorem~\ref{thm:sigmaofquotient} says this isomorphism preserves $\TopS^n$-sets. Thus, we may assume from now on that~$Q = \ZZ^{m_1} \times \RR^{m_2}$, and  we import the notation of Setup~\ref{setup:coabelian}.
	
	For the rest of this proof, we will deviate from the usual notation by writing, for $R\ge 0$ and $X$~a subcomplex of~$\E G$,
	\[X_R := X \cap \E(p^{-1}(R \DD^{m})),\] 
	where $R\DD^m \subseteq \RR^m$~is the closed disk of radius~$R$ centered at~$0$. In other words, $X_R$~designates the subcomplex of~$X$ spanned by the vertices within distance~$R$ of~$N$, as measured in~$Q$. This notation should cause no confusion, as no particular character will be used in what follows.
	
	By the case $\chi=0$ of Proposition~\ref{prop:properaction}, we can prove that $N$~is of type~$\mathrm C_n$ by showing that the filtration $(N\cdot \E C)_{C\in \C(G)}$ of~$\E G$ is essentially $(n-1)$-connected. So let $C \in \C(G)$; we will produce a set $D^+\in \C(G)_{\supseteq C}$ such that the inclusion $N\cdot \E C \into N \cdot \E D^+$ is $\pi_k$-trivial for all $k\le n-1$.

	Let~$K$ be as given by Proposition~\ref{prop:radialalonso}. Thus, there are $m\in \NN$ and $\varphi \colon \SD^m(G\cdot \E K^{(n)}) \to G\cdot \E K$, together with constants~$\mathcal R$ and~$\epsilon$ such that for every $R\ge 0$, the map~$\varphi$ sends $\SD^m((G\cdot \E K)_R^{(n)})$ to $(G\cdot \E K)_{R'}$,
	where $R':=\max\{\mathcal R, R - \epsilon\}$. Since $K$~may be assumed to be arbitrarily large and $G$~is of type~$\mathrm C_n$, we may further assume that the inclusion $G\cdot \E C \into G\cdot \E K$ is $\pi_k$-trivial for $k\le n-1$.
	We may also assume, for later convenience, that $\mathcal R \ge \Vert C \Vert$.
	
	We now apply Lemma~\ref{lem:interpol_homotopy} to find $D\in\C(G)_{\supseteq K}$ such that the simplicial homotopy~$H$ from $\Phi^m \colon \SD^m(G\cdot \E K^{(n)}) \to G\cdot \E K$ to~$\varphi$ has image in $G\cdot \E D$.
	Note that for every simplex $\sigma$ of $(G\cdot \E K)^{(n)}$, the vertices in the $H$-image of $\SD^m(\sigma) \times \Delta^1$ are precisely the ones in $\sigma$ and in $\varphi(\SD^m(\sigma))$. Therefore, $H$~is has the following ``norm-nonincreasing'' property: for every $R\ge \RR$,
	$$H\bigl( \SD^m((G\cdot \E K)_R^{(n)}) \times \Delta^1 \bigr) \subseteq  (G\cdot \E D)_R.$$
	
	Finally, using Lemma~\ref{lem:cptlifts}, let $L \in \C(G)$ be such that $p(L) =  Q\cap \mathcal R \DD^m$. We claim that the set $D^+ := L D^{-1} D$ is as desired.
	
	To see this, observe first that $(G\cdot \E D)_{\mathcal R} \subseteq N\cdot \E D^+$.
	Indeed, given a simplex $\sigma = g\cdot (x_0, \ldots, x_r) \in (G\cdot \E D)_{\mathcal R}$, where $g\in G$ and $x_0, \ldots, x_r \in D$, one can write $gx_0=hl$ with $h\in N$ and $l \in L$, and so $\sigma = h\cdot (l, l x_0^{-1} x_1, \ldots, l x_0^{-1} x_r) \in N\cdot \E D^+$.
	The inclusion $N\cdot \E C \into N \cdot \E D^+$, which we wish to show is $\pi_k$-trivial, thus factors as
	\[N\cdot \E C \into (G\cdot \E C)_{\mathcal R} \into (G\cdot \E D)_{\mathcal R} \into N\cdot \E D^+,\]
	the first inclusion owing to our insistence that $\mathcal R \ge \Vert C\Vert$. Hence, it suffices we prove the inclusion $(G\cdot \E C)_{\mathcal R} \into (G\cdot \E D)_{\mathcal R}$ is $\pi_k$-trivial. From here on, the proof is similar to that of Lemma~\ref{lem:pushfilingup}, with the role of the $\chi$-value increase being replaced by a decrease in distance from~$N$.
	
	Let $\eta  \in  \Map(\partial \Delta^{k+1}, (G\cdot \E C)_{\mathcal R})$. Since $G\cdot \E C \into G\cdot \E K$ is $\pi_k$-trivial, Lemma~\ref{lem:combfilling} produces a filling $\mu \in \Map(\Delta^{k+1},G\cdot \E K)$. 	If $\Vert \mu\Vert =: R \le \mathcal R$, then $\eta$~is trivial in $(G\cdot \E K)_{\mathcal R} \subseteq(G\cdot \E D)_{\mathcal R}$, as we wish to show. Otherwise,
	the composition
	\[\partial \Delta^{k+1}  \times \Delta^1 \xrightarrow{\eta \times \mathrm{id}} (G\cdot \E C)^{(n)}_{\mathcal R} \times \Delta^1 \xrightarrow{H} (G\cdot \E D)_{\mathcal R}\]
	(where occurrences of $\SD$ and $\Phi$ are suppressed) 	is a homotopy from~$\eta$ to a map~$\eta'\in \Map(\partial \Delta^{k+1}, (G\cdot \E K)_{\mathcal R})$, which is filled by
	\[\mu'\colon \Delta^{k+1} \xrightarrow{\mu} (G\cdot \E K)^{(n)}_R \xrightarrow{\varphi} (G\cdot \E K)_{R'}.\]
	In particular, $\Vert \mu' \Vert \le R - \epsilon$ or $\Vert \mu'\Vert \le \mathcal R$.

	If still $\Vert \mu'\Vert > \mathcal R$, then we may again use~$H$ to homotope $\eta'$ to $\eta''$, which has a filling~$\mu''$ with $\Vert \mu ''\Vert \le R - 2\epsilon$ or $\Vert \mu''\Vert \le \mathcal R$. After finitely many applications of this procedure, we conclude $\eta$~is homotopic in $(G\cdot \E D)_{\mathcal R}$ to a map that has a filling in $(G\cdot \E K)_{\mathcal R}$.
\end{proof}

\printbibliography

\textsc{Kai-Uwe Bux}, Fakultät für Mathematik, Universität Bielefeld, Postfach 100131, Universitätsstraße 25, D-33501 Bielefeld, Germany\\
\textit{E-mail: }
\texttt{\href{mailto:bux@math.uni-bielefeld.de}{bux@math.uni-bielefeld.de}}\medskip

\textsc{Elisa Hartmann}, Fakultät für Mathematik, Universität Bielefeld, Postfach 100131, Universitätsstraße 25, D-33501 Bielefeld, Germany\\
\textit{E-mail: }
\texttt{\href{mailto:ehartmann@math.uni-bielefeld.de}{ehartmann@math.uni-bielefeld.de}}\medskip

\textsc{José Pedro Quintanilha}, Institut für Mathematik IMa, Im Neuenheimer Feld 205, 69120 Heidelberg, Germany\\
\textit{E-mail: }
\texttt{\href{mailto:jquintanilha@mathi.uni-heidelbeg.de}{jquintanilha@mathi.uni-heidelbeg.de}}

\end{document}